\documentclass[10pt]{article}
\usepackage{latexsym}
\usepackage{indentfirst}
\usepackage{amssymb,amsfonts,amsmath,amsthm}
\usepackage{graphicx}
\usepackage{mathrsfs}
\usepackage{array}
\usepackage{color}
\usepackage{epstopdf}
\usepackage[all]{xy}
\usepackage{hyperref}
\usepackage{url}
\usepackage[french,english]{babel}
\usepackage{geometry}
\newtheorem{thm}{Theorem}[section]
\newtheorem{prop}[thm]{Proposition}
\newtheorem{lem}[thm]{Lemma}

\newtheorem{exam}[thm]{Example}

\newtheorem{rmk}[thm]{Remark}
\newtheorem{dfn}[thm]{Definition}
\newtheorem{cons}[thm]{Construction}
\newtheorem{cond}[thm]{Conditions}
\numberwithin{equation}{section}

\newcommand{\frakg}{{\mathfrak g}}

\newcommand{\frakm}{{\mathfrak m}}

\newcommand{\frakA}{{\mathfrak A}}

\newcommand{\frakF}{{\mathfrak F}}

\newcommand{\frakM}{{\mathfrak M}}

\newcommand{\frakS}{{\mathfrak S}}

\newcommand{\frakX}{{\mathfrak X}}

\newcommand{\bA}{{\mathbb A}}

\newcommand{\bC}{{\mathbb C}}

\newcommand{\bE}{{\mathbb E}}
\newcommand{\bF}{{\mathbb F}}

\newcommand{\bL}{{\mathbb L}}
\newcommand{\bM}{{\mathbb M}}
\newcommand{\bN}{{\mathbb N}}

\newcommand{\bQ}{{\mathbb Q}}

\newcommand{\bZ}{{\mathbb Z}}

\newcommand{\calA}{{\mathcal A}}
\newcommand{\calB}{{\mathcal B}}
\newcommand{\calC}{{\mathcal C}}

\newcommand{\calE}{{\mathcal E}}
\newcommand{\calF}{{\mathcal F}}
\newcommand{\calG}{{\mathcal G}}

\newcommand{\calI}{{\mathcal I}}

\newcommand{\calL}{{\mathcal L}}

\newcommand{\calO}{{\mathcal O}}

\newcommand{\calR}{{\mathcal R}}

\newcommand{\calT}{{\mathcal T}}
\newcommand{\calU}{{\mathcal U}}

\newcommand{\calX}{{\mathcal X}}
\newcommand{\calY}{{\mathcal Y}}

\newcommand{\Ani}{{\mathrm{Ani}}}

\newcommand{\Bun}{{\mathrm{Bun}}}           
\newcommand{\colim}{{\mathrm{colim}}}       
\newcommand{\Gal}{{\mathrm{Gal}}}           
\newcommand{\Hom}{{\mathrm{Hom}}}           
\newcommand{\Ima}{{\mathrm{Im}}}            
\newcommand{\Isom}{{\mathrm{Isom}}}         
\newcommand{\Ker}{{\mathrm{Ker}}}           
\newcommand{\Lan}{{\mathrm{Lan}}} 
\newcommand{\Map}{\mathrm{Map}}
\newcommand{\Mod}{{\mathrm{Mod}}}           
\newcommand{\nil}{\mathrm{nil}}

\newcommand{\Rep}{{\mathrm{Rep}}}           
\newcommand{\Ring}{\mathrm{Ring}}

\newcommand{\Spf}{{\mathrm{Spf}}}           
\newcommand{\Spec}{{\mathrm{Spec}}}         
\newcommand{\Sym}{{\mathrm{Sym}}}           
\newcommand{\Tor}{{\mathrm{{Tor}}}}         
\newcommand{\Vect}{{\mathrm{Vect}}}         

\newcommand{\PreStk}{{\mathrm{PreStk}}}

\newcommand{\GL}{{\mathrm{GL}}}             

    \newcommand{\Aut}{\mathrm{Aut}}         
\newcommand{\cl}{{\mathrm{\rm cl}}}             
\newcommand{\cyc}{{\mathrm{cyc}}}           
\newcommand{\dR}{{\mathrm{dR}}}             
\newcommand{\gen}{{\mathrm{gen}}} 

\newcommand{\Herr}{{\mathrm{Herr}}}    
\newcommand{\Nilp}{\mathrm{Nilp}}

\newcommand{\perf}{\mathrm{perf}}           
\newcommand{\univ}{\mathrm{univ}}

\newcommand{\bfA}{{\textbf{A}}}
\newcommand{\bfC}{{\textbf{C}}}
\newcommand{\RHom}{{\mathrm{RHom}}}
\newcommand\rightthreearrow{\substack{\rightarrow\\[-1em] \rightarrow \\[-1em] \rightarrow}
}

\usepackage{relsize}
\usepackage[bbgreekl]{mathbbol}
\DeclareSymbolFontAlphabet{\mathbb}{AMSb} 
\DeclareSymbolFontAlphabet{\mathbbl}{bbold}
\newcommand{\Prism}{{\mathlarger{\mathbbl{\Delta}}}} 

\newcommand{\Fun}{\mathrm{Fun}}

\newcommand{\Ev}{\mathrm{Ev}}
\newcommand{\Twist}{\mathrm{Twist}}
\newcommand{\disc}{\mathrm{disc}}
\newcommand{\surj}{\mathrm{surj}}
\newcommand{\Res}{\mathrm{Res}}

\newcommand{\CAlg}{\mathrm{CAlg}}
\newcommand{\cn}{\mathrm{cn}}
\newcommand{\bfAni}{\textbf{Ani}}

\usepackage{titletoc}
\titlecontents{section}[1.72em]{\vspace{1pt}}%
    {\thecontentslabel.\enspace}
    {}
    {\titlerule*[0.5pc]{.}\contentspage}%

\titlecontents{subsection}[1.72em]{\vspace{1pt}}%
    {\thecontentslabel.\enspace}
    {}
    {\titlerule*[0.5pc]{.}\contentspage}%

\begin{document}
\title{Classicality of derived Emerton–Gee stack II: generalised reductive groups}
\author{Yu Min}
\date{}
\maketitle

\begin{abstract}
We use the Tannakian formalism to define the Emerton--Gee stack for general groups. For a flat algebraic group $G$ over $\bZ_p$, we are able to prove the associated Emerton--Gee stack is a formal algebraic stack locally of finite presentation over $\Spf(\bZ_p)$. We also define a derived stack of Laurent $F$-crystals with $G$-structure on the absolute prismatic site, whose underlying classical stack is proved to be equivalent to the Emerton--Gee stack.

In the case of connected reductive groups, we show that the derived stack of Laurent $F$-crystals with $G$-structure is classical in the sense that
when restricted to truncated animated rings, it is the \'etale sheafification of the left Kan extension of the Emerton--Gee stack along the inclusion from classical commutative rings to animated rings. Moreover, when $G$ is a generalised reductive group, the classicality result still holds for a modified version of the Emerton--Gee stack. In particular, this completes the picture that the derived stack of local Langlands parameters for the Langlands dual group of a reductive group is classical.
\end{abstract}

\tableofcontents

\section{Introduction}
Let $K$ be a finite extension of $\bQ_p$ with the ring of integers $\calO_K$. In \cite{Min23}, we constructed a derived stack of Laurent $F$-crystals of rank $d$ on the absolute prismatic site $(\calO_K)_{\Prism}$, whose underlying classical stack is equivalent to the Emerton--Gee stack for $\GL_d$. Moreover, we showed this derived stack is determined by the Emerton--Gee stack, in the sense that it is essentially the left Kan extension of the Emerton--Gee stack along the inclusion from classical commutative rings to animated rings. In other words, this left Kan extension, which is defined in an abstract way, admits a nice moduli interpretation in terms of Laurent $F$-crystals.

In this paper, we will deal with more general groups $G$. The motivation is two-fold. On the one hand, as in \cite{EG22}, the Emerton--Gee stack plays an important role in proving the existence of crystalline lifts of residual Galois representations valued in $\GL_d$. It also provides a natural geometric setting for the weight part of Serre conjecture and the Breuil--M\'ezard conjecture. So in order to study these questions for more general groups $G$, it is desirable to have a $G$-version of the Emerton--Gee stack. On the other hand, the Emerton--Gee stack is regarded as the natural geometric object on the spectral side of a conjectural categorical $p$-adic local Langlands correspondence in \cite{EGH22}. Its counterpart in the $l\neq p$ case, i.e. the (derived) stack of $L$-parameters, is shown to be classical for reductive groups (cf. \cite{FS21},\cite{DHKM20},\cite{Zhu20}). So it is natural to expect the existence of a derived Emerton--Gee stack for general groups and show it is indeed classical in the case of reductive groups. This will also enable us to study the parabolic induction of the (ind-)coherent sheaves on the spectral side.

\subsection{The Emerton--Gee stack for general groups}
We first explain how to construct a stack of \'etale $(\varphi,\Gamma)$-modules for a general group $G$. The natural method of dealing with general groups is using the Tannakian formalism. Let $G$ be a flat affine group scheme of finite type over $\Spec(\bZ_p)$ and $\Rep(G)$ be the category of algebraic representations on finite free $\bZ_p$-modules. We can give a direct definition as follows.

\begin{dfn}
    Let $\calX_G:\Nilp_{\bZ_p}\to {\rm Groupoids}$ be the functor sending each $R\in \Nilp_{\bZ_p}$ to the groupoid of exact $\otimes$-functors $\Fun^{\rm ex,\otimes}(\Rep(G),\Mod^{\varphi,\Gamma}(\bfA_R))$, where $\Nilp_{\bZ_p}$ is the category of $p$-nilpotent rings, ${\rm Groupoids}$ is the $2$-category of groupoids, and $\Mod^{\varphi,\Gamma}(\bfA_R)$ is the category of \'etale $(\varphi,\Gamma)$-modules with $R$-coefficient of all ranks. 
\end{dfn}

The main issue is to show this prestack is indeed a well-behaved stack. We can prove the following theorem.
\begin{thm}[Theorem \ref{main-1}]\label{main-1-intro}
    Let $G$ be a flat affine group scheme of finite type over $\Spec(\bZ_p)$. The prestack $\calX_G$ is a formal algebraic stack locally of finite presentation over $\Spf(\bZ_p)$.
\end{thm}

Before we explain our strategy of studying $\calX_G$, we first recall the existing works for some general groups. When $G$ is the Langlands dual group over $\bZ_p$ of a tame group over $K$, the corresponding Emerton--Gee stack was defined in the same way (more precisely, it is the same as our Definition \ref{intro-dfn}) and studied in \cite{Lin23}. And for the group ${\rm GSp}_4$, this was studied in \cite{Lee23} using the concrete description of ${\rm GSp}_4$. The strategy in \cite{Lin23} is roughly the same as \cite{EG22}, i.e. first deal with an affine Grassmannian, which has reasonable geometric properties, then study the stack of \'etale $\varphi$-modules using the results in the previous step, and finally consider the additional $\Gamma$-action to study the \'etale $(\varphi,\Gamma)$-modules. But as observed in \cite{Lin23}, the Tannakian formalism is not always compatible with the important step in \cite{EG22} reducing the case of ramified $K$ to that of unramified $K$.

In order to deal with more general groups, we take a different perspective. In fact, we aim to extract as much information as possible from the $\GL_d$ case. To motivate this, let us recall a more classical context. Suppose $G$ is a linear algebraic group over $\bC$ and $X$ is a smooth projective curve over $\bC$. The classical way to prove $\Bun_G$, the stack of $G$-torsors on $X$, is an algebraic stack locally of finite type is as follows. One first fixes an embedding $G\hookrightarrow \GL_n$. For $\Bun_{\GL_n}$, one can construct explicit smooth atlas. Then one can try to prove the morphism $\Bun_G\to \Bun_{\GL_n}$ is schematic and locally of finite presentation. We want to emphasize that the projectivity of $X$ plays an important role in the second step. 

Now taking $K=\bQ_p$ for simplicity, one could regard $\bfA_{\bZ_p}=\bZ_p((T))^\wedge_p$ as a ``curve" and view $\bfA_R=R((T))=\bfA_{\bZ_p}``\otimes"_{\bZ_p}R$ as the base change of the curve along $\bZ_p\to R$ for a $p$-nilpotent ring $R$. There is no projectivity involved but the additional $(\varphi,\Gamma)$-action can provide some finiteness. So we could hope to prove the morphism $\calX_G\to \calX_{\GL_n}$ is schematic. But it seems still hard to show this.

\begin{rmk}
    From the above perspective, the Emerton--Gee stack shares a similarity with $\Bun_G$ associated to the Fargues--Fontaine curve. This similarity will be discussed elsewhere.
\end{rmk}

We can go a bit further. In \cite{Bro13}, Broshi proved the stack of $G$-torsors over a flat proper curve $X$ over any field $k$ is an algebraic stack locally of finite type using a new perspective on $G$-torsors. In fact, for a flat affine group scheme $G$ of finite type over $\bZ_p$, Broshi proved that one can find an algebraic representation $V$ of $G$, a tensorial construction $t(V)$ and a line bundle $L\subset V$ which splits, such that $G\simeq \underline\Aut(V,L)$\footnote{For those groups which have reductive generic fibers, this result is proved in \cite[Proposition 1.3.2]{kisin2010integral}.}, where $\underline\Aut(V,L)$ sends a $\bZ_p$-algebra $R$ to the set of automorphisms $f:V_R\to V_R$ such that $t(f):t(V)_R\to t(V)_R$ fixes $L_R$. Using this description, Broshi further proved that a $G$-torsor over a scheme $Y$ over $\bZ_p$ is the same as a so-called $Y$-twist $(\calE_Y,\calL_Y)$, i.e. a pair of vector bundles on $Y$ such that there exists a fppf cover $\tilde Y\to Y$ such that $(\calE_{\tilde Y},\calL_{\tilde Y})$ is isomorphic to $(V_{\tilde Y},L_{\tilde Y})$. In particular, there is a morphism $\Bun_G\to \Bun_{\GL_d}\times \Bun_{\GL_{t(d)-1}}$ where $d=\dim_{\bZ_p}V$ and $t(d)=\dim_{\bZ_p}t(V)$ (we secretly changed the locally split line bundle to the corresponding quotient bundle). So now one can try to prove $\Bun_G\to \Bun_{\GL_d}\times \Bun_{\GL_{t(d)-1}}$ is schematic. This is the route we will follow in the context of \'etale $(\varphi,\Gamma)$-modules.

More precisely, let $G$ be a flat affine group scheme of finite type over $\bZ_p$. Fix a pair $(V,L)$ such that $G\simeq \underline\Aut(V,L)$. We will show \'etale $(\varphi,\Gamma)$-modules with $G$-structure are the same as \'etale $(\varphi,\Gamma)$-twists (cf. Lemma \ref{Lemma-Gamma}). Due to some technical issue concerning topology, this is first proved for finite type $\bZ_p$-algebras in Subsection \ref{subsec2.1} and then for general algebras in Subsection \ref{subsec2.2} after we have shown the stack of \'etale $\varphi$-modules with $G$-structure is limit preserving. Then we can get a natural morphism $\calX_G\to \calX_{\GL_d}\times \calX_{\GL_{t(d)-1}}$. We can further factor this morphism through another stack $\calX_G^{\circ}$ parametrising tuples $(\calE,\calF,f:t(\calE)\twoheadrightarrow \calF)$, which consist of a vector bundle $\calE$ of rank $d$, a vector bundle $\calF$ of rank $t(d)-1$ as well as a surjective morphism $f:t(\calE)\twoheadrightarrow \calF$. Then we get a natural monomorphism $i:\calX_G\to \calX_G^{\circ}$ and a forgetful morphism $\pi:\calX_G^{\circ}\to \calX_{\GL_d}\times \calX_{\GL_{t(d)-1}}$.

For the morphism $\pi:\calX_G^{\circ}\to \calX_{\GL_d}\times \calX_{\GL_{t(d)-1}}$, we can actually show it is representable in formal schemes. This boils down to showing the surjective Hom functor $\underline\Hom^{\varphi,\Gamma}_{\surj}(\calE,\calL)$ between two \'etale $(\varphi,\Gamma)$-modules is representable in formal schemes (cf. Lemma \ref{representability-phi}). In the context of $\Bun_G$ associated to a curve, this is the representability of the Quot functor. In our context, the interpretation of surjectivity is involved. In fact, the de Rham prestack functor naturally appears, which accounts for the representability in formal schemes instead of schemes.

For the monomorphism $i:\calX_G\to \calX_G^{\circ}$, the situation is much more complicated. Again, in the context of $\Bun_G$ associated to a curve, the representability of this morphism is basically guaranteed by \cite{Ols06}. But there is no shortcut in our context. The basic tool for proving such representability is Artin's criteria. We have to check all the axioms by hand (cf. Proposition \ref{etale}, \ref{effectivity}, \ref{etale-gamma} and \ref{main-2}). Among them, the effectivity of formal objects and the openness of versality are much more delicate yet more interesting. For the effectivity of formal objects, the subtlety lies in that the Emerton--Gee stack for $\GL_n$ is only a formal algebraic stack. So the formal objects consisting of \'etale $(\varphi,\Gamma)$-modules are not expected to be effective. This makes the effectivity of formal objects in our context quite surprising.

In order to prove the openness of versality, Artin introduced the deformation theory and the obstruction theory and some crucial conditions they need to satisfy in \cite{Art74}. As $i:\calX_G\to \calX_G^{\circ}$ is a monomorphism, the deformation theory is actually trivial in our context. But we do need to construct a reasonable obstruction theory and prove it satisfies Artin's conditions (see Section \ref{obstruction}). This might be the most difficult  part in the proof of Theorem \ref{main-1-intro}. In fact, we have to use derived algebraic geometry, more precisely the theory of cotangent complexes of derived Artin stacks, to deal with some technical issues. The basic observation is that we can similarly define a stack $(BG)^{\circ}$ and a map $BG\to (BG)^{\circ}$ as above. Regarding them as derived stacks, the relative cotangent complex of $BG\to (BG)^{\circ}$ will naturally admit a $(\varphi,\Gamma)$-structure and we can find an obstruction class stable under the $(\varphi,\Gamma)$-action, which will finally give us a good candidate of the desired obstruction theory.

\begin{rmk}
    It is not obvious to us if the obstruction theory we constructed also works for \'etale $\varphi$-modules. The subtlety lies in that we are not sure whether the Frobenius invariants and coinvariants of finitely generated \'etale $\varphi$-modules are still finite over the coefficient ring or not.
\end{rmk}

\begin{rmk}\label{rmk-quasi-compactness}
    In \cite{Lin23}, Lin can also prove the quasi-compactness of the corresponding Emerton--Gee stacks. This requires to study the mod $p$ representations in detail. In fact, we have been informed by Lin that the quasi-compactness in our case can be deduced from his paper \cite{lin2023deligne} at least when $p>3$.
\end{rmk}

\subsection{Derived Emerton--Gee stack for general groups $G$}
As in the case of $\GL_d$, we can also construct a derived stack of Laurent $F$-crystals with $G$-structure on the asbolute prismatic site $(\calO_K)_{\Prism}$, which is hoped to be classical as well. For technical reasons, we will use the transveral absolute prismatic site $(\calO_K)_{\Prism}^{\rm tr}$, which simply consists of all the transveral prisms in $(\calO_K)_{\Prism}$, i.e. prisms $(A,I)$ with $A/I[p]=0$.

\begin{dfn}
Let $\textbf{Nilp}_{\bZ_p}$ be the $\infty$-category of $p$-nilpotent animated rings and $\bfAni$ be the $\infty$-category of anima. We define the derived stack of Laurent $F$-crystals $\frakX_G:\textbf{Nilp}_{\bZ_p}\to \bfAni$ by sending each $R\in\textbf{Nilp}_{\bZ_p}$ to the anima 
\[
\frakX_G(R):=\varprojlim_{(A,I)\in (\calO_K)_{\Prism}^{\rm tr}}BG(A\widehat\otimes_{\bZ_p}^{\bL}R[\frac{1}{I}])^{\varphi=1}
\]
\end{dfn}

To study $\frakX_G$, one might try to use the Tannakian formalism again. Unfortunately, there seems to be no Tannaka duality in the setting of ($p$-nilpotent) animated rings. This causes lots of trouble when we try to prove the classicality of the derived stack $\frakX_G$, or even to prove it satisfies the sheaf property as there is no analogue of Drinfeld--Mathew's descent result in the case of general groups, which is one of the reasons that we resort to transversal prisms. To fix this issue, we will take a general reduction procedure, i.e. start with discrete animated rings, where the Tannaka duality is available, and then use cotangent complex to linearise everything.

As in \cite{Min23}, we can still prove the following result.
\begin{thm}[Theorem \ref{classical}]\label{step1}
    Let $G$ be a flat affine group scheme of finite type over $\bZ_p$. There is an equivalence between the stack $\calX_G$ and the underlying classical stack $^{\rm cl}\frakX_G$ of the derived stack $\frakX_G$.
\end{thm}

This is the starting point of proving  the classicality of the derived stack $\frakX_G$. We first look at the case where $G$ is a connected reductive group. The basic tool is the same as \cite{Min23}, i.e. using \cite[Chapter 1, Proposition 8.3.2]{GR17}. In fact, this is the only way we know to prove classicality of a general derived prestack. \cite[Chapter 1, Proposition 8.3.2]{GR17} enables us to compare two derived prestacks admitting a deformation theory by comparing their pro-cotangent complexes at classical points. We refer to \cite[Section 3]{Min23} for more details on the derived deformation theory.

\begin{thm}[Theorem \ref{main-connected}]\label{classicality}
    Let $G$ be a connected reductive group over $\bZ_p$. The derived stack $\frakX_G$ is classical up to nilcompletion, i.e. $(\Lan ^{\cl}\frakX_G)^{\#,\nil}\simeq \frakX_G^{\nil}$, where $(\Lan ^{\cl}\frakX_G)^{\#,\nil}$ is the nilcompletion of the \'etale sheafification of the left Kan extension of $^{\cl}\frakX_G$ along the inclusion $\Nilp_{\bZ_p}\hookrightarrow \textbf{Nilp}_{\bZ_p}$.
\end{thm}

To prove this theorem, we have to show both $(\Lan ^{\cl}\frakX_G)^{\#,\nil}$ and $\frakX_G^{\nil}$ admit a deformation theory as required by \cite[Chapter 1, Proposition 8.3.2]{GR17}, and prove they are locally almost of finite type, which makes it sufficient to compare pro-cotangent complexes at  points valued in discrete algebras of finite type. The proofs of these results are quite different from \cite{Min23}: we argue by induction on the truncation degrees of animated rings and use the cotangent complexes.

Next in order to compare pro-cotangent complexes at  points valued in discrete algebras of finite type, the naive idea is to use some arguments involving faithfully flat descent and Nakayama's lemma to reduce to the points valued in finite fields. This works for complexes but not for pro-complexes. Fortunately, we can actually show the pro-cotangent complexes at  points valued in discrete algebras of finite type are indeed complexes. For $(\Lan ^{\cl}\frakX_G)^{\#,\nil}$, the important input is that $\calX_G$ is a formal algebraic stack\footnote{It might be even better to know $\calX_G$ is an ind-algebraic stack. This is true if $\calX_G$ is quasi-compact, which is the case at least when $p>3$ by Remark \ref{rmk-quasi-compactness}. }. For $\frakX_G^{\nil}$, we can directly show the pro-cotangent complex is given by the shifted Herr complex associated to the adjoint \'etale $(\varphi,\Gamma)$-module corresponding to the point.

\begin{rmk}
    In \cite{EG22}, it is shown that the Herr complex associated to the adjoint \'etale $(\varphi,\Gamma)$-module corresponding to a point of finite type provides a good obstruction theory for the (classical) Emerton--Gee stack. This makes our derived extension of the Emerton--Gee stack fit into a picture similar to the perfect obstruction theory introduced in \cite{BF97} and studied in \cite{STV15}, where derived extensions of Deligne--Mumford stacks are studied in order to provide a $[-1,0]$-perfect obstruction theory. The shifted Herr complex is a perfect complex of amplitude $[-1,1]$ as for Emerton--Gee stacks, we essentially consider the ``unframed local deformation theory" instead of the ``framed local deformation theory".
\end{rmk}

Once we are reduced to classical points valued in finite fields, we can use the same strategy of \cite[Proposition 3.51]{Min23}, combined with Theorem \ref{inf} below due to Paškūnas--Quast,  to prove Theorem \ref{classicality}. This strategy requires us to introduce the work \cite{Zhu20} of Xinwen Zhu on derived representations and discuss how to relate two different kinds of left Kan extension: one is from discrete commutative rings to animated rings, the other is from discrete finite Artinian local rings to animated Artinian local rings. We also remark that Theorem \ref{inf} will guarantee an infinitesimal version of Theorem \ref{classicality} by \cite[lemma 7.5]{GV18}.

Now we shift our focus to the generalised reductive groups introduced in \cite{PQ24}.

\begin{dfn}[{\cite[Definition 2.5]{PQ24}}]
    A generalised reductive group scheme over a scheme $S$ is a smooth
aﬃne $S$-group scheme $G$, such that the geometric fibres of $G^0$ are reductive and $G/G^0\to S$ is finite.
\end{dfn}
Paškūnas and Quast are able to prove the following important result about the framed local deformation ring, which generalises the $\GL_n$-case proved in \cite{BIP23}.

\begin{thm}[{\cite[Theorem 1.1]{PQ24}}]\label{inf}
   Let $G$ be a generalised reductive group over $\bZ_p$. Let $k_f$ be a finite field of characteristic $p$. Let $\bar\rho:G_K\to G(k_f)$ be a residual representation. Then the framed local Galois deformation ring $R^{\square}_{\bar \rho}$ is a local complete intersection.
\end{thm}

Note that in the setting of generalised reductive groups, for a representation $\rho:G_K\to G(R)$, we will only consider the $G^\circ$-conjugation instead of the whole $G$-conjugation. This is then compatible with the setting of the Langlands parameters. In order to be compatible with Galois representations, we have to consider a modified version of the Emerton--Gee stack.

\begin{dfn}[Definition \ref{dfn-gen-EG}, \ref{dfn-gen-derived}]\label{intro-dfn}
      For a generalised reductive group $G$ over $\bZ_p$, we define a stack $\calX_G^{\gen}$ by the following pullback diagram
    \begin{equation}
        \xymatrix{
        \calX_G^{\gen}\ar[r]\ar[d]& \Spf(\bZ_p)\ar[d]\\
        \calX_G\ar[r]& \calX_{\bar G}.
        }
    \end{equation}
    where $\bar G:=G/G^{\circ}$ and the right vertical map $\Spf(\bZ_p)\to \calX_{\bar G}$ means the trivial \'etale $(\varphi,\Gamma)$-module with $\bar G$-structure. As in practice we really care about $\calX_G^{\gen}$ instead of $\calX_G$, we call $\calX_G^{\gen}$ the Emerton--Gee stack associated to the generalised reductive group $G$.
    
    We can also define a modified derived stack $\frakX_G^{\gen}$ of Laurent $F$-crystals by replacing $\calX_G$ (resp. $\calX_{\bar G}$)  in the above diagram with $\frakX_G$ (resp. $\frakX_{\bar G}$).
\end{dfn}

    It is essential to consider this modified version of (derived) Emerton--Gee stack when we want to relate it to Galois representations. In fact, let $R$ be a finite Artinian local ring, a Galois representation $\rho:G_K\to G(R)$ is only concerned with the trivial $G$-torsor over $R$. When $G$ is a connected reductive group, every $G$-torsor over $R$ is trivial, which can be proved by reducing to the finite field case and using Lang's theorem. If $G/G^{\circ}$ is nontrivial, then there might be non-trivial $G$-torsors over $R$. So we can not expect \'etale $(\varphi,\Gamma)$-modules with $R$-coefficient to be the same as Galois representations over $R$ for such group $G$. 

    Now we state the classicality result for generalised reductive groups.

    \begin{thm}[Theorem \ref{main-gen}]\label{gen}
         Let $G$ be a generalised reductive group over $\bZ_p$. The derived stack $\frakX^{\rm gen}_G$ is classical up to nilcompletion, i.e. $(\Lan ^{\cl}\frakX^{\rm gen}_G)^{\#,\nil}\simeq \frakX_G^{\nil}$, where $(\Lan ^{\cl}\frakX^{\rm gen}_G)^{\#,\nil}$ is the nilcompletion of the \'etale sheafification of the left Kan extension of $^{\cl}\frakX^{\rm gen}_G$ along the inclusion $\Nilp_{\bZ_p}\hookrightarrow \textbf{Nilp}_{\bZ_p}$.
    \end{thm}
    In the case where $G$ is the $L$-group of a reductive group over $K$, Theorem \ref{gen} then completes the picture that the derived stack of local Langlands parameters is classical for arbitrary reductive groups.

    \subsection*{Acknowledgments}
We would like to thank Toby Gee and Vytautas Paškūnas for very helpful discussions regarding this work. We also thank Zhongyipan Lin for useful correspondence and suggestions and Ning Guo for some helpful discussions.

The author has received funding from the European Research Council (ERC) under the European Union's Horizon 2020 research and innovation programme (grant agreement No. 884596).





\section{Emerton--Gee stack for general groups}
Let $K$ be a finite extension of $\bQ_p$ with residue field $k$. Let $G$ be a flat affine group scheme of finite type over $\bZ_p$. We will construct a moduli stack of \'etale $(\varphi,\Gamma)$-modules with $G$-structue and show it is a formal algebraic stack locally of finite presentation over $\Spf(\bZ_p)$.

We write $K(\zeta_{p^{\infty}})$ to denote the extension of $K$ obtained by adjoining all $p$-power roots of unity. The Galois group $\Gal(K(\zeta_{p^{\infty}})/K)$ is naturally identified with an open subgroup of $\bZ_p^{\times}$. Let $K_{\cyc}$ denote the unique subextension of $K(\zeta_{p^{\infty}})$ whose Galois group over $K$ is isomorphic to $\bZ_p$ and $k_{\infty}$ be its residue field. We write $\Gamma:=\Gal(K_{\cyc}/K))$ and fix a topological generator $\gamma\in \Gamma$.

As in \cite{EG22}, for any $p$-nilpotent ring $R$, we write $\bfA_R^+=(W(k_{\infty})\otimes_{\bZ_p}R)[[T]]$ and $\bfA_R=(W(k_{\infty})\otimes_{\bZ_p}R)((T))$. There are compatible $\varphi$-action and $\Gamma$-action on $\bfA_R$. But if $K$ is ramified, $\bfA_R^+$ might not be $\varphi$-stable.

\begin{dfn}
     An \'etale $(\varphi,\Gamma)$-module with $R$-coefficient is a finitely projective module $M$ over $\bfA_R$ together with a semi-linear $\varphi$-action $\varphi_M:M\to M$ and a continuous semi-linear $\Gamma$-action such that
     \begin{enumerate}
         \item The $\varphi$-action commutes with the $\Gamma$-action;
         \item The $\bfA_R$-linear map $\varphi^*_MM\to M$ induced by $\varphi_M$ is an isomorphism.
     \end{enumerate}
     Let $\Mod^{\varphi,\Gamma}(\bfA_R)$ denote the category of \'etale $(\varphi,\Gamma)$-modules over $\bfA_R$. Forgetting about the $\Gamma$-action, we can also define the category $\Mod^{\varphi}(\bfA_R)$ of \'etale $\varphi$-modules over $\bfA_R$.
\end{dfn}

\begin{dfn}[The Emerton--Gee stack for $\GL_d$]
    The Emerton--Gee stack for $\GL_d$ is the functor $\calX_{\GL_d}:\Nilp_{\bZ_p}\to {\rm Groupoids}$ which sends each $p$-nilpotent ring $R$ to the $1$-groupoid of \'etale $(\varphi,\Gamma)$-modules of rank $d$ over $\bfA_R$.
\end{dfn}

To define \'etale $(\varphi,\Gamma)$-modules with $G$-structure, we will adopt the approach of the Tannakian formalism.

\begin{dfn}
Let $R\in \Nilp_{\bZ_p}$. An \'etale $(\varphi,\Gamma)$-modules with $G$-structure and $R$-coefficient is an exact $\otimes$-functor $F:\Rep(G)\to \Mod^{\varphi,\Gamma}(\bfA_R)$, where $\Rep(G)$ is the category of algebraic representations of $G$ on finite free $\bZ_p$-modules. Let $F_1, F_2$ be two such functors. A morphism $F_1\to F_2$ is a natural transformation $\calU: F_1\to F_2$ satisfying $\calU_{V_1\otimes V_2}\simeq\calU_{V_1}\otimes\calU_{V_2}$.

Let $\Mod^{\varphi,\Gamma}_{G}(\bfA_R)$ denote the category of \'etale $(\varphi,\Gamma)$-modules with $G$-structure and $R$-coefficient and $\Mod^{\varphi,\Gamma}_{G}(\bfA_R)^{\simeq}$ be its underlying groupoid.
\end{dfn}

Now we can define the prestack of  \'etale $(\varphi,\Gamma)$-modules with $G$-structure.
\begin{dfn}
    Let $\calX_G: \Nilp_{\bZ_p}\to {\rm Groupoids}$ be the prestack sending each $R\in \Nilp_{\bZ_p}$ to the groupoid $\Mod^{\varphi,\Gamma}_{G}(\bfA_R)^{\simeq}$.
\end{dfn}

In order to study the prestack $\calX_G$, we need to have a good understanding of $\Mod^{\varphi,\Gamma}_{G}(\bfA_R)^{\simeq}$. Note that if we forget about the $(\varphi,\Gamma)$-structure, then we get a functor $\calX_G(R)\to BG(\bfA_R)$ for each $R$. Traditionally, there are three equivalent definitions of $G$-torsors, i.e. the geometric one, the Tannakian one and the cohomological one (cf. \cite[Appendix to Lecture 19]{WS20}).

For our purpose, we also need another perspective introduced in \cite{Bro13}. Let us first recall some important notions in \cite{Bro13}.

\begin{dfn}[Tensorial construction]
   Let $V$ be an algebraic representation of $G$ on finite free $\bZ_p$-modules. By a tensorial construction $t(V)$, we mean some finite iteration of the operations $\bigwedge^i, \otimes, \oplus, \Sym^i, (\cdot)^{\vee}$.
\end{dfn}

Broshi gave another description of the algebraic group $G$, which leads to a new interpretation of the $G$-torsors.

\begin{thm}[{\cite[Theorem 1.1]{Bro13}}]\label{Broshi-G}
Let $G$ be a flat affine group scheme of finite type over a Dedekind scheme $X$. Then there is an algebraic representation $V$ of $G$, a tensorial construction $t(V)$ and a locally split line bundle $L\subset t(V)$ such that $G\simeq \underline\Aut(V,L)$, where $\underline\Aut(V,L)$ is the representable functor whose $T$-points are automorphisms $f$ of $V\otimes \calO_T$ such that $f(L\otimes \calO_T)=L\otimes \calO_T$.
\end{thm}

In particular, this theorem applies to $X=\Spec(\bZ_p)$. This is indeed a generalisation of Chevalley's theorem to mixed characteristic.

From now on, we fix a triple $(V,t(V),L)$ as in Theorem \ref{Broshi-G} such that $G\simeq \underline\Aut(V,L)$. Then we have the following definition.

\begin{dfn}[Twist of $(V,L)$]\label{Y-twist}
    Let $Y$ be a scheme over $\bZ_p$. We define a $Y$-twist to be a pair $(\calE_Y,\calL_Y)$ where $\calE_Y$ is a vector bundle on $Y$ and $\calL_Y$ is a locally split line bundle in $t(\calE_Y)$ such that $(\calE_Y,\calL_Y)$ is fppf locally isomorphic to $(V,L)$, i.e. there is a fppf cover $Y'\to Y$ and an isomorphism $f: \calE_{Y'}\xrightarrow{\simeq}V_{Y'}$ such that $f(\calL_{Y'})=L_{Y'}$. An isomorphism of $Y$-twists $g:(\calE_Y,\calL_Y)\to (\calE'_Y,\calL'_Y)$ is an isomorphism of vector bundles $g:\calE_Y\to \calE'_Y$ such that $g(\calL_Y)=\calL'_Y$. Let ${\rm Twist}_Y(V,L)$ denote the groupoid of $Y$-twists.
\end{dfn}

One nice consequence of the characterisation $G\simeq \underline\Aut(V,L)$ is the following equivalent descriptions of $G$-torsors.

\begin{thm}[{\cite[Theorem 1.2, Corollary 1.3]{Bro13}}]\label{G-description}
    Let $G$ be a flat affine group scheme of finite type over $\bZ_p$. For any scheme $Y$ over $\bZ_p$, there is a natural equivalence functorial in $Y$ of the following groupoids.
    \begin{enumerate}
        \item The groupoid $BG(Y)$ of geometric $G$-torsors over $Y$.
        \item The groupoid $\Fun^{{\rm ex},\otimes}(\Rep(G),\Bun(Y))$ of exact $\otimes$-functors $F:\Rep(G)\to \Bun(Y)$, where $\Bun(Y)$ is the category of vector bundles on $Y$.
        \item The groupoid ${\rm Twist}_Y(V,L)$ of $Y$-twists of $(V,L)$.
    \end{enumerate}

    Explicitly, given a geometric $G$-torsor $P$, the corresponding functor $F\in \Fun^{{\rm ex},\otimes}(\Rep(G),\Bun(Y))$ sends $V\in \Rep(G)$ to $P\times^G (V\otimes_{\bZ_p}\calO_Y)$. Given a functor $F\in \Fun^{{\rm ex},\otimes}(\Rep(G),\Bun(Y))$, the corresponding $Y$-twist is given by $(F(V),F(L))$. Given a $Y$-twist $(\calE_Y,\calL_Y)$, the corresponding $G$-torsor is $\underline\Isom((V_Y,L_Y),(\calE_Y,\calL_Y))$.
\end{thm}

Given a $Y$-twist $(\calE,\calL)$, let us explain in detail why $P\times^G (V\otimes_{\bZ_p}\calO_Y)=\underline\Isom((V_Y,L_Y),(\calE_Y,\calL_Y))\times^GV_Y$ is isomorphic to $\calE_Y$. In fact, the vector bundle $P\times^G (V\otimes_{\bZ_p}\calO_Y)$ corresponds to the descent data $(V_P, \epsilon: V_P\otimes_{\calO_P,\pi_1}\calO_{P^1}\cong V_P\otimes_{\calO_P,\pi_2}\calO_{P^1}$, where $P^1=P\times_YP$ and $\pi_1,\pi_2$ correspond to the projections from $P^1$ to $P$. It turns out $\epsilon$ is just the action given by the unique element in $G(P^1\to Y)$ turning $pr_2:P^1\to P$ into $pr_1: P^1\to P$. Note that there is a universal isomorphism $\alpha:V_P\cong \calE_P$ corresponding to the identity $id: P\to P=\underline\Isom((V_Y,L_Y),(\calE_Y,\calL_Y))$. Then we have $pr_1^*\alpha: V_{P^1}\cong \calE_{P^1}$ (resp. $pr_2^*\alpha: V_{P^1}\cong \calE_{P^1}$) corresponding to the projection $pr_1:P^1\to P$ (resp. $pr_2: P^1\to P$). Now recall that the action of $G=\underline\Isom((V_Y,L_Y),(V_Y,L_Y))$ on $P=\underline\Isom((V_Y,L_Y),(\calE_Y,\calL_Y))$ is just by composition. Then one can check there is a commutative diagram
\begin{equation*}
    \xymatrix{
    V_P\otimes_{\calO_P,\pi_1}\calO_{P^1}\ar[r]^{\epsilon}\ar[d]^{pr_1^*\alpha}& V_P\otimes_{\calO_P,\pi_2}\calO_{P^1}\ar[d]^{pr_2^*\alpha}\\
    \calE_{P^1}\ar[r]^{id}& \calE_{P^1}.
    }
\end{equation*}
This actually gives an isomorphism of descent data between $(V_P, \epsilon: V_P\otimes_{\calO_P,\pi_1}\calO_{P^1}\cong V_P\otimes_{\calO_P,\pi_2}\calO_{P^1})$ and $(\calE_P, id: \calE_{P^1}\to \calE_{P^1})$. The latter just corresponds to the vector bundle $\calE_Y$. So we get an isomorphism $P\times^G (V\otimes_{\bZ_p}\calO_Y)\cong \calE_Y$.

\begin{rmk}
    The assumption that $Y$ is faithfully flat over $\bZ_p$ in \cite[Theorem 1.2]{Bro13} is unnecessary as mentioned in the footnote of \cite[Theorem 19.5.1]{WS20}. Also, we often do not distinguish a geometric $G$-torsor from its associated cohomological $G$-torsor (see \cite[Theorem 19.5.1]{WS20}).
\end{rmk}

\subsection{\'Etale $(\varphi,\Gamma)$-modules with $G$-structure}\label{subsec2.1}
Let $R\in \Nilp_{\bZ_p}$. The goal of this subsection is to show \'etale $(\varphi,\Gamma)$-modules with $G$-structure and $R$-coefficient are equivalent to \'etale $(\varphi,\Gamma)$-$\bfA_R$-twists, which are $\bfA_R$-twists equipped with compatible $(\varphi,\Gamma)$-actions.
Note that we can apply Theorem \ref{G-description} to $Y=\Spec(\bfA_R)$. In particular, we get an equivalence $\Fun^{{\rm ex},\otimes}(\Rep(G),\Bun(\bfA_R))\simeq {\rm Twist}_{\bfA_R}(V,L)$. Then the Frobenius map $\varphi_{\bfA_R}:\bfA_R\to \bfA_R$ induces a functor $\Fun^{{\rm ex},\otimes}(\Rep(G),\Bun(\bfA_R))\to \Fun^{{\rm ex},\otimes}(\Rep(G),\Bun(\bfA_R))$ by composing with the exact $\otimes$-functor $\varphi^*_{\bfA_R}:\Bun(\bfA_R)\to \Bun(\bfA_R)$. In the same way, we also have $\varphi^*_{\bfA_R}:{\rm Twist}_{\bfA_R}(V,L)\to {\rm Twist}_{\bfA_R}(V,L)$ as pullback preserves tensorial constructions. Then we get a commutative diagram
\begin{equation}\label{Diagram-1}
    \xymatrix@=1cm{
    \Fun^{{\rm ex},\otimes}(\Rep(G),\Bun(\bfA_R))\ar[d]^{\Ev}\ar[r]^{\varphi^*_{\bfA_R}} & \Fun^{{\rm ex},\otimes}(\Rep(G),\Bun(\bfA_R))\ar[d]^{\Ev}\\
    {\rm Twist}_{\bfA_R}(V,L)\ar[r]^{\varphi^*_{\bfA_R}} & {\rm Twist}_{\bfA_R}(V,L).
    }
\end{equation}

Note that the category $\Mod^{\varphi}(\bfA_R)$ of \'etale $\varphi$-modules over $\bfA_R$ can be defined via the following $2$-fiber product in $(2,1)$-category of categories
\begin{equation}\label{Diagram-2}
    \xymatrix@=1cm{
    \Mod^{\varphi}(\bfA_R)\ar[d]\ar[r] & \Bun(\bfA_R)\ar[d]^{(\varphi^*_{\bfA_R},id)}\\
    \Bun(\bfA_R)\ar[r]^-{\Delta}& \Bun(\bfA_R)\times \Bun(\bfA_R).
    }
\end{equation}

\begin{dfn}
    Let $\Twist^{\varphi}_{\bfA_R}(V,L)$ denote the $2$-fiber product
    \begin{equation}
    \xymatrix@=1cm{
    \Twist^{\varphi}_{\bfA_R}(V,L)\ar[d]\ar[r] & \Twist_{\bfA_R}(V,L)\ar[d]^{(\varphi^*_{\bfA_R},id)}\\
    \Twist_{\bfA_R}(V,L)\ar[r]^-{\Delta}& \Twist_{\bfA_R}(V,L)\times \Twist_{\bfA_R}(V,L).
    }
\end{equation}
and $\Fun^{{\rm ex},\otimes,\varphi}(\Rep(G),\Bun(\bfA_R))$ denote the $2$-fiber product
\begin{equation}
    \xymatrix@=1cm{
    \Fun^{{\rm ex},\otimes,\varphi}(\Rep(G),\Bun(\bfA_R))\ar[d]\ar[r] & \Fun^{{\rm ex},\otimes}(\Rep(G),\Bun(\bfA_R))\ar[d]^{(\varphi^*_{\bfA_R},id)} \\
    \Fun^{{\rm ex},\otimes}(\Rep(G),\Bun(\bfA_R))\ar[r]^-{\Delta} & \Fun^{{\rm ex},\otimes}(\Rep(G),\Bun(\bfA_R))\times \Fun^{{\rm ex},\otimes}(\Rep(G),\Bun(\bfA_R)).
    }
\end{equation}
\end{dfn}

Explicitly, an object in $\Twist^{\varphi}_{\bfA_R}(V,L)$ is a pair of \'etale $\varphi$-modules $(\calE_{\bfA_R}, \calL_{\bfA_R})$ such that the inclusion $\calL_{\bfA_R}\to t(\calE_{\bfA_R})$ is compatible with the induced Frobenius action on $t(\calE_{\bfA_R})$. Note that $\calL_{\bfA_R}$ is not required to be locally split as \'etale $\varphi$-modules.
\begin{lem}\label{Lem-varphi}
    There is an equivalence of groupoids $\Ev:\Fun^{{\rm ex},\otimes}(\Rep(G),\Mod^{\varphi}(\bfA_R))\xrightarrow{\simeq} \Twist^{\varphi}_{\bfA_R}(V,L)$.
\end{lem}

\begin{proof}
    As the two functors $\Mod^{\varphi}(\bfA_R)\to \Bun(\bfA_R)$ sending $(\varphi^*M\xrightarrow{\simeq}M)$ to $\varphi^*M$, $M$ respectively are both exact $\otimes$-functors, we then have \[\Fun^{{\rm ex},\otimes}(\Rep(G),\Mod^{\varphi}(\bfA_R))\simeq \Fun^{{\rm ex},\otimes,\varphi}(\Rep(G),\Bun(\bfA_R))\] by Diagram \ref{Diagram-2} and the fact that the morphisms between exact $\otimes$-functors are required to be symmetrical monoidal. Then this lemma follows from Diagram \ref{Diagram-1}.
\end{proof}

Now choosing a topological generator $\gamma$ of $\Gamma$, we can consider the subgroup $\Gamma_{\rm disc}:=\langle\gamma\rangle\subset \Gamma$ generated by $\gamma$. We can then define $\Fun^{{\rm ex},\otimes}(\Rep(G),\Mod^{\varphi,\Gamma_{\rm disc}}(\bfA_R))$ and $\Twist^{\varphi,\Gamma_{\disc}}_{\bfA_R}(V,L)$ similarly. 

\begin{lem}\label{Gamma_disc}
    There is an equivalence of groupoids $\Ev:\Fun^{{\rm ex},\otimes}(\Rep(G),\Mod^{\varphi,\Gamma_{\rm disc}}(\bfA_R))\xrightarrow{\simeq} \Twist^{\varphi,\Gamma_{\disc}}_{\bfA_R}(V,L)$.
\end{lem}

\begin{proof}
    Similar to Lemma \ref{Lem-varphi}.
\end{proof}

By replacing $\Gamma_{\disc}$ with $\Gamma$, we can define the groupoids $\Fun^{{\rm ex},\otimes}(\Rep(G),\Mod^{\varphi,\Gamma}(\bfA_R))$ and $\Twist^{\varphi,\Gamma}_{\bfA_R}(V,L)$ by requiring the $\Gamma$-actions to be continuous. As $\Mod^{\varphi,\Gamma}(\bfA_R)$ is a full subcategory of $\Mod^{\varphi,\Gamma_{\rm disc}}(\bfA_R)$ and the inclusion functor is an exact $\otimes$-functor, we see $\Fun^{{\rm ex},\otimes}(\Rep(G),\Mod^{\varphi,\Gamma}(\bfA_R))$ is also a full subcategory of $\Fun^{{\rm ex},\otimes}(\Rep(G),\Mod^{\varphi,\Gamma_{\disc}}(\bfA_R))$. By Lemma \ref{Gamma_disc}, we get a fully faithful functor $\Fun^{{\rm ex},\otimes}(\Rep(G),\Mod^{\varphi,\Gamma}(\bfA_R))\to \Twist^{\varphi,\Gamma_{\disc}}_{\bfA_R}(V,L)$, whose essential image lies in $\Twist^{\varphi,\Gamma}_{\bfA_R}(V,L)$.

Then we have the following equivalence.

\begin{lem}\label{Lemma-Gamma}
    The evaluation functor $Ev:\Fun^{{\rm ex},\otimes}(\Rep(G),\Mod^{\varphi,\Gamma}(\bfA_R))\xrightarrow{\simeq} \Twist^{\varphi,\Gamma}_{\bfA_R}(V,L)$ is an equivalence.
\end{lem}

It is clear that this functor is fully faithful. So it remains to prove it is essentially surjective. To prove this, we need to use the explicit functors described in Theorem \ref{G-description}, especially the geometric $G$-torsors. We assume $R$ is of finite type over $\bZ_p$ at the moment and leave the general case to the next section after we prove the stack of \'etale $(\varphi,\Gamma)$-twists is limit preserving.


Given an $\bfA_R$-twist $(\calE_{\bfA_R},\calL_{\bfA_R})\in \Twist^{\varphi,\Gamma}_{\bfA_R}(V,L)$, the geometric $G$-torsor corresponding is \[P:=\underline\Isom((V_{\bfA_R},L_{\bfA_R}),(\calE_{\bfA_R},\calL_{\bfA_R})).\] In particular, we have compatible $\varphi$-action and $\Gamma$-action on $P$ induced by those on $(\calE_{\bfA_R},\calL_{\bfA_R})$. More precisely, we can consider the isomorphism sheaf $P':=\underline\Isom((V_{\bfA_R},L_{\bfA_R}),(\varphi^*\calE_{\bfA_R},\varphi^*\calL_{\bfA_R}))$. As $(V,L)$ is actually defined over $\bZ_p$ which is $\varphi$-invariant, there is a natural isomorphism $(V_{\bfA_R},L_{\bfA_R})\simeq (\varphi^*V_{\bfA_R},\varphi^*L_{\bfA_R})$. Then for any scheme $Y$ over $\Spec(\bfA_R)$, we have $\Hom(Y,P')=\Hom(\varphi_*Y,P)$ where $\varphi_*Y$ is the $\bfA_R$-scheme $Y\to \Spec(\bfA_R)\xrightarrow{\varphi}\Spec(\bfA_R)$. This implies $P'=\varphi^*P$. So we get an isomorphism $\varphi^*P\xrightarrow{\simeq}P$ over $\Spec(\bfA_R)$ induced by $(\varphi^*\calE_{\bfA_R},\varphi^*\calL_{\bfA_R})\xrightarrow{\simeq}(\calE_{\bfA_R},\calL_{\bfA_R})$. The same reason shows that the $G$-torsor $P$ is equipped with a compatible $\Gamma$-action. We will have to discuss the ``continuity" of the $\Gamma$-action on $P$. To this end, we need to introduce a natural topology on $\calO(P)$.

Note that $\bfA_R$ is a Banach ring with the $T$-adic topology. In particular, $\bfA_R$ contains a topologically nilpotent unit. As $R$ is finite type over $\bZ_p$, we see that $\bfA_R$ is Noetherian. Then by \cite[Remark 2.2.11]{KL15}, any finite $\bfA_R$-module is naturally a finite Banach $\bfA_R$-module.

We collect some useful lemmas about finitely generated modules over $\bfA_R$.

\begin{lem}
    \begin{enumerate}
        \item Let $M$ be a finitely generated $\bfA_R$-module. Then any submodule $N$ of $M$ is closed.
        \item For any $\bfA_R$-linear morphism $f:M_1\to M_2$ of finite $\bfA_R$-modules, $f$ is continuous and strict (i.e. the quotient topology on $f(M_1)$ is the same as the subspace topology inherited from $M_2$).
    \end{enumerate}
\end{lem}
\begin{proof}
    \begin{enumerate}
        \item Note that $\bfA_R$ contains a topological nilpotent unit $T$. Then by \cite[Lemma 2.2.6(b), Theorem 2.2.8]{KL15}, the open mapping theorem is valid. Now by the same proof of \cite[Proposition 3.7.2.1]{BGR84}, we see $N$ is closed.
        \item By  \cite[Remark 2.2.11(c)]{KL15}, the continuity is clear. Now by \cite[Theorem 2.2.8]{KL15}, the natural topology on $f(M_1)$ is equivalent to the quotient toplogy. So we just need to show the natural topology on $f(M_1)$ is equivalent to the subspace topology inherited from $M_2$. As the subspace topology makes $f(M_1)$ a finite Banach $\bfA_R$-module by Item (1), this follows from \cite[Remark 2.2.11(c)]{KL15}.
    \end{enumerate}
\end{proof}

\begin{dfn}[Filtered colimit topology]
    For any $\bfA_R$-module $M$, we can write $M=\colim M_i$ as the filtered colimit of all its finite $\bfA_R$-submodules. By endowing each $M_i$ with its natural topology, we can equip $M$ with the filtered colimit topology, i.e. every subset $N\subset M$ is open if and only if the intersection $N\cap M_i$ is open in $M_i$ for each $i$. 
\end{dfn}

The next lemma shows the compatibility of the filtered colimit topology and the natural topology for finite $\bfA_R$-modules.
\begin{lem}
    For any finitely generated $\bfA_R$-module $M$, the natural topology is equivalent to the filtered colimit topology.
\end{lem}
\begin{proof}
    Let $U$ be an open subset of $M$ under the natural topology. Then for any submodule $M_i$ of $M$, the intersection $U\cap M_i$ is open in $M_i$ under the subspace topology on $M_i$. But the subspace topology is equivalent to the natural topology, we see $U\cap M_i$ is open in $M_i$ under the natural topology. So $U$ is open in $M$ under the filtered colimit topologt. For thr converse, let $U'$ be an open subset of $M$ under the filtered colimit topology, then by the definition of the filtered colimit topology, $U'$ is open in $M$ under the natural topology.
\end{proof}

\begin{lem}
    Let $M$ be an $\bfA_R$-module equipped with the filtered colimit topology. Then any $\bfA_R$-submodule $N$ of $M$ is closed in $M$.
\end{lem}
\begin{proof}
    Note that the intersection of $N$ with each finite $\bfA_R$-submodule $M_i$ of $M$ is again finite $\bfA_R$-submodule of $M_i$. As any finite $\bfA_R$-submodule of $M_i$ is closed in $M_i$, we see that $N$ is closed in $M$.
\end{proof}

Write $P=\Spec(S)$ with $S$ being faithfully flat over $\bfA_R$. Equip $S$ with the filtered colimit topology.

\begin{lem}\label{Key-0}
    The $\Gamma$-action on $S$ is continuous.
\end{lem}

\begin{proof}
    Note that $P=\Spec(S)=\underline\Isom((V_{\bfA_R},L_{\bfA_R}),(\calE_{\bfA_R},\calL_{\bfA_R}))$. We have the following pullback diagram
    \begin{equation}\label{Diagram-5}
        \xymatrix@=1cm{
        \underline\Isom((V_{\bfA_R},L_{\bfA_R}),(\calE_{\bfA_R},\calL_{\bfA_R}))\ar[d]\ar[r] & \underline\Isom(L_{\bfA_R},\calL_{\bfA_R})\ar[d]\\
        \underline\Isom(V_{\bfA_R},\calE_{\bfA_R})\ar[r] & \underline\Hom(L_{\bfA_R},t(\calE_{\bfA_R}))
        }
    \end{equation}
    where the right vertical map is induced by the inclusion $\calL_{\bfA_R}\hookrightarrow t(\calE_{\bfA_R})$ and the bottom horizonal map is induced by the inclusion $L_{\bfA_R}\hookrightarrow t(V_{\bfA_R})$.

   Note that $\underline\Isom(V_{\bfA_R},\calE_{\bfA_R})=\Spec(S_1)$ is a closed subscheme of $\underline\Hom(V_{\bfA_R},\calE_{\bfA_R})\times \underline\Hom(\calE_{\bfA_R},V_{\bfA_R})=\Spec(\tilde S_1)$, where $\tilde S_1=\Sym_{\bfA_R}^{\bullet}(V_{\bfA_R}\otimes \calE_{\bfA_R}^{\vee})\otimes_{\bfA_R}\Sym_{\bfA_R}^{\bullet}(V_{\bfA_R}^\vee\otimes \calE_{\bfA_R})$. There is a $\Gamma$-action on $\tilde S_1$ induced by the $\Gamma$-action on $\calE_{\bfA_R}$. As the $\Gamma$-action on $\calE_{\bfA_R}$ is continuous, we see that the $\Gamma$-action on $V_{\bfA_R}\otimes \calE_{\bfA_R}^{\vee}$ is also continuous, which implies that the induced $\Gamma$-action on $\tilde S_1$ is also continuous with respect to the filtered colimit topology. 
   
   Let $\underline\Isom(L_{\bfA_R},\calL_{\bfA_R})=\Spec(S_2)$ be the closed subscheme of $\underline\Hom(L_{\bfA_R},\calL_{\bfA_R})\times \underline\Hom(\calL_{\bfA_R},L_{\bfA_R})=\Spec(\tilde S_2)$, where $S_2=\Sym_{\bfA_R}^{\bullet}(L_{\bfA_R}\otimes \calL_{\bfA_R}^{\vee})\otimes_{\bfA_R}\Sym_{\bfA_R}^{\bullet}(L_{\bfA_R}^{\vee}\otimes \calL_{\bfA_R})$. Then the $\Gamma$-actions on $L_{\bfA_R}\otimes \calL_{\bfA_R}^{\vee}$ and $\tilde S_2$ are also continuous. 

   Then we see the induced $\Gamma$-action on $\tilde S_1\otimes_{\bfA_R}\tilde S_2$ is also continuous. Write $\underline\Hom(L_{\bfA_R},t(\calE_{\bfA_R}))=\Spec(S_3)$. We then have a surjective $\bfA_R$-morphism $\tilde S_1\otimes_{\bfA_R}\tilde S_2\to S_1\otimes_{S_3}S_2$ with compatible $\Gamma$-actions, i.e. there is a commutative diagram
   \begin{equation}
       \xymatrix@=1cm{
       \Gamma\times \tilde S_1\otimes_{\bfA_R}\tilde S_2\ar@{->>}[d]\ar@{->>}[r] & \tilde S_1\otimes_{\bfA_R}\tilde S_2\ar@{->>}[d] \\
       \Gamma \times S_1\otimes_{S_3}S_2\ar@{->>}[r] & S_1\otimes_{S_3}S_2.
       }
   \end{equation}
  The upper horizontal map is continuous. Then by Lemma \ref{key-1}, the bottom horizontal map is also continuous, i.e. the $\Gamma$-action on $S=S_1\otimes_{S_3}S_2$ is continuous.

\end{proof}


\begin{lem}\label{key-1}
    Let $f:M\twoheadrightarrow N$ be a surjective $\bfA_R$-linear morphism of $\bfA_R$-modules. Equip $M$ and $N$ with the filtered colimit topology. Then we have \begin{enumerate}
        \item The morphism $f$ is continuous.
        \item If $M_1$ is a saturated closed (resp. open) subset of $M$, i.e. $M_1=f^{-1}(f(M_1))$, then $f(M_1)$ is a closed (resp. open) subset of $N$. In particular, this means $f$ is indeed open. 
    \end{enumerate} 
\end{lem}

\begin{proof}
\begin{enumerate}
    \item Let $N_1$ be a closed subset of $N$. Then we need to prove that for any finite $\bfA_R$-submodule $M_0$ of $M$, the intersection $f^{-1}(N_1)\cap M_0$ is closed in $M_0$. Note that $f(M_0)$ is a finite $\bfA_R$-submodule of $N$. Therefore, $N_1\cap f(M_0)$ is closed in $f(M_0)$. Let $f_0:M_0\to f(M_0)$ be the restriction of $f$ on $M_0$. Then we have $f^{-1}(N_1)\cap M_0=f_0^{-1}(N_1\cap f(M_0))$, which is then closed in $M_0$ as $f_0$ is continuous.
    
    \item Let $M_1$ be a saturated closed subset of $M$ and $N_0$ be a finite $\bfA_R$-submodule of $N$. Consider $N_0\backslash (N_0\cap f(M_1))$, the complement of $N_0\cap f(M_1)$ in $N_0$. It suffices to prove $N_0\backslash (N_0\cap f(M_1))$ is open in $N_0$. Let $M_i$ be any finite $\bfA_R$-submodule of $M$ such that $f(M_i)=N_0$ and $f_i:M_i\to N_0$ be the restriction of $f$ on $M_i$. By the open mapping theorem, we see that $f_i$ is open. In particular, the image of the complement of $M_1\cap M_i$ in $M_i$ under the map $f_i$ is open in $N_0$. Take the union of images of all finite $\bfA_R$-submodules which map onto $N_0$, which is exactly $N_0\backslash (N_0\cap f(M_1))$. This proves the case of saturated closed subsets, which also implies the case of saturated open subsets. 
    
      Now suppose $M'\subset M$ is an open subset of $M'$. Let $I$ be the kernel of $f$. Then we see that the saturation ${\rm Sat}(M'):=f^{-1}(f(M'))$ is exactly the set $I+M'$, which is open. So this means $f(M')$ is open.
\end{enumerate}
     
\end{proof}

Now we finish the proof of Lemma \ref{Lemma-Gamma} for finite type $\bZ_p$-algebra $R$.
\begin{proof}[Proof of Lemma \ref{Lemma-Gamma} for finite type $\bZ_p$-algebra $R$]
    Given a geometric $G$-torsor $P=\Spec(S)$ coming from $\Twist^{\varphi,\Gamma}_{\bfA_R}(V,L)$, the associated functor $F_P\in \Fun^{{\rm ex},\otimes}(\Rep(G),\Mod^{\varphi,\Gamma_{\disc}}(\bfA_R))$ is given by $P\times^G(W\otimes_{\bZ_p}\calO_{\bfA_R}))$ for any $W\in \Rep(G)$, which is a vector bundle due to the fppf descent of vector bundles. In particular, we have the equaliser
    \begin{equation}\label{equaliser}
    \xymatrix{
    F_P(W)\ar[r] & \calO(P)\otimes_{\bZ_p}W\ar@<.5ex>[r]\ar@<-.5ex>[r] &\calO(P)\otimes_{\bZ_p} \calO(P)\otimes_{\bZ_p}W
    }
\end{equation}
where $\calO(P)\otimes_{\bZ_p}W$ (resp. $\calO(P)\otimes_{\bZ_p} \calO(P)\otimes_{\bZ_p}W$) is the vector bundle on $P$ (resp. $P^1:=P\times_{\Spec(\bfA_R)}P$) given by the pullback of $P\times^G(W\otimes_{\bZ_p}\calO_{\bfA_R}))$ along $P\to \Spec(\bfA_R)$ (resp. $P^1\to \Spec(\bfA_R)$). The $\varphi$-action and $\Gamma$-action on $F_P(W)$ are inherited from $P$. 

Now we want to prove the $\Gamma$-action on $F_P(W)$ is continuous. Note that $F_P(W)$  is a finite projective $\bfA_R$-submodule of $S\otimes_{\bZ_p}W$. Then the subspace topology on $F_P(W)\subset S\otimes_{\bZ_p}W$ is equivalent to its natural topology. To see this, let $U$ be an open subset of $F_P(W)$ under the subspace topology, then there exists an open subset $U'$ of $S\otimes_{\bZ_p}W$ such that $U'\cap F_P(W)=U$. By the definition of filtered colimit topology, we see $U$ is an open subset of $F_P(W)$ under the natural topology. For the converse, let $V$ be an open subset of $F_P(W)$ under the natural topology. As $F_P(W)$ is a closed subset of $S\otimes_{\bZ_p}W$, we see the union of $V$ and the complement $S\otimes_{\bZ_p}W\backslash F_P(W)$ is open in $S\otimes_{\bZ_p}W$, whose intersection with $F_P(W)$ is $V$. So $V$ is open under the subspace topology.

Now as the $\Gamma$-action on $S\otimes_{\bZ_p}W$ is continuous by Lemma \ref{Key-0}, we see $\Gamma\times F_P(W)\to F_P(W)$ must also be continuous. So $F_P$ is indeed in $\Fun^{{\rm ex},\otimes}(\Rep(G),\Mod^{\varphi,\Gamma_{}}(\bfA_R))$. We are done.
\end{proof}

\begin{rmk}
    Let $R$ be a finite type $\bZ_p$-algebra. The proof of Lemma \ref{Lemma-Gamma} shows that these two groupoids are also equivalent to the groupoid of $G$-torsors equipped with compatible semi-linear $\varphi$ and $\Gamma$-actions such that the $\Gamma$-action is continuous with respect to the filtered colimit topology.
\end{rmk}

\subsection{Surjective Hom functor of \'etale $\varphi$-modules}\label{subsec2.2}
In this section, we first study the prestack $\calR_G$ of \'etale $\varphi$-modules with $G$-structure and then use this to show $\calX_G$ is also a formal algebraic stack locally of finite presentation over $\Spf(\bZ_p)$.

Recall we fixed a pair $(V,L)$ such that $G\cong \underline\Aut(V,L)$. 
\begin{dfn}

 Let $\calR_G$ denote the prestack $\Nilp_{\bZ_p}\to {\rm Groupoids}$ sending each $R\in \Nilp_{\bZ_p}$ to the groupoid $\Fun^{{\rm ex},\otimes}(\Rep(G),\Mod^{\varphi}(\bfA_R))$ ($\xrightarrow{\simeq} \Twist^{\varphi}_{\bfA_R}(V,L)$ by Lemma \ref{Lem-varphi}).
        
\end{dfn}

Instead of using locally split line bundles in the definition of twists of $(V,L)$, the following lemma enables us to work with surjections of vector bundles, which will be more convenient for our aim.

\begin{lem}\label{easy-description}
    The groupoid $\Twist^{\varphi}_{\bfA_R}(V,L)$ (resp. $\Twist^{\varphi,\Gamma}_{\bfA_R}(V,L)$) is equivalent to the groupoid of pairs $(\calE_{\bfA_R},f:t(\calE_{\bfA_R})\twoheadrightarrow \calF_{\bfA_R})$ where $\calE_{\bfA_R}$ is an \'etale $\varphi$-module (resp. \'etale $(\varphi,\Gamma)$-modules) and $f$ is a surjective morphism of \'etale $\varphi$-modules (resp. \'etale $(\varphi,\Gamma)$-modules) such that $ker(f)$ is finite projective over $\bfA_R$ of rank $1$, and there exists a fppf cover $Y\to \Spec(\bfA_R)$ such that $(\calE_{Y},f:t(\calE_{Y})\twoheadrightarrow \calF_{Y})$ is isomorphic to $(V_Y, pr: t(V_Y)\to t(V_Y)/L_Y)$. 
\end{lem}

\begin{proof}
    This follows from the fact: If $\calL_{\bfA_R}$ is locally split in $t(\calE_{\bfA_R})$ as vector bundles, then the quotient $t(\calE_{\bfA_R})/\calL_{\bfA_R}$ is also a vector bundle.
\end{proof}

 Let $d=\dim_{\bZ_p}V$ and $t(d)=\dim_{\bZ_p}t(V)$. Attached to the chosen pair $(V,L)$, there is a natural morphism of prestacks $\pi:\calR_G\to \calR_{d}\times_{\Spf(\bZ_p)}\calR_{t(d)-1}$ defined by 
 \[
(\calE_{\bfA_R},f:t(\calE_{\bfA_R})\twoheadrightarrow \calF_{\bfA_R})\mapsto (\calE_{\bfA_R},\calF_{\bfA_R})
 \]
 where $\calR_d$, resp. $\calR_{t(d)-1}$, is the prestack of finite projective \'etale $\varphi$-modules of rank $d$, resp. $t(d)-1$.

 In order to study the morphism $\pi$, let us introduce another prestack $\calR^{\circ}_G$ which sends each $R\in \Nilp_{\bZ_p}$ to the groupoid of pairs $(\calE_{\bfA_R},f:t(\calE_{\bfA_R})\twoheadrightarrow \calF_{\bfA_R})$ where $\calE_{\bfA_R}$ is an \'etale $\varphi$-module of rank $d$, $\calF_{\bfA_R}$ is an \'etale $\varphi$-module of rank $t(d)-1$, and $f$ is a surjective map of \'etale $\varphi$-modules. So the morphism $\pi$ factors through the natural map $\pi^{\circ}:\calR^{\circ}_G\to \calR_{d}\times_{\Spf(\bZ_p)}\calR_{t(d)-1}$.

\begin{prop}\label{surj1}
 The morphism $\pi^{\circ}:\calR^{\circ}_G\to \calR_{d}\times_{\Spf(\bZ_p)}\calR_{t(d)-1}$ is representable in formal schemes which locally are completions of finite type affine schemes. In particular, $\calR^{\circ}_G$ is limit preserving.
\end{prop}

\begin{proof}
Consider any point $\Spec(R)\to \calR_{d}\times_{\Spf(\bZ_p)}\calR_{t(d)-1}$. Let $\calA_R$ be the following pullback of prestacks
\begin{equation}
    \xymatrix@=1cm{
    \calA_R\ar[d]\ar[r] & \Spec(R)\ar[d]\\
    \calR^{\circ}_G\ar[r]^-{\pi^{\circ}} & \calR_{d}\times_{\Spf(\bZ_p)}\calR_{t(d)-1}.
    }
\end{equation}

We need to prove $\calA_R$ is a formal scheme which locally is a completion of finite type affine schemes over $R$.

 As $\calR_{d}\times_{\Spf(\bZ_p)}\calR_{t(d)-1}$ is limit preserving, we can assume $R$ is of finite type over $\bZ_p$. The point $\Spec(R)\to \calR_{d}\times_{\Spf(\bZ_p)}\calR_{t(d)-1}$ gives us $\calE_{\bfA_R}$, which is an \'etale $\varphi$-module of rank $d$, and $\calF_{\bfA_R}$, which is an \'etale $\varphi$-module of rank $t(d)-1$. Then for any finite type $R$-algebra $S$, $\calA_R(S)$ is the setoid (by \cite[\href{https://stacks.math.columbia.edu/tag/05UI}{Tag 05UI}]{stacks-project}) of the surjective morphisms from $t(\calE_{\bfA_S})$ to $\calF_{\bfA_S}$, where $\calE_{\bfA_S}$ (resp. $\calF_{\bfA_S}$) is the base change of $\calE_{\bfA_R}$ (resp. $\calF_{\bfA_R}$) along $\bfA_R\to \bfA_S$. In particular, this means as a prestack over $R$, $\calA_R$ is equivalent to the functor $\underline\Hom_{\surj}^{\varphi}(t(\calE_{\bfA_R}),\calF_{\bfA_R})$ sending each $R$-algebra $S$ to the set of surjective morphisms from $t(\calE_{\bfA_S})$ to $\calF_{\bfA_S}$ as \'etale $\varphi$-modules. So it suffices to prove this surjective Hom functor is representable by a formal scheme  which locally is a completion of finite type affine schemes over $R$, which is provided by Lemma \ref{representability-phi}.

 \end{proof}

\begin{lem}\label{representability-phi}
    Let $\calE_1,\calE_2$ be finite projective two \'etale $\varphi$-modules over $\bfA_R$ where $R$ is of finite type over $\bZ/p^a$. Then the surjective Hom functor $\underline\Hom_{\surj}^{\varphi}(\calE_1,\calE_2)$ is representable by a formal scheme  which locally is a completion of finite type affine schemes over $R$.
\end{lem}

Before we prove Lemma \ref{representability-phi}, let us first recall the notion of de Rham prestack. The reason why we care about de Rham prestack is that the surjectivity of any morphism of modules can be checked after base changing to the reduced coefficient ring.

\begin{lem}\label{surj-reduced}
    Let $A$ be commutative ring. Let $M,N$ be two finitely generated $A$-modules. Then a map $f:M\to N$ is surjective if and only if $f_{\rm red}:M\otimes_AA_{\rm red}\to N\otimes_AA_{\rm red}$ is surjective.
\end{lem}
\begin{proof}
    Let $K$ be the cokernel of $f$. Then that $f$ is surjective is equivalent to $K=0$. Note that $K=0$ is equivalent to $K\otimes_AA_{\frakm}=0$ for all maximal ideals $\frakm$ of $A$. By Nakayama's lemma, this is equivalent to $K\otimes_AA/\frakm=0$ for all maximal ideals of $A$. Now the lemma follows from $A/\frakm=A_{\rm red}/{\frakm}$. 
\end{proof}

\begin{dfn}[de Rham prestack]
    Let $X$ be a prestack. Define its de Rham prestack $X_{\rm dR}$ by $X_{\rm dR}(A):=X(A_{\rm red})$ for any commutative ring $A$.
\end{dfn}

We give some useful lemmas concerning the de Rham prestack functor.
\begin{lem}\label{limit preserving of dR space}
    If $X$ a limit preserving prestack over $\bZ_p$, then $X_{\dR}$ is also limit preserving.
\end{lem}
\begin{proof}
    Let $A$ be any commutative ring. Write $A=\colim A_i$ as the filtered colimit of all its finite type $\bZ_p$-subalgebras. There is a natural map $\colim A_{i,\rm red}\to A_{\rm red}$, which is obviously surjective. Note that for any inclusions $A_i\hookrightarrow A$, we also get injective maps $A_{i,\rm red}\hookrightarrow A_{\rm red}$. So the map $\colim A_{i,\rm red}\to A_{\rm red}$ is also injective. This imples $X_{\dR}(A)=X(A_{\rm red})=\colim X(A_{i,\rm red})=\colim X_{\rm dR}(A_i)$. We are done.
\end{proof}

\begin{lem}\label{dR-open-closed}
    Let $f:X\to Y$ be a morphism of limit preserving prestacks over $\bZ_p$. 
    \begin{enumerate}
        \item If $f$ is an open immersion, then $f_{\rm dR}:X_{\rm dR}\to Y_{\rm dR}$ is also an open immersion.
        \item If $f$ is a closed immerison, then $f_{\rm dR}:X_{\rm dR}\to Y_{\rm dR}$ is a formal completion of a closed immersion.
    \end{enumerate}
\end{lem}

\begin{proof}
Let $x:\Spec(C)\to Y_{\rm dR}$ be a $C$-point of $Y_{\rm dR}$. By Lemma \ref{limit preserving of dR space}, $Y_{\dR}$ is limit preserving. Then we may assume $C$ is of finite type over $\bZ_p$. The point $x$ corresponds to a point $\tilde x:\Spec(C_{\rm red})\to Y$. 
    \begin{enumerate}
        \item  If $f$ is an open immersion, then the fiber product $X_{\tilde x}=X\times_{Y,\tilde x}\Spec(C_{\rm red})$ is an open subscheme of $\Spec(C_{\rm red})$. In particular, $X_{\tilde x}$ induces an open subscheme $X_{x}$ of $\Spec(C)$ whose underlying reduced scheme is just $X_{\tilde x}$. 

        Let $\Spec(D)\to \Spec(C)$ be a map whose reduced map $\Spec(D_{\rm red})\to \Spec(C_{\rm red})$ factors through $X_{\tilde x}\to \Spec(C_{\rm red})$. Then $\Spec(D)\to \Spec(C)$ must uniquely factor through $X_{x}\to \Spec(C)$. This implies that $X_{ x}$ is just the fiber product $X_{\dR}\times_{Y_{\dR},x}\Spec(C)$.

        \item If $f$ is a closed immersion, then the fiber product $X_{\tilde x}=X\times_{Y,\tilde x}\Spec(C_{\rm red})$ is a closed subscheme of $\Spec(C_{\rm red})$. Write $X_{\tilde x}=\Spec(C_{\rm red}/I)$. Let $\Spec(D)$ be of finite type over $\bZ_p$ and $\Spec(D)\to \Spec(C)$ be a map whose reduced map $\Spec(D_{\rm red})\to \Spec(C_{\rm red})$ factors through $X_{\tilde x}\to \Spec(C_{\rm red})$, i.e. we have the following commutative diagram
        \begin{equation}
            \xymatrix@=1cm{
            C\ar[d]\ar[rr] &  & D\ar[d]\\
            C_{\rm red}\ar[r] & C_{\rm red}/I\ar[r] & D_{\rm red}.
            }
        \end{equation}
Let $\tilde I$ be the kernel of $C\to C_{\rm red}/I$. Then $\tilde I$ is sent to the nilradical ${\rm rad}(D)$ of $D$ under the map $C\to D$. As $D$ is of finite type over $\bZ_p$, there exists some integer $n\geq 1$ such that ${\rm rad}(D)^n=0$, which implies the map $C\to D$ factors through $C/{\tilde I}^n\to D$. So we see that the fiber product $X_{\dR}\times_{Y_{\dR},x}\Spec(C)$ is the formal scheme $\Spf(C^{\wedge}_{\tilde I})$.
    \end{enumerate}
\end{proof}

Now in order to prove Lemma \ref{representability-phi}, we want to reduce to the basic case. Recall the following definition.

\begin{dfn}[{\cite[Definition 3.2.3]{EG22}}]
    If $K$ is a finite extension of $\bQ_p$, we set $K^{\rm basic}:=K\cap K_0(\zeta_{p^{\infty}})$, where $K_0$ is the maximal absolutely unramified subextension of $K$.
\end{dfn}
Then we have coefficient rings $\bfA_K$ and $\bfA_{K^{\rm basic}}$. By the paragraph just below \cite[Definition 3.2.3]{EG22}, we know that $\bfA_K$ is a finite free $\bfA_{K^{\rm basic}}$ of rank $[K,K^{\rm basic}]$ and the inclusion $\bfA_{K^{\rm basic}}\subset \bfA_K$ is $\varphi$-equivariant.

For any $R$-algebra $S$, we write $\bfA_{K^{\rm basic},S}$ and $\bfA_{K,S}$ as the corresponding coefficient rings to distinguish  $K$ from $K^{\rm basic}$. By the proof of \cite[Lemma 3.2.5]{EG22}, we know that for a finite projective \'etale $\varphi$-module $M_S$ over $\bfA_{K^{\rm basic},S}$, the additional structure required to make $M_S$ into a finite projective \'etale $\varphi$-module over $\bfA_{K,S}$ is the data of a morphism of \'etale $\varphi$-modules over $\bfA_{K^{\rm basic},S}$
\[
f:\bfA_{K,S}\otimes_{\bfA_{K^{\rm basic},S}}M_S\to M_S
\]
which satisfies the following conditions
\begin{enumerate}
    \item the composite $M_S\xrightarrow{i}\bfA_{K,S}\otimes_{\bfA_{K^{\rm basic},S}}M_S\xrightarrow{f} M_S$ is the identity morphism;
    \item the kernel of $f$ is $\bfA_{K,S}$-stable.
\end{enumerate}
Then the $\bfA_{K,S}$-structure on $M_S$ is given by $a\cdot m:=f(a\otimes m)$. Using this description, we can obtain the following lemma.

\begin{lem}\label{reduction-basic}
     Let $\calE_1,\calE_2$ be finite projective two \'etale $\varphi$-modules over $\bfA_{K,R}$ where $R$ is of finite type over $\bZ/p^a$. Let $\calE_1^{\circ},\calE_2^{\circ}$ be their underlying \'etale $\varphi$-modules over $\bfA_{K^{\rm basic,R}}$ respectively. Then the natural morphism $\underline\Hom^{\varphi}(\calE_1,\calE_2)\to \underline\Hom^{\varphi}(\calE_1^{\circ},\calE_2^{\circ})$ is a closed immersion of affine schemes of finite presentation over $R$ and so is $\underline\Isom^{\varphi}(\calE_1,\calE_2)\to \underline\Isom^{\varphi}(\calE_1^{\circ},\calE_2^{\circ})$.
\end{lem}

\begin{proof}
    Let $f_i:\bfA_{K,S}\otimes_{\bfA_{K^{\rm basic},S}}\calE_{i}^{\circ}\to \calE_{i}^{\circ}$ be the map corresponding to the \'etale $\varphi$-structure over $\bfA_{K,R}$ of $\calE_i$ for $i=1,2$. Let $S$ be an $R$-algebra and $g:\calE_{1,S}^{\circ}\to \calE^{\circ}_{2,S}$ be an element in $\underline\Hom^{\varphi}(\calE_1^{\circ},\calE_2^{\circ})(S)$. Then $g$ is an $\bfA_{K,S}$-linear map if and only if for any $\lambda\in \bfA_{K,S}$ and $m\in \calE^{\circ}_{1,S}$, we have $g(f_1(\lambda\otimes m))=f_2(\lambda\otimes g(m))$. By \cite[Lemma 5.4.8]{EG21}, we know that $\underline\Hom^{\varphi}(\calE_1^{\circ},\calE_2^{\circ})$ is an affine scheme of finite presentation over $R$. Also following the argument in the proof of \cite[Lemma 5.4.8]{EG21}, we may assume both $\calE_1$ and $\calE_2$ are finite free. In this case, $g, f_1, f_2$ are matrices with entries in $\bfA_{K^{\rm basic},S}=(W(k_{\infty}^{\rm basic})\otimes_{\bZ_p}S)((T))$. Write $g=\sum_{i\geq s}^{\infty}g_iT^i$ with each $g_i$ being a matrix with entries in $W(k_{\infty}^{\rm basic})\otimes_{\bZ_p}S$. \cite[Lemma 5.4.8]{EG21} shows that $g$ is controlled by all $g_i$'s with $i\leq t$ and $t$ does not depend on $S$, which gives the finite presentation of $\underline{\Hom}^{\varphi}(\calE_1^{\circ},\calE_2^{\circ})$. Then $g(f_1(\lambda\otimes m))=f_2(\lambda\otimes g(m))$ is a system of equations with finitely many variables. Hence $\underline\Hom^{\varphi}(\calE_1,\calE_2)$ is the closed subscheme of $\underline\Hom^{\varphi}(\calE_1^{\circ},\calE_2^{\circ})$ cut out by the above equations. So we are done. The statement about the isomorphism functors then follows from the following Cartesian diagram
    \begin{equation}
    \xymatrix{
    \underline\Isom^{\varphi}(\calE_1,\calE_2)\ar[r]\ar[d]& \underline\Isom^{\varphi}(\calE_1^{\circ},\calE_2^{\circ})\ar[d]\\
    \underline\Hom^{\varphi}(\calE_1,\calE_2)\ar[r]& \underline\Hom^{\varphi}(\calE_1^{\circ},\calE_2^{\circ}).
    }
\end{equation}

\end{proof}

Now we are ready to prove Lemma \ref{representability-phi}.

\begin{proof}[Proof of Lemma \ref{representability-phi}]
Note that we have a Cartesian diagram
\begin{equation}
    \xymatrix{
    \underline\Hom_{\surj}^{\varphi}(\calE_1,\calE_2)\ar[r]\ar[d]& \underline\Hom_{\surj}^{\varphi}(\calE_1^{\circ},\calE_2^{\circ})\ar[d]\\
    \underline\Hom^{\varphi}(\calE_1,\calE_2)\ar[r]& \underline\Hom^{\varphi}(\calE_1^{\circ},\calE_2^{\circ}).
    }
\end{equation}
Then by Lemma \ref{reduction-basic}, we may assume $\bfA_R^{+}$ is $\varphi$-stable so that we can use all results concerning \'etale $\varphi$-modules in \cite{EG21}.

Recall that $R$ is of finite type over $\bZ/p^a$. We first assume $\calE_1$ and $\calE_2$ are finite free of rank $r,t$ respectively and $r\geq t$. Choosing bases of $\calE_1,\calE_2$, an element $g$ in $\Hom(\calE_1,\calE_2)$ is given by a matrix in $M_{r\times t}(\bfA_R)$. Write $g=\sum_{i\geq -s}g_iT^i$ with $g_i\in M_{r\times t}(W(k_{\infty})\otimes_{\bZ_p}R)$. For simplicity, we let $\tilde R:=W(k_{\infty})\otimes_{\bZ_p}R$. Let $X$ and $Y$ be the matrices of $\varphi_{\calE_1},\varphi_{\calE_2}$ with respect to the chosen bases respectively and $n\geq 0$ be an integer such that $X,Y,X^{-1},Y^{-1}$ all have entries with poles of degree at most $n$. Then by the proof of \cite[Proposition 5.4.8]{EG21}, we must have $s<(2n+ap)/(p-1)$ and all $g_i$ are determined by the $g_i$ for $i\leq (2n+(a-1)p)/(p-1)$.

Let $s_1=\lfloor {(2n+ap)/(p-1)}\rfloor$ and $s_2=\lfloor (2n+(a-1)p)/(p-1)\rfloor$. Then there is a natural morphism
\[
\iota:\underline\Hom^{\varphi}(\calE_1,\calE_2)\to \prod_{i\in [-s_1,s_2]} (\Res_{\tilde R/R}(\bA^{r\times t}))_{\dR}
\]
by sending each $g=\sum_{i\geq -s}g_iT^i\in \underline\Hom^{\varphi}(\calE_1,\calE_2)(S)$ to $\prod_{i\in [-s_1,s_2]}\bar g_i$, where $\bar g_i$ is the image of $g_i$ in $(W(k_{\infty})\otimes_{\bZ_p}S)_{\rm red}=k_{\infty}\otimes_{\bF_p}S_{\rm red}$ for any $R$-alegbra $S$. The reason why we consider this map is the following. If $g$ represents a surjective map, then by Lemma \ref{surj-reduced}, this is equivalent to that $\bar g\in M^{m\times n}((W(k_{\infty})\otimes_{\bZ_p}S)_{\rm red}((T)))$ represents a surjective map. Moreover that $\bar g$ represents a surjective map is equivalent to that the ideal generated by all the $(t\times t)$-minors of $\bar g=\sum_{i\geq -s}\bar g_iT^i$ is the unit ideal. As $(W(k_{\infty})\otimes_{\bZ_p}S)_{\rm red}$ is reduced, this is equivalent to that the leading coefficient $\bar g_i$ of $\bar g$ represents a surjective map from $(W(k_{\infty})\otimes_{\bZ_p}S)_{\rm red}^r$ to $(W(k_{\infty})\otimes_{\bZ_p}S)_{\rm red}^t$. So we can characterise surjective maps via their images under the morphism $\iota$.

By \cite[\href{https://stacks.math.columbia.edu/tag/09TP}{Tag 09TP}]{stacks-project}, the surjective Hom functor $\frakF:=\underline\Hom_{\surj}(\tilde R^r,\tilde R^t))$ is represented by an open subscheme of $\underline\Hom(\tilde R^r,\tilde R^t))=\bA_{\tilde R}^{r\times t}$. There is also the closed immersion $\Spec\tilde R\to \bA_{\tilde R}^{r\times t}$ corresponding to the trivial map from $\tilde R^r$ to $\tilde R^t$. By \cite[Proposition 7.6/2]{BLR12}, the map $\Res_{\tilde R/R}(\frakF)\to \Res_{\tilde R/R}(\bA_{\tilde R}^{r\times t})$ is still an open immersion and the map $\Res_{\tilde R/R}(\Spec\tilde R)\to \Res_{\tilde R/R}(\bA_{\tilde R}^{r\times t})$ is still a closed immersion.

For each $i\in [-s_1,s_2]$, we can define the following pullback diagram
\begin{equation}
    \xymatrix@=1cm{
   \frakA_i\ar[d]\ar[r] & (\prod_{j\in[-s_1,i-1]}(\Res_{\tilde R/R}(\Spec\tilde R))_{\dR})\times (\Res_{\tilde R/R}(\frakF))_{\dR}\ar[d] \\
    \prod_{j\in [-s_1,s_2]} (\Res_{\tilde R/R}(\bA_{\tilde R}^{m\times n}))_{\dR}\ar[r]^{{}} & \prod_{j\in [-s_1,i]} (\Res_{\tilde R/R}(\bA_{\tilde R}^{m\times n}))_{\dR}
    }
\end{equation}
where the bottom horizontal map is the projection map and the right vertical map is induced by applying the de Rham prestack functor to the product of the closed immersion and open immersion mentioned above. By Lemma \ref{dR-open-closed}, the left vertical map is representable in formal schemes which locally are completions of finite type affine schemes. As the trivial morphism is not surjective, we have $\bigcup_{i\in [-s_1,s_2]} \frakA_i=\coprod_{i\in [-s_1,s_2]} \frakA_i$.

Now by the above discussion on the leading coefficients of $\bar g$, we can see that $\underline\Hom_{\surj}^{\varphi}(\calE_1,\calE_2)$ is exactly the pullback of the following diagram
\begin{equation}
    \xymatrix@=1cm{
    \underline\Hom_{\surj}^{\varphi}(\calE_1,\calE_2)\ar[d]\ar[r] &\coprod_{i\in[-s_1,s_2]}\frakA_i\ar[d]\\
    \underline\Hom^{\varphi}(\calE_1,\calE_2)\ar[r] & \prod_{j\in [-s_1,s_2]} (\Res_{\tilde R/R}(\bA_{\tilde R}^{m\times n}))_{\dR}.
    }
\end{equation}
This implies the natural map $\underline\Hom_{\surj}^{\varphi}(\calE_1,\calE_2)\to  \underline\Hom^{\varphi}(\calE_1,\calE_2)$ is representable in formal schemes which locally are completions of finite type affine schemes. So we are done with the case of free modules.

For general $\calE_1,\calE_2$, we can always find $\calF_1,\calF_2$ such that $\calF_1\oplus \calE_1\oplus \calF_2$ and $\calE_2\oplus \calF_2$ are both finite free. Let $e_1$ be the idempotent attached to $\calF_1$ in $\calF_1\oplus \calE_1\oplus \calF_2$, $e_2$ be the idempotent attached to $\calE_1$ in $\calF_1\oplus \calE_1\oplus \calF_2$, $e_3$ be the idempotent attached to $\calF_2$ in $\calF_1\oplus \calE_1\oplus \calF_2$ and $e_4$ be the idempotent attached to $\calF_2$ in $\calE_2\oplus \calF_2$, $e_5$ be the idempotent attached to $\calE_2$ in $\calE_2\oplus \calF_2$. We first consider the closed subspace $\calY$ in $\underline\Hom_{\surj}^{\varphi}(\calF_1\oplus \calE_1\oplus \calF_2,\calE_2\oplus \calF_2)$ cut out by the equations $ge_1=0$, $e_4ge_2=0$ and $(1-e_4)ge_3=0$. Then there is a natural map $\calY\to \underline\Hom^{\varphi}(\calF_2,\calF_2)$. Let $i:\Spec(R)\to \underline\Hom^{\varphi}(\calF_2,\calF_2)$ be the morphism corresponding to the identity map. Then $i$ is a closed immersion. To see this, note that $\underline\Hom^{\varphi}(\calF_2,\calF_2)$ is the closed subscheme of $\underline\Hom^{\varphi}(\calE_2\oplus \calF_2,\calE_2\oplus\calF_2)$ cut out by the equations $ge_5=0, e_5g=0$. And the point corresponding to the identity map is the closed subspace of $\underline\Hom^{\varphi}(\calF_2,\calF_2)$ cut out further by the equation $g=e_4$. So $i:\Spec(R)\to \underline\Hom^{\varphi}(\calF_2,\calF_2)$ is a closed immersion. Now we see that $\underline\Hom_{\surj}^{\varphi}(\calE_1,\calE_2)$ is the the following pullback
\begin{equation}
    \xymatrix{
    \underline\Hom_{\surj}^{\varphi}(\calE_1,\calE_2)\ar[r]\ar[d]& \Spec(R)\ar[d]^{i}\\
    \calY\ar[r]& \underline\Hom^{\varphi}(\calF_2,\calF_2).\\
    }
\end{equation}
So $\underline\Hom_{\surj}^{\varphi}(\calE_1,\calE_2)$ is a closed subscheme of $\underline\Hom_{\surj}^{\varphi}(\calF_1\oplus \calE_1\oplus \calF_2,\calE_2\oplus \calF_2)$. We are done.
\end{proof}


Similarly, we can also define a prestack $\calX_G^{\circ}$ which sends each $R\in \Nilp_{\bZ_p}$ to the groupoid of pairs $(\calE_{\bfA_R},f:t(\calE_{\bfA_R})\twoheadrightarrow \calF_{\bfA_R})$ where $\calE_{\bfA_R}$ is an \'etale $(\varphi,\Gamma)$-module of rank $d$, $\calF_{\bfA_R}$ is an \'etale $(\varphi,\Gamma)$-module of rank $t(d)-1$, and $f$ is a surjective map of \'etale $(\varphi,\Gamma)$-modules. Then we have the factorisation $\calX_G\to \calX_G^{\circ}\to \calX_d\times \calX_{t(d)-1}$.

\begin{lem}\label{surj2}
    The morphism $\calX_G^{\circ}\to \calX_d\times \calX_{t(d)-1}$ is representable in formal schemes which locally are completions  of finite type affine schemes.
\end{lem}

\begin{proof}
Let $\calE_{\bfA_R}\in\calX_d(R)$ and $\calF_{\bfA_{R}}\in \calX_{t(d)-1}$. As $\calX_d\times \calX_{t(d)-1}$ is limit preserving, we may assume $R$ is of finite type over $\bZ_p$.

By Lemma \ref{representability-phi}, we know that the surjective Hom functor $\underline\Hom_{\surj}^{\varphi}(t(\calE_{\bfA_R}),\calF_{\bfA_{R}})$ is representable by a formal scheme which locally is a completion of finite type affine scheme over $R$. Note that the surjective Hom functor $\underline\Hom_{\surj}^{\varphi,\Gamma_{\disc}}(t(\calE_{\bfA_R}),\calF_{\bfA_R})$ (as \'etale $(\varphi,\Gamma_{\disc})$-modules) can be defined as the pullback diagram
 \begin{equation}
     \xymatrix@=1cm{
     \underline\Hom_{\surj}^{\varphi,\Gamma_{\disc}}(t(\calE_{\bfA_S}),\calF_{\bfA_R})\ar[d]\ar[r]& \underline\Hom_{\surj}^{\varphi}(t(\calE_{\bfA_R}),\calF_{\bfA_{R}})\ar[d]^{(\gamma^*,id)}\\
     \underline\Hom_{\surj}^{\varphi}(t(\calE_{\bfA_R}),\calF_{\bfA_{R}})\ar[r]^-{\Delta} &\underline\Hom_{\surj}^{\varphi}(t(\calE_{\bfA_R}),\calF_{\bfA_{R}})\times \underline\Hom_{\surj}^{\varphi}(t(\calE_{\bfA_R}),\calF_{\bfA_{R}})
     }
 \end{equation}
 where $\gamma^*$ is the functor sending $f$ to $\gamma^{-1}\circ f\circ \gamma $ and $\Delta$ is the diagonal functor. Then we deduce that $\underline\Hom_{\surj}^{\varphi,\Gamma_{\disc}}(t(\calE_{\bfA_R}),\calF_{\bfA_R})$ is also representable by a formal scheme which locally is a completion of finite type affine scheme over $R$ by \cite[\href{https://stacks.math.columbia.edu/tag/0AIP}{Tag 0AIP}]{stacks-project}.
 
 Note that any surjective map $t(\calE_{\bfA_R})\to \calF_{\bfA_R}$ of \'etale $(\varphi,\Gamma_{\disc})$-modules is necessarily continuous with respect to the natural topology and hence must also be a surjective map of \'etale $(\varphi,\Gamma)$-modules. So the surjective Hom functor $\underline\Hom_{\surj}^{\varphi,\Gamma}(t(\calE_{\bfA_R}),\calF_{\bfA_R})$ of \'etale $(\varphi,\Gamma)$-modules is the same as $\underline\Hom_{\surj}^{\varphi,\Gamma_{\disc}}(t(\calE_{\bfA_R}),\calF_{\bfA_R})$. This means $\underline\Hom_{\surj}^{\varphi,\Gamma}(t(\calE_{\bfA_R}),\calF_{\bfA_R})$ is also representable by a formal scheme which locally is a completion of finite type affine scheme over $R$.

\end{proof}

As this point, we can also finish the proof of Lemma \ref{Lemma-Gamma}.
\begin{proof}[Proof of Lemma \ref{Lemma-Gamma}]
    As $\calX_G^{\circ}$ is limit preserving by Lemma \ref{surj2}, we know $\calX_G$ and the stack of \'etale $(\varphi,\Gamma)$-twists are both limit preserving by \cite[Lemma D.28]{EG22} and Lemma \ref{Gamma_disc}. Then this lemma follows from the case of finite type $\bZ_p$-algebras, which we have already proved.
\end{proof}

\subsection{Artin's criteria}
Now in order to study the morphism $\calX_G\to \calX_d\times \calX_{t(d)-1}$, it remains to study the morphism $\calX_G\to \calX_G^{\circ}$ thanks to Lemma \ref{surj2}. Note that $\calX_G\to \calX_G^{\circ}$ is a monomorphism as the existence of an fppf cover such that the pullback of a pair $(\calE_{\bfA_R},\calF_{\bfA_R})$ is isomorphic to $(V_{\bfA_R},V_{\bfA_R}/L_{\bfA_R})$ is a property instead of part of the data.  We will resort to Artin's axioms for the representability of functors by
algebraic spaces. 

\begin{lem}[{\cite[\href{https://stacks.math.columbia.edu/tag/07Y2}{Tag 07Y2}]{stacks-project}}]\label{artin}
    Let $\calF\to \calG$ be a transformation of functors $\alpha:(Sch/\Spec(\bZ_p))^{\rm opp}_{\rm fppf}\to Sets$. Assume that
    \begin{enumerate}
        \item $\alpha$ is injective,
        \item $\calF$ satisfies the following conditions
        \begin{enumerate}
            \item $\calF$ is a sheaf for the \'etale topology,
            \item $\calF$ is limit preserving,
            \item $\calF$ satisfies the Rim-Schlessinger condition,
            \item Every formal object is effective,
            \item $\calF$ satisfies openness of versality.
        \end{enumerate}
        \item $\calG$ is an algebraic space locally of finite presentation over $\bZ_p$.
    \end{enumerate}
    Then $\calF$ is an algebraic space locally of finite presentation over $\bZ_p$.
\end{lem}

We first discuss the case of \'etale $\varphi$-modules. For any map $\Spec(R)\to \calR_G^{\circ}$, let $\calB_R$ denote the pullback of the diagram $\calR_G\to \calR_G^{\circ}\xleftarrow{} \Spec(R)$. In particular, $\calB_R\to \Spec(R)$ is a monomorphism. We then have the following proposition.
\begin{prop}\label{etale}
   $\calB_R$ satisfies the conditions $(a),(b),(c)$ in Lemma \ref{artin}.
\end{prop}
\begin{proof}
 As $\calR^{\circ}_G$ is locally of finite presentation over $\Spf(\bZ_p)$, we may assume $R$ is also of finite type.  Let $(\calE_{\bfA_R}, t(\calE_{\bfA_R})\twoheadrightarrow \calF_{\bfA_R})$ be the data corresponding to $\Spec(R)\to \calR^{\circ}_G$.

\begin{enumerate}

  \item Condition (a): $\calB_R$ is an \'etale sheaf over $R$. As $\calB_R(A)$ for any $A$ is either empty set or a singleton set, we see $\calB_R(A_1\times A_2)=\calB_R(A_1)\times \calB_R(A_2)$ as $\bfA_{A_1\times A_2}\cong \bfA_{A_1}\times \bfA_{A_2}$. Now let $A_1\to A_2$ be a faithfully flat map of $R$-algebras. Then we need to prove $\calB_R(A_1)$ is the equalizer ${\rm Eq}(\calB_R(A_2)\rightrightarrows \calB_R(A_2\otimes_{A_1}A_2))$. It then suffices to prove when $\calB_R(A_2)=\{*\}$, we have $\calB_R(A_1)=\{*\}$. In other words, given an exact $\otimes$-functor $F_2: \Rep(G)\to \Mod^{\varphi}(\bfA_{A_2})$, we need to show $F_2$ descends to an exact $\otimes$-functor $F_1:\Rep(G)\to \Mod^{\varphi}(\bfA_{A_1})$. For each $V\in \Rep(G)$, we can define $F_1(V)$ to be the \'etale $\varphi$-module over $\bfA_{A_1}$ obtained by the faithfully flat descent of $F_2(V)$, which is guaranteed by that \'etale $\varphi$-modules satisfy faithfully flat descent. For any map $f:V_1\to V_2$, as $\underline\Hom^{\varphi}(F_1(V_1),F_1(V_2))$ is an affine scheme over $A_1$ by Lemma \ref{reduction-basic}, which is an fpqc sheaf, we get a morphism $F_1(f):F_1(V_1)\to F_1(V_2)$ by descending the morphism $F_2(f):F_2(V_1)\to F_2(V_2)$. Then for any short exact sequence $0\to V_1\to V_2\to V_3\to 0$ in $\Rep(G)$, we get a sequence $0\to F_1(V_1)\to F_1(V_2)\to F_1(V_3)\to 0$. We first assume $A_1$ is of finite type over $\bZ_p$, then $\bfA_{A_1}\to \bfA_{A_2}$ is faithfully flat by \cite[Lemma 5.1.7]{EG21}. So we see $0\to F_1(V_1)\to F_1(V_2)\to F_1(V_3)\to 0$ is also exact. It is clear that $F_1$ is a $\otimes$-functor, so $\calB_R(A_1)=\{*\}$. By the condition (b) which will be proved below, i.e. $\calB_R$ is limit preserving, we conclude $\calB_R$ is an \'etale sheaf by \cite[Chapter 2, Proposition 2.2.7]{GR19}.
    
    \item Condition (b): $\calB_R$ is limit preserving. Let $A$ be an $R$-algebra. Write $A=\varinjlim_i A_i$ as the filtered colimit of finite type $R$-subalgebras. Then we need to prove $\calB_R(A)=\varinjlim_i \calB_R(A_i)$. Note that $\calB_R(A_i)$ is either empty or a singleton set. So it suffices to prove that if $\calB_R(A)=\{*\}$, then there exists an $A_i$ such that $\calB_R(A_i)=\{*\}$.  Let $F:\Rep(G)\to \Mod^{\varphi}(\bfA_A)$ be the exact $\otimes$-functor corresponding to $\calB_R(A)$. Let $\bar A=\varinjlim_i \bfA_{A_i}$. For each $V\in \Rep(G)$, we have $F(V)\in \Mod^{\varphi}(\bfA_A)$. Then there exists $A_i$ such that $F(V)\in \Mod^{\varphi}(\bfA_{A_i})$ as the moduli stack of \'etale $\varphi$-modules is limit preserving. Now for any $h:V_1\to V_2\in \Rep(G)$, we first can choose a $j$ such that $F(V_1)$ and $F(V_2)$ are actually in $\Mod^{\varphi}(\bfA_{A_j})$. Now as the affine scheme $\underline\Hom^{\varphi}(F(V_1),F(V_2))$  by Lemma \ref{reduction-basic} is of finite presentation over $A_j$, we see that $\underline\Hom^{\varphi}(F(V_1),F(V_2))(A)=\varinjlim_i \underline\Hom^{\varphi}(F(V_1),F(V_2))(A_i)$. So there exists a $k$ such that $F(h)$ actually lies in $\underline\Hom^{\varphi}(F(V_1),F(V_2))(A_k)$. In particular, this mean we get a functor $\bar F: \Rep(G)\to \Mod^{\varphi}(\bar A)$ such that after base change along $\bar A\to \bfA_A$, we get the original $F:\Rep(G)\to \Mod^{\varphi}(\bfA_A)$. 

    It is clear that $\bar F$ is a $\otimes$-functor. Let us now prove it is exact. Let $0\to V_1\to V_2\to V_3\to 0$ be a short exact sequence in $\Rep(G)$. Then we have a short exact sequence $0\to F(V_1)\to F(V_2)\to F(V_3)\to 0$. As $\bar A\to \bfA_A$ is injective, we see that $\bar F(V_i)\to F(V_i)$ is also injective for $i=1,2,3$. So $\bar F(V_1)\to \bar F(V_2)$ is injective and the composite $\bar F(V_1)\to \bar F(V_2)\to \bar F(V_3)$ is $0$. Now by Lemma \ref{representability-phi}, we see the surjective Hom functor is limit preserving. So $\bar F(V_2)\to \bar F(V_3)$ is also surjective. Let $K$ be the kernel of $\bar F(V_2)\to \bar F(V_3)$. As the isomorphism Hom functor of \'etale $\varphi$-modules is also limit preserving by Lemma \ref{reduction-basic}, we see the natural map $\bar F(V_1)\to K$ is also an isomorphism. So we have a short exact sequence $0\to \bar F(V_1)\to \bar F(V_2)\to \bar F(V_3)\to 0$.

    This means we have an exact $\otimes$-functor $\bar F: \Rep(G)\to \Mod^{\varphi}(\bar A)$. As the moduli stack of $G$-torsors is locally of finite presentation, we then get a functor $\calF: \Rep(G)\to \Mod(\bfA_{A_j})$ for some $j$, whose base change to $\bar A$ coincides with the underlying $G$-torsor of $\bar F$. In particular, this means there exists a fppf cover $\calY_j\to \Spec(\bfA_{A_j})$ such that $(\calF(V)_{\calY_j},\calF(L)_{\calY_j})$
    is isomorphic to $(V_{Y_j},L_{Y_j})$. Now we can choose an $s$ such that $F(V),F(L)$ lie in $\Mod^{\varphi}(\bfA_{A_s})$ and there is a map from $\bfA_{A_j}\to \bfA_{A_s}$. Then we see $\calB_R(A_s)=\{*\}$.

    \item Condition (c): $\calB_R$ satisfies the Rim-Schlessinger condition, i.e. for any pullback diagram of local Artinian rings of finite type over $R$
    \begin{equation*}
        \xymatrix{
        A\ar@{->>}[r]\ar[d]& B\ar[d]\\
        C\ar@{->>}[r]& D,
        }
    \end{equation*}
    the functor $\calB_R(A)\to \calB_R(C)\times_{\calB_R(D)}\calB_R(B)$ is an equivalence. To verify this, it suffices to assume $C\to D$ is surjective with square zero kernel. Then we have a pullback diagram of  rings 
    \begin{equation*}
        \xymatrix{
        \bfA_A\ar@{->>}[r]\ar[d]& \bfA_B\ar[d]\\
        \bfA_C\ar@{->>}[r]& \bfA_D
        }
    \end{equation*}
    such that the square of the kernel of $\bfA_C\to \bfA_D$ is $0$. Now we can use the facts that the moduli stack of $G$-torsors is an algebraic stack and that an algebraic stack satisfies the strong Rim-Schlessinger condition (cf. \cite[\href{https://stacks.math.columbia.edu/tag/0CXP}{Tag 0CXP}]{stacks-project}). This implies $\calB_R$ satisfies the Rim-Schlessinger condition.

\end{enumerate}

\end{proof}

The remaining conditions $(d),(e)$ are more delicate. The rest of this subsection is devoted to proving condition $(d)$ for $\calB_R$. The condition $(e)$ will be dealt with in the next subsection.

We begin with a lemma concerning algebraic representations of $G$.

\begin{lem}\label{Lem-rep}
  Let $W$ be a finite free $\bZ_p$-representation of $G$. Then there exists finite free $\bZ_p$-representations $W_1,W_2$ such that 
  \begin{enumerate}
      \item $W$ is a subrepresentation of $W_1$ and $W_1/W$ is finite free over $\bZ_p$;
      \item  $W_1$ is a quotient representation of $W_2=(\bigotimes^{s'}(\Sym^d(V^{\oplus d}))^{\vee})\bigotimes (\bigoplus_{i=0}^s\Sym^i(V^{\oplus d}))$ for some integers $s,s'\geq 0$, where $V$ is our fixed $\bZ_p$-representation and $d=\dim_{\bZ_p}V$.
      
  \end{enumerate}
\end{lem}
\begin{proof}
    The proof is basically the same as \cite[Theorem 4.12]{milne2017algebraic}.  We sketch the proof.
Write $\dim_{\bZ_p}(W)=m$. Let $(W,\rho)$ be the corresponding $\calO(G)$-comodule and $(\calO(G),\Delta)$ be the Hopf algebra. Then 
    $W\xrightarrow{
    \rho} W_0\otimes\calO(G)=\calO(G)^m$ is an injective homomorphism of $\calO(G)$-comodules, where $W_0$ is the underlying space of $W$ with trivial group action. In particular, $W$ is a direct summand of $W_0\otimes\calO(G)$ as $\bZ_p$-modules as $(id\otimes \epsilon)\circ\rho=id$ where $\epsilon: \calO(G)\to \bZ_p$ is the co-identity map. Let $W^j$ be the image of $W$ under the $j$-th projection map $\calO(G)^m\to \calO(G)$. Then $W\hookrightarrow \bigoplus_jW^j$. So we may assume $W\subset \calO(G)$.

    We choose a basis of $V$ and identity $G$ as a subgroup of $GL_d$. Then there is a surjective homomorphism
    \[
    \calO(GL_d)=\bZ_p[T_{ij},1/{\det}]_{1\leq i,j\leq }\twoheadrightarrow \calO(G)=\bZ_p[t_{ij},1/\det]_{1\leq i,j\leq d}
    \]
    sending $T_{ij}$ to $t_{ij}$. As $W$ is finite dimensional, it is contained in a subspace 
    \[
  W_1:= \{f(t_{ij})\mid \deg(f)\leq s\}\cdot {\rm det}^{-s'}
    \]
    of $\calO(G)$ for some $s,s'\in \bN$. 

   As $W$ is a direct summand of $\calO(G)$ as $\bZ_p$-modules, it is also true that $W$ is a direct summand of $W_1$ as $\bZ_p$-modules.  Note that $W_1$ is the image of 
   \[N:= \{f(T_{ij})\mid \deg(f)\leq s\}\cdot {\rm det}^{-s'}\subset \calO(GL_d).\]

   The space $\{f(T_{ij})\mid \deg(f)\leq s\}$ is exactly the direct sum $\bigoplus_{i=0}^s\Sym^i(V^{\oplus d})$. The determinant representation $\rm det$ is contained in $\Sym^d(V^{\oplus d})$, which implies $(\rm det)^{-1}$ is a quotient of $(\Sym^d(V^{\oplus d}))^{\vee}$. So $N$ is a quotient of $W_2=(\bigotimes^{s'}(\Sym^d(V^{\oplus d}))^{\vee})\bigotimes (\bigoplus_{i=0}^s\Sym^i(V^{\oplus d}))$, which implies that $W_1$ is also a quotient of $W_2$. We are done.

\end{proof}
Let $(S,\xi_n,f_n)$ be a formal object over $R$, i.e. $\xi_n:\Rep(G)\to \Mod^{\varphi}(\bfA_{S_n})$ is a compatible system of exact $\otimes$-functors, where $S$ is a Noetherian complete local $\bZ_p$-algebra with maximal ideal $\frakm$ and $S_n=S/\frakm^n$. We want to prove all these functors come from an exact $\otimes$-functor $\Rep(G)\to \Mod^{\varphi}(\bfA_{S})$.

Recall that $(\calE_{\bfA_R}, t(\calE_{\bfA_R})\twoheadrightarrow \calF_{\bfA_R})$ is the data corresponding to $\Spec(R)\to \calR^{\circ}_G$. Let $(\calE_{\bfA_S}, t(\calE_{\bfA_S})\twoheadrightarrow \calF_{\bfA_S})$ be the point $\Spec(S)\to \Spec(R)\to \calR^{\circ}_G$. Write $\calL_{\bfA_S}=\ker(t(\calE_{\bfA_S})\twoheadrightarrow \calF_{\bfA_S})$.  Then $P:=\underline{\Isom}((V_{\bfA_S},L_{\bfA_S}),(\calE_{\bfA_S},\calL_{\bfA_S}))$ is an affine scheme of finite presentation over $\bfA_S$ by Diagram \ref{Diagram-5} and we have an isomorphism $\varphi^*_{\bfA_S}P\simeq P$. Moreover, the formal object $(S,\xi_n,f_n)$ implies  that $P_n:=P\times_{\Spec(\bfA_S)}\Spec(\bfA_{S_n})$ is a $G$-torsor over $\bfA_{S_n}$ for each $n$. 

Now consider $P^1:=P\times_{\bfA_S}P$, which is just $\underline{\Isom}((V_P,L_P),(\calE_P,\calL_P))$. There is a natural isomorphism $(V_P,L_P)\simeq (\calE_P,\calL_P)$ corresponding to the identity map $id: P\to P$. In particular, this means $(\calE_P,\calL_P)$ corresponds to the trivial $G$-torsor over $P$, i.e. $P^1\simeq P\times_{\bZ_p}G$.

For each $W\in \Rep(G)$, we have a compatible system of finite projective \'etale $\varphi$-modules $\xi_n(W)$ over $\bfA_{S_n}$. We can consider the inverse limit $\varprojlim_n\xi_n(W)$, which is a finite projective \'etale $\varphi$-module over $\widehat {\bfA_S}:=\varprojlim_n \bfA_{S_n}$, which is just the $\frakm$-adic completion of $\bfA_S$. In particular, this induces an exact $\otimes$-functor $\hat \xi:\Rep(G)\to \Mod^{\varphi}(\widehat{\bfA_S})$.  So $\widehat P:=P\times_{\Spec(\bfA_S)}\Spec(\widehat\bfA_S)$ is a $G$-torsor over $\widehat \bfA_S$.

Note that there is a natural $G$-action on $P$ induced by the isomorphism $G\simeq \underline{\Isom}((V,L),(V,L))$. So $P$ is in fact a $G$-torsor over the quotient stack $[P/G]$. And the self-product $P\times_{[P/G]}P$ is the trivial $G$-torsor $P\times_{\Spec(\bZ_p)}G$ over $P$. Let $W\in \Rep(G)$. The construction $P\times^G(W\otimes_{\bZ_p}\calO_{[P/G]})$ also gives a vector bundle on $[P/G]$.

\begin{rmk}
   In fact, we can also write the $G$-torsor $P$ over $[P/G]$ as an isomorphism functors of $[P/G]$-twists. Namely, for any morphism $T\to [P/G]$, the corresponding  $G$-torsor $\widetilde T$ over $T$ is isomorphic to $\underline{\Isom}((V_T,L_T),(M_T,N_T))$ for some $T$-twist $(M_T,N_T)$. The $G$-torsor $\varphi^*\widehat T$ over $\widehat T$ then corresponds to $(\varphi^*M_T,\varphi^*N_T)$. The $T$-twist $(M_T,N_T)$ with \'etale $\varphi$-structure forms a compatible system when $T\to [P/G]$ varies. We then get an \'etale $\varphi$-$[P/G]$-twist $(M_{[P/G]},N_{[P/G]})$ over $[P/G]$ such that $P\simeq \underline{\Isom}((V_{[P/G]},L_{[P/G]}),(M_{[P/G]},N_{[P/G]}))$.
\end{rmk}

The vector bundle $P\times^G(W\otimes_{\bZ_p}\calO_{[P/G]})$ corresponds to a descent data for the cover $P\to [P/G]$. We can consider the induced equaliser of $\bfA_S$-modules
\begin{equation}\label{Diagram-xiW}
    \xymatrix{
    \xi(W)\ar[r]& \calO(P)\otimes_{\bZ_p}W\ar@<.5ex>[r]\ar@<-.5ex>[r]&\calO(P\times_{[P/G]}P)\otimes_{\bZ_p}W (\cong \calO(P)\otimes_{\bfA_S}\calO(P)\otimes_{\bZ_p} W).
    }
\end{equation}
All the equaliser diagrams we will see in the following are the base change of the above one.

 Recall that the \'etale $\varphi$-structure on $(\calE_{\bfA_S},\calL_{\bfA_S})$ induces an \'etale $\varphi$-structure on $\calO(P)$, i.e. there exists an isomorphism $\varphi_P:\varphi^*_{\bfA_S}\calO(P)\cong \calO(P)$. And the $G$-action induced by $(V,L)$ is compatible with the Frobenius induced by $(\calE_{\bfA_S},\calL_{\bfA_S})$. As $\varphi:\bfA_S\to \bfA_S$ is flat by \cite[Lemma 5.2.5]{EG21} and the paragraph above \cite[Lemma 2.2.16]{EG22}, there is an induced isomorphism $\varphi^*\xi(W)\cong \xi(W)$.

Since $\widehat P$ is a $G$-torsor over $\widehat \bfA_S$, the module $\widehat \xi(W)$ is just the equaliser given by the associated descent data by Theorem \ref{G-description}
\begin{equation}
    \xymatrix{
    \widehat\xi(W)\ar[r]& \calO(\widehat P)\otimes_{\bZ_p}W\ar@<.5ex>[r]\ar@<-.5ex>[r] &\calO(\widehat P)\otimes_{\widehat \bfA_S}\calO(\widehat P)\otimes_{\bZ_p} W.
    }
\end{equation}
As $\bfA_S$ is Noetherian, the map $\bfA_S\to \widehat \bfA_S$ is flat. So we see that $\xi(W)\otimes_{\bfA_S}\widehat\bfA_S\cong \widehat \xi(W)$.


Let $\xi_P:\Rep(G)\to \Mod^{\varphi}(\calO(P))$ (resp. $\xi_{\widehat P}:\Rep(G)\to \Mod^{\varphi}(\calO(\widehat P)$) be the exact $\otimes$-functor corresponding to the trivial $G$-torsor $P^1$ (resp. $\widehat P^1=\widehat P\times_{\widehat\bfA_S}\widehat P$) over $P$ (resp. $\widehat P$). Since $\varphi^*_{\bfA_S}\calO(P)\cong \calO(P)$ (resp. $\varphi^*_{\widehat \bfA_S}\calO(\widehat P)\cong \calO(\widehat P)$), for each $W\in \Rep(G)$, the module $\xi_P(W)$ (resp. $\xi_{\widehat P}(W)$) is also an \'etale $\varphi$-module over $\bfA_S$ (resp. $\widehat\bfA_S$), though not finitely generated over $\bfA_S$ (resp. $\widehat\bfA_S$). 

By the description in Theorem \ref{G-description}, we have the equaliser
\begin{equation}
    \xymatrix{
    \xi_P(W)\ar[r]& \calO(P^1)\otimes_{\bZ_p}W\ar@<.5ex>[r]\ar@<-.5ex>[r] &\calO(P^2)\otimes_{\bZ_p}W.
    }
\end{equation}

Considering the base change of the Equaliser diagram \ref{Diagram-xiW}  along $P\to \Spec(\bfA_S)$, we get a diagram
\begin{equation}
    \xymatrix{
    \xi(W)\otimes_{\bfA_S}\calO(P)\ar[r]& \calO(P^1)\otimes_{\bZ_p}W\ar@<.5ex>[r]\ar@<-.5ex>[r]&\calO(P^2)\otimes_{\bZ_p}W.
    }
\end{equation}
Then, we get a unique map $i_W:\xi(W)\to \xi(W)\otimes_{\bfA_S}\calO(P)\to \xi_P(W)$  by the universality of the equaliser.  As both $\xi(W)$ and $\xi_P(W)$ are (not necessarily finitely generated) \'etale $\varphi$-modules over $\bfA_S$ and $\varphi:\bfA_S\to \bfA_S$ is flat, the image $\Ima(i_W)$ also inherits an \'etale $\varphi$-module structure over $\bfA_S$.

When $W=V$, we have $\xi_P(V)=V\otimes_{\bZ_p}\calO(P)=:V_P$. There is also the universal isomorphism $\alpha:\calE_P:=\calE_{\bfA_S}\otimes_{\bfA_S}\calO(P)\cong\xi_P(V)$ corresponding to the identity $id: P\to P$. Note there is a natural map $\calE_{\bfA_S}\to \calE_P=\calE_{\bfA_S}\otimes_{\bfA_S}\calO(P)$, whose image we denote $\Ima(\calE_{\bfA_S})$. It is easy to see that $\Ima(\calE_{\bfA_S})\otimes_{\bfA_S}\widehat\bfA_S\cong \calE_{\widehat\bfA_S}$.

Let $\calE^{\circ}_V$ be the intersection $\Ima(\calE_{\bfA_S})\cap \Ima(i_V)$ in $\calE_P$. Then $\calE^{\circ}_V$ is a finitely generated \'etale $\varphi$-submodule of $\Ima(\calE_{\bfA_S})$. There is a left exact sequence of \'etale $\varphi$-modules
\[
0\to \calE^{\circ}_V\xrightarrow{f} \Ima(\calE_{\bfA_S})\oplus \Ima(i_V)\xrightarrow{g} \calE_P
\]
where $f(a)=a\oplus a$ and $g(x\oplus y)=x-y$. Now consider its base change along $\bfA_S\to \widehat\bfA_S$, we get a short exact sequence of \'etale $\varphi$-modules over $\widehat \bfA_S$
\[
0\to \calE^{\circ}_V\otimes_{\bfA_S}\widehat\bfA_S\xrightarrow{f} \calE_{\widehat \bfA_S}\oplus \Ima(i_V)\otimes_{\bfA_S}\widehat\bfA_S\xrightarrow{g} \calE_{\widehat P}.
\]

Recall that $\widehat P$ is a $G$-torsor over $\widehat \bfA_S$ and we have the following commutative diagram
\begin{equation}
    \xymatrix{
   \xi(V)\ar[r]\ar[d]^{i_V}& \widehat \xi(V)\ar[d]^{i_V\otimes_{\bfA_S}\widehat\bfA_S}\ar[dr]&\\
   \calE_P\ar[rd]\ar[r]& \calE_{\widehat P}\ar[rd] & \calE_{\widehat\bfA_S}\ar[d]\\
    & \calE_{\bfA_S}\otimes_{\bfA_S}\calO(P)\ar[r]& \calE_{\widehat\bfA_S}\otimes_{\bfA_S}\calO(P).
    }
\end{equation}
The rightmost small diagram is commutative as we explained in the paragraph after Theorem \ref{G-description}.

Then $\Ima(i_V)\otimes_{\bfA_S}\widehat\bfA_S$ is naturally isomorphic to $\calE_{\widehat\bfA_S}\subset \calE_{\widehat P}=\calE_{\widehat\bfA_S}\otimes_{\bfA_S}\calO(P)$. Most importantly, both $\Ima(i_V)$ and $\calE_{\bfA_S}$ land inside $\calE_{\widehat\bfA_S}\subset \calE_{\widehat\bfA_S}\otimes_{\bfA_S}\calO(P)$. This means that $\calE^{\circ}_V\otimes_{\bfA_S}\widehat \bfA_S$ is isomorphic to $\calE_{\widehat\bfA_S}$ as \'etale $\varphi$-modules. Moreover, as $\widehat \bfA_S\to \calO(\widehat P)$ is faithfully flat, the composite $\calE_{\bfA_S}\to \calE_{\widehat\bfA_S}\to \calE_{\widehat\bfA_S}\otimes_{\widehat\bfA_S}\calO(\widehat P)$ is injective, which implies $\Ima(\calE_{\bfA_S})=\calE_{\bfA_S}$.

By \cite[Theorem 5.5.20]{EG21}\footnote{\cite[Theorem 5.5.20]{EG21} is first proved by considering $\bF_p$-algebras. In this case, $\bfA_S^+$ is also $\varphi$-stable. So the same arguments work for $\bfA_S$. The step inducing the general case from the case of $\bF_p$-algebras is standard and works for $\bfA_S$ as well. So \cite[Theorem 5.5.20]{EG21} holds true for $\bfA_S$.} and the following Lemma \ref{canonical-descent}, we see that $\calE^{\circ}_V\cong \calE_{\bfA_S}$, i.e. $\calE_{\bfA_S}\subset \Ima(i_V)$.

\begin{lem}\label{canonical-descent}
    Let $\widehat M$ is a finite projective \'etale $\varphi$-module over $\widehat {\bfA_S}$. If $M_1\subset M_2\subset \widehat M$ are two finite projective \'etale $\varphi$-modules over $\bfA_S$ such that $M_1\otimes_{\bfA_S}\widehat{\bfA_S}\cong \widehat M\cong M_2\otimes_{\bfA_S}\widehat{\bfA_S}$, then $M_1=M_2$ as \'etale $\varphi$-modules over $\bfA_S$.
\end{lem}
\begin{proof}
Let $N$ be the quotient $M_2/M_1$, which is a finitely generated \'etale $\varphi$-module over $\bfA_S$. Then $N\otimes_{\bfA_S}\widehat\bfA_S=0$. By \cite[Theorem 5.5.20]{EG21}, we see $N=0$. So $M_1=M_2$. 
\end{proof}

Let $\bar\xi(W)$ be the image of $\xi(W)$ in $\widehat\xi(W)$. As $\bfA_S\to \widehat \bfA_S$ is flat, we have $\xi(W)\otimes_{\bfA_S}\widehat\bfA_S\cong \bar \xi(W)\otimes_{\bfA_S}\widehat\bfA_S\cong \widehat \xi(W)$ as \'etale $\varphi$-modules. When $W=V$, we also have $\calE_{\bfA_S}\subset \bar\xi(V)=\Ima(i_V)\subset\calE_{\widehat P}$ as \'etale $\varphi$-modules. As exact  $\otimes$-functors between rigid categories preserve dualities by \cite[Proposition 1.9]{deligne1982tannakian}, we can apply the above argument to $V^{\vee}$ to get $\calE^{\vee}_{\bfA_S}\subset \bar\xi(V^{\vee})\subset\calE_{\widehat P}^\vee$ and $\bar \xi(V^{\vee})\otimes_{\bfA_S}\widehat\bfA_S\cong \widehat \xi(V^{\vee})$.

\begin{lem}\label{1-out-of-3}
   Let $0\to W_1\xrightarrow{f} W_2\xrightarrow{g}W_3\to 0$ be a short exact sequence in $\Rep(G)$. If there exists a finitely generated \'etale $\varphi$-submodule $\calE_2\subset \bar\xi(W_2)$ over $\bfA_S$ such that $\calE_2\otimes_{\bfA_S}\widehat{\bfA_S}\cong \bar\xi(W_2)\otimes_{\bfA_S}\widehat{\bfA_S}\cong \widehat\xi(W_2)$ as finite projective \'etale $\varphi$-modules over $\widehat{\bfA_S}$, then $\calE_2$ is finite projective and there exists a short exact sequence of finite projective \'etale $\varphi$-modules over $\bfA_S$ 
   \[
   0\to \calE_1\xrightarrow{\tilde f}\calE_2\xrightarrow{\tilde g}\calE_3\to 0
   \]
   whose base change to $\widehat{\bfA_S}$ is isomorphic to $0\to \widehat\xi(W_1)\xrightarrow{f} \widehat\xi(W_2)\xrightarrow{g}\widehat\xi(W_3)\to 0$.
  
\end{lem}

\begin{proof}
That $\calE_2$ is finite projective follows from \cite[Theorem 5.5.20]{EG21}.

Now note that we have an injection of \'etale $\varphi$-modules $\bar\xi(W_1)\hookrightarrow\bar\xi(W_2)$. We let $\calE_1\subset \bar\xi(W_1)$ be the preimage of $\calE_2$. Then there is a left short exact sequence
\[
0\to \calE_1\xrightarrow{i} \calE_2\oplus \bar\xi(W_1)\xrightarrow{j} \bar\xi(W_2)
\]
where $i(a)=a\oplus a$ and $j(x\oplus y)=x-y$. In particular, $\calE_1$ is a finitely generated \'etale $\varphi$-module over $\bfA_S$. 

Now consider the base change of the above left exact sequence along $\bfA_S\to \widehat\bfA_S$
\[
0\to \calE_1\otimes_{\bfA_S}\widehat\bfA_S\xrightarrow{i} (\calE_2\oplus \bar\xi(W_1))\otimes_{\bfA_S}\widehat\bfA_S\xrightarrow{j} \bar\xi(W_2)\otimes_{\bfA_S}\widehat\bfA_S\cong \widehat\xi(W_2).
\]

By our assumption $\calE_2\otimes_{\bfA_S}\widehat{\bfA_S}\cong \bar\xi(W_2)\otimes_{\bfA_S}\widehat{\bfA_S}\cong \widehat\xi(W_2)$ and $\bar\xi(W_1)\otimes_{\bfA_S}\widehat\bfA_S\cong \widehat\xi(W_1)$, we see $\calE_1\otimes_{\bfA_S}\widehat\bfA_S\cong \widehat\xi(W_1)$. Then by \cite[Theorem 5.5.20]{EG21}, the \'etale $\varphi$-module $\calE_1$ is indeed finite projective.

Let $\calE_3$ be the quotient $\calE_2/\calE_1$. Then $\calE_3$ is a finitely generated \'etale $\varphi$-modules. The base change of the short exact sequences $0\to \calE_1\to \calE_2\to \calE_3\to 0$ along the map $\bfA_S\to \widehat\bfA_S$ is then  isomorphic to $0\to \widehat\xi(W_1)\xrightarrow{f} \widehat\xi(W_2)\xrightarrow{g}\widehat\xi(W_3)\to 0$. So we get $\calE_3\otimes_{\bfA_S}\widehat\bfA_S\cong \widehat\xi(W_3)$. By \cite[Theorem 5.5.20]{EG21} again, we see $\calE_3$ is a finite projective \'etale $\varphi$-module. 
\end{proof}

Now we are ready to prove $\calB_R$ satisfies condition $(d)$ In Lemma \ref{artin}.
\begin{prop}\label{effectivity}
    $\calB_R$ satisfies the condition $(d)$ in Lemma \ref{artin}, i.e. every formal object is effective.
\end{prop}
\begin{proof}
    Let $(S,\xi_n,f_n)$ be a formal object, i.e. $\xi_n:\Rep(G)\to \Mod^{\varphi}(\bfA_{S_n})$ is a compatible system of exact $\otimes$-functors, where $S$ is a Noetherian complete local $\bZ_p$-algebra with maximal ideal $\frakm$ and $S_n=S/\frakm^n$.

    For any $n,m\in\bZ_{\geq 0}$, we have a natural map $\bigotimes^n(\calE_V^{\oplus m})\to \bar \xi(\bigotimes^n(V^{\oplus m}))$  by the above discussion. Then its image is a finitely generated \'etale $\varphi$-submodule of $\bar \xi(\bigotimes^n(V^{\oplus m}))$ satisfying the assumption of Lemma \ref{1-out-of-3}. The same is true for $\bigotimes^n(V^{\oplus m})^{\vee}$. Then by Lemma \ref{Lem-rep} and Lemma \ref{1-out-of-3}, for any short exact sequence $0\to W_1\xrightarrow{f} W_2\xrightarrow{g}W_3\to 0$ in $\Rep(W)$, we get a short exact sequence $0\to \xi(W_1)\xrightarrow{\tilde f} \xi(W_2)\xrightarrow{\tilde g}\xi(W_3)\to 0$ of finite projective \'etale $\varphi$-modules over $\bfA_S$, whose base change is $0\to \widehat\xi(W_1)\xrightarrow{ f} \widehat\xi(W_2)\xrightarrow{g}\widehat \xi(W_3)\to 0$. But in Lemma \ref{Lem-rep} there are many different presentations of $W$ as a subquotient of $(\bigotimes^{s'}(\Sym^d(V^{\oplus d}))^{\vee})\bigotimes (\bigoplus_{i=0}^s\Sym^i(V^{\oplus d}))$. So we need to verify they will give rise to the same finite projective \'etale $\varphi$-modules over $\bfA_S$. By the construction in Lemma \ref{1-out-of-3}, for each $W\in \Rep(G)$, all the possible $\xi(W)$ are contained in $\bar\xi(W)$. So we need to prove that if $M_1,M_2$ are two finite projective \'etale $\varphi$-submodules of $\bar\xi(W)$, then $M_1=M_2$. To prove this, consider the left short exact sequence
    \[
    0\to M_1\cap M_2\xrightarrow{\alpha} M_1\oplus M_2\xrightarrow{\beta} \bar\xi(W)
    \]
    where $\alpha(a)=a\oplus a$ and $\beta(x\oplus y)=x-y$. We  see that $M_1\cap M_2$ is still a finitely generated \'etale $\varphi$-module and its base change to $\widehat\bfA_S$ is isomorphic to $\widehat\xi(W)$. So $M_1\cap M_2$ is finite projective by \cite[Theorem 5.5.20]{EG21}. Then by Lemma \ref{canonical-descent}, we have $M_1=M_2$.

    Hence we get an exact $\otimes$-functor $\xi:\Rep(G)\to \Mod^{\varphi}(\bfA_{S})$ whose base change along $\bfA_S\to \widehat\bfA_S$ is exactly $\widehat \xi:\Rep(G)\to \Mod^{\varphi}(\widehat\bfA_{S})$.
\end{proof}



Now let $\Spec(R)\to \calX_G^{\circ}$ be an $R$-point and $\calB_R^{\Gamma}$ be the pullback of the diagram $\calX_G\to \calX_G^{\circ}\leftarrow \Spec(R)$. Then $\calB_R^{\Gamma}\to \Spec(R)$ is a monomorphism.
\begin{prop}\label{etale-gamma}
    $\calB_R^{\Gamma}$ satisfies the conditions (a),(b),(c),(d) in Lemma \ref{artin}.
\end{prop}
\begin{proof}

    We first consider the case of \'etale $(\varphi,\Gamma_{\disc})$-modules with $G$-structure. By using the same arguments as the proof of Proposition \ref{etale} and \ref{effectivity}, we can see the corresponding $\calB_R^{\Gamma_{\disc}}$ satisfies the conditions $(a),(b),(c),(d)$ in Lemma \ref{artin}.

    Note that there is a pullback diagram
    \begin{equation}
        \xymatrix{
        \calX_G\ar[r]\ar[d]&\calX_G^{\circ}\ar[d]\\
        \calR_G^{\Gamma_{\disc}}\ar[r]&\calR_{G}^{\Gamma_{\disc},\circ}
        }
    \end{equation}
    which implies $\calB_R^{\Gamma}=\calB_{R}^{\Gamma_{\disc}}$. So we are done.
\end{proof}

Now in order to show $\calB_R^{\Gamma}$ is an algebraic space locally of finite presentation over $\Spec(R)$, we have to prove the condition (e), i.e. the openness of versality. As always, this is the most difficult one to verify.

\subsection{Obstruction theory}\label{obstruction}
We focus on verifying the openness of versality for $\calB_R^{\Gamma}$ in this subsection. The basic tool is still provided by \cite{Art74}. Let us first recall some fundamental notions in \cite{Art74} tailored in our context. Let $\calF$ be a limit preserving functor from $\Nilp_{\bZ_p}$ to $\rm Sets$.

\begin{dfn}[Deformation theory {\cite[2.5]{Art74}}]
    Let $A_0$ be a reduced ring in $\Nilp_{\bZ_p}$. For any finite $a_0\in \calF(A_0)$ and finite $A_0$-module $M$, define $D_{a_0}(M):={\rm Fib}_{a_0}(\calF(A_0\oplus M)\to \calF(A_0))$.
\end{dfn}

\begin{dfn}[Obstruction theory {\cite[2.6]{Art74}}]\label{def-obstruction theory}
    An obstruction theory for $\calF$ is the following data.
    \begin{enumerate}
        \item For each infinitesimal extension $A\twoheadrightarrow A_0$, i.e. $A_0$ is reduced and the kernel of this map is nilpotent, and $a\in \calF(A)$, a functor
        \[
        \calO_a:\{\text{finite}\ A_0\text{-modules}\}\to \{\text{finite}\ A_0\text{-modules}\}.
        \]
        \item For each deformation situation, i.e. a diagram
        \[
        A'\to A\to A_0, \ \Ker(A'\to A)=M
        \]
        of infinitesimal extensions of a reduced ring $A_0$, where $A'\to A$ is surjective and $M$ is a finite $A_0$-module, and $a\in \calF(A)$, an element $o_a(A')\in \calO_a(M)$ which is zero if ${\rm Fib}_a(\calF(A')\to \calF(A))$ is not empty.
    \end{enumerate}
\end{dfn}

There are some conditions that the deformation theory and the obstruction theory are supposed to satisfy in order to prove the openness of versalituy.

\begin{cond}[{\cite[4.1]{Art74}}]\label{Artin-condition}
    \begin{enumerate}
    \item The modules $D$ and $\calO$ are compatible with \'etale localization: If $g:A\to B$ is \'etale, $b=f(a)$, then
    \[
    D_{b_0}(M\otimes_{A_0} B_0)\cong D_{a_0}(M)\otimes_{A_0} B_0,
    \]
    and
    \[
    \calO_b(M\otimes_{A_0} B_0)\cong \calO_a(M)\otimes_{A_0} B_0.
    \]
    
    \item $D$ is compatible with completions: If $\frakm$ is a maximal ideal of $A_0$, then
    \[
    D_{a_0}(M)\otimes \hat A_0\cong \varprojlim D_{a_0}(M/\frakm^nM).
    \]
    \item  There is an open dense subset of $\Spec(A_0)$ such that any finite type point $s$ in this subset satisfies 
    \[
    D_{a_0}(M)\otimes_{A_0} k(s)\cong D_{a_0}(M\otimes_{A_0} k(s)),
    \]
    and 
    \[
    \calO_a(M)\otimes_{A_0} k(s)\subseteq \calO_a(M\otimes_{A_0} k(s)).
    \]
\end{enumerate}
\end{cond}

These conditions will imply the openness of versality by \cite[Theorem 4.4]{Art74}. Let us first deal with the deformation theory, which is the easy part. In fact, the functor $\calB_R^{\Gamma}\to \Spec(R)$ is a monomorphism, which implies the fiber $D_{a_0}(M)$ is trivial. So $D_{a_0}(M)$ must satisfy all the above conditions. 

Now we move to study the obstruction theory. 
For any point $a:\Spec(A)\to \calB_R^{\Gamma}\to \Spec(R)$ with $R$ being of finite type over $\bZ_p$ and $A$ being of finite type over $R$, we need to construct a functor from the category of finite $A_0$-modules to the category of finite $A_0$-modules satisfying Artin's conditions, where $A_0$ is the reduced ring of $A$. Intuitively, given a finite $A_0$-module $M$, this functor is concerned with lifting the point $a$ to $\Spec(A\oplus M)$. Unveiling the definition of $\calB_R^{\Gamma}$, this only depends on whether we can lift the $G$-torsor over $\bfA_{A}$ to $\bfA_{A\oplus M}$ such that this lift has underlying pair given by the map $\Spec(A\oplus M)\to \Spec(R)$.

The above discussion leads us to consider the pullback diagram
\begin{equation*}
    \xymatrix{
    \calC_R\ar[r]\ar[d]&\Spec(\bfA_R)\ar[d]&\\
    BG\ar[r]&(BG)^{\circ}.
    }
\end{equation*}
where $(BG)^{\circ}$ is the stack of pairs $(\calE,\alpha:t(\calE)\twoheadrightarrow\calL)$ such that $\calE$ is a vector bundle of rank $d$ and $\calL$ is a vector bundle of rank $t(d)-1$, and $\alpha$ is a surjective map. The right vertical morphism is induced by the morphism $\Spec(R)\to \calX_G^{\circ}$ and Lemma \ref{Lemma-Gamma}.

\begin{lem}\label{BG-BGo}
    The stack $(BG)^{\circ}$ is an algebraic stack locally of finite presentation and the monomorphism $BG\to (BG)^{\circ}$ is representable and locally of finite presentation. 
    \end{lem}

\begin{proof}
    We use the same arguments as in \cite[Section 2]{Bro13}. By the representability of relative Quotient functor, the natural map $(BG)^{\circ}\to B\GL_{d}$ is representable and locally of finite presentation. So $(BG)^{\circ}$ is an algebraic stack locally of finite presentation.
    
    Let  $(\calE,\calL\subset t(\calE))$ be the universal pair over $(BG)^{\circ}$. Let $\calI:=\underline\Isom((V,L),(\calE,\calL\subset t(\calE)))$. Then $\calI\to (BG)^{\circ}$ is representable in schemes, affine and of finite presentation.

     Now by \cite[lemma 2.5]{Bro13} (let $X=\Spec(\bZ_p)$ in \textit{loc.cit}), there is an Artin stack $\calT$ locally of finite presentation over $(BG)^{\circ}$ representing the condition on $(BG)^{\circ}$-schemes $T$ such that $\calI\times_{(BG)^{\circ}}T$ is flat and pure over $T$. In particular, $\calI\times_{(BG)^{\circ}}\calT$ is flat over $\calT$. Let $\calU$ be its open image in $\calT$. Then $\calU$ represents the condition on $(BG)^{\circ}$-schemes $T$ such that $\calU\times_{(BG)^{\circ}}T$ is flat, surjective and pure over $T$. As $G$ is faithfully flat over $\bZ_p$, it is pure over $\bZ_p$. Then by the last paragraph of the proof of \cite[Lemma 2.1]{Bro13}, we see $\calU\simeq BG$. Hence $BG\to (BG)^{\circ}$ is representable and locally of finite presentation.
\end{proof}
Lemma \ref{BG-BGo} then implies that $\calC_R\to \Spec(\bfA_R)$ is locally of finite presentation. In particular, this means the relative cotangent complex\footnote{We refer to \cite[Chapter 1]{GR17} for a systematic presentation of (relative) cotangent complex.} $T^*(\calC_R/\Spec(\bfA_R))$ is an almost perfect complex of $\calO_{\calC_R}$-modules, i.e. for each $n\geq 0$, there exists a perfect complex $T_n$ of $\calO_{\calC_R}$-modules such that $\tau_{\leq n}T_n\simeq \tau_{\leq n}T^*(\calC_R/\Spec(\bfA_R))$. Moreover, $H^1(T^*(\calC_R/\Spec(\bfA_R)))$ provides a natural obstruction theory in Artin's sense. Our observation is that there is a natural $(\varphi,\Gamma)$-action on $H^1(T^*(\calC_R/\Spec(\bfA_R)))$ which preserves the obstruction class. This will enable us to construct an obstruction theory for the morphism $\calB_R^{\Gamma}\to \Spec(R)$.

To make this idea work, we have to consider the derived version of $\calC_R$ (cf. Remark \ref{why-derived}). By abuse of notations, we still write $BG$\footnote{This is actually the same as Definition \ref{BG} we will introduce later.} (resp. $(BG)^{\circ}$) as the \'etale sheafification of the left Kan extension of $BG$ (resp. $(BG)^{\circ}$) along the inclusion from the classical commutative rings to animated rings. Then both $BG$ and $(BG)^{\circ}$ are derived Artin stack locally almost of finite presentation by \cite[Proposition 4.4.3 and 4.5.2]{GR17}. By the proof of \cite[Proposition 4.4.3]{GR17}, we see $BG\to(BG)^{\circ}$ is still representable. In particular, we can consider the following homotopy pullback diagram
\begin{equation*}
    \xymatrix{
    \bfC_R\ar[r]\ar[d]&\Spec(\bfA_R)\ar[d]&\\
    BG\ar[r]&(BG)^{\circ}
    }
\end{equation*}
where $\bfC_R$ is a derived ($0$-)Artin stack whose underlying classical stack is $\calC_R$. Since $\bfA_R$ is Noetherian which is locally almost of finite presentation, we see $\bfC_R$ is also locally almost of finite presentation as filtered colimit commutes with fiber product.

For any point $a:\Spec(A)\to \calB_R^{\Gamma}\to \Spec(R)$ with $R$ being of finite type over $\bZ_p$ and $A$ being of finite type over $R$, we can consider the following induced diagram
\begin{equation}\label{diagram for continuity}
    \xymatrix{
    \Spec(\widetilde{\bfA_A})\ar[r]^{\pi_2}\ar[d]^{\tilde a'}&\Spec(\bfA_A)\ar[d]^{a'}\ar[rd]^{a}& \\
    \widetilde{\bfC_R}\ar[r]^{\pi_1}\ar[d]^{\tilde \beta}&\bfC_R\ar[r]^{\iota_1}\ar[d]^{\beta}& \Spec(\bfA_R)\ar[d]^{\alpha}\\
    \Spec(\bZ_p)\ar[r]^{\pi_0}& BG\ar[r]^{\iota_0}& (BG)^{\circ}
    }
\end{equation}
where all the squares are homotopy pullback diagrams and $\pi_0:\Spec(\bZ_p)\to BG$ is the universal $G$-torsor over $BG$. In particular, the map $a':\Spec(\bfA_A)\to \bfC_R$ factors through $\calC_R$ as $\bfA_A$ is discrete. And the maps $\pi_0,\pi_1,\pi_2$ are all faithfully flat. So $\widetilde{\bfA_A}$ is also discrete. In fact, it is the $G$-torsor over $\bfA_A$ corresponding to the map $\beta\circ a'$.

Let $A_0$ still be the reduced ring of $A$ and $M$ be a finite $A_0$-module. Then $M\otimes_{A_0}\bfA_{A_0}$ is a finite $\bfA_{A_0}$-module and $\bfA_{A_0}$ is the reduced ring of $\bfA_{A}$. Now we consider the relative cotangent complex $T^*_{a'}(\bfC_R/\Spec(\bfA_R))$ of $\calC_R\to \Spec(\bfA_R)$ at the point $a'$. For simplicity, we write $T^*_{a'}:=T^*_{a'}(\bfC_R/\Spec(\bfA_R))=T^*(\bfC_R/\Spec(\bfA_R))\otimes^{\bL}_{\calO_{\bfC_R}}\bfA_A$. As $\bfC_R$ is locally almost of finite presentation, the cotangent complex $T^*(\bfC_R/\Spec(\bfA_R))$ is an almost perfect complex of $\calO_{\bfC_R}$-modules. So $T^*_{a'}$ is also an almost perfect complex of $\bfA_A$-modules.

Suppose there is a deformation situation $A'\to A\to A_0, \ \Ker(A'\to A)=M$. Then it induces a deformation situation $\bfA_{A'}\to \bfA_A\to \bfA_{A_0}, \ \Ker(\bfA_{A'}\to \bfA_A)=M\otimes_{A_0}\bfA_{A_0}$. By the general theory of cotangent complex, there exists an obstruction class $o_{a'}(\bfA_{A'})\in H^1(T^*_{a'}(M)):=H^1(\RHom(T^*_{a'}\otimes^{\bL}_{\bfA_A}\bfA_{A_0},M\otimes_{A_0}\bfA_{A_0})$, which vanishes if and only if there exists a lifting $\Spec(\bfA_{A'})\to \bfC_R$ of $a':\Spec(\bfA_A)\to \bfC_R$. More precisely, there is a square-zero extension
\begin{equation}\label{square-zero-ext}
    \xymatrix{
    \bfA_{A'}\ar[r]\ar[d]&\bfA_{A}\ar[d]^{\rm triv}\\
    \bfA_A\ar[r]^-{s}& \bfA_A\oplus (M\otimes_{A}\bfA_A)[1]
    }
\end{equation}
where ${\rm triv}:\bfA_A\to \bfA_A\oplus (M\otimes_{A}\bfA_A)[1]$ is the trivial section and $s:\bfA_A\to \bfA_A\oplus (M\otimes_{A}\bfA_A)[1]$ is the section induced by $M\hookrightarrow A'\twoheadrightarrow A$. Then the existence of a lifting $\Spec(\bfA_{A'})\to \bfC_R$ of $a':\Spec(\bfA_A)\to \bfC_R$ is equivalent to the equivalence between the composite $\Spec(\bfA_A\oplus (M\otimes_{A}\bfA_A)[1])\xrightarrow{s}\Spec(\bfA_A)\xrightarrow{a'}\bfC_R$ and the composite $\Spec(\bfA_A\oplus (M\otimes_{A}\bfA_A)[1])\xrightarrow{\rm triv}\Spec(\bfA_A)\xrightarrow{a'}\bfC_R$. By the definition of cotangent complex, the composite $\Spec(\bfA_A\oplus (M\otimes_{A}\bfA_A)[1])\xrightarrow{s}\Spec(\bfA_A)\xrightarrow{a'}\bfC_R$ corresponds to an element in $H^1(T^*_{a'}(M))$, which we denote by $o_{a'}(\bfA_{A'})$. In fact, $\Spec(\bfA_A\oplus (M\otimes_{A}\bfA_A)[1])\xrightarrow{s}\Spec(\bfA_A)$ corresponds to an element in $H^1(\RHom(T^*_a(\Spec(\bfA_A)/\Spec(\bfA_R)),M\otimes_{A}\bfA_A))$, whose image in $H^1(T^*_{a'}(M))$ is just  $o_{a'}(\bfA_{A'})$. Also the composite $\Spec(\bfA_A\oplus (M\otimes_{A}\bfA_A)[1])\xrightarrow{\rm triv}\Spec(\bfA_A)\xrightarrow{a'}\bfC_R$ corresponds to $0\in H^1(T^*_{a'}(M))$. So the existence of a lifting is equivalent to $o_{a'}(\bfA_{A'})=0$.

We claim that there is a $(\varphi,\Gamma)$-action on $T^*_{a'}$. In fact, there is an equivalence between $\alpha\circ\varphi_{\bfA_A}$ and $\alpha$, which induces an equivalence between $\beta\circ a'\circ\varphi$ and $\beta\circ a'$. So this induces further an isomorphism 
\[
T^*_{\beta\circ a'}(BG/(BG)^{\circ})\otimes^{\bL}_{\bfA_A,\varphi}\bfA_A\simeq T^*_{\beta\circ a'}(BG/(BG)^{\circ}).
\]
So we get a $\varphi$-action on $T^*_{\beta\circ a'}(BG/(BG)^{\circ})\simeq T^*_{a'}(\bfC_R/\Spec(\bfA_R))$. The same arguments can produce a compatible $\Gamma$-action $T^*_{\beta\circ a'}(BG/(BG)^{\circ})\simeq T^*_{a'}(\bfC_R/\Spec(\bfA_R))$. This also yields a $(\varphi,\Gamma)$-action on $H^1(T^*_{a'}(M))$. More precisely, as  $\varphi:\bfA_{A_0}\to \bfA_{A_0}$ is flat, we have 
\[
\varphi^*H^1(T^*_{a'}(M)):=H^1(T^*_{a'}(M))\otimes_{\bfA_{A_0},\varphi}\bfA_{A_0}=H^1(\RHom(T^*_{a'}\otimes^{\bL}_{\bfA_{A}}\bfA_{A_0},M\otimes_{A_0}\bfA_{A_0})\otimes_{\bfA_{A_0}}\bfA_{A_0}).
\]
Recall that $T^*_{a'}$ is an almost perfect complex. As $M\otimes_{A_0}\bfA_{A_0}$ is discrete and we are only concerned with $H^1$, we may assume $T^*_{a'}$ is a perfect complex. Then this implies
\[
H^1(\RHom(T^*_{a'}\otimes^{\bL}_{\bfA_{A}}\bfA_{A_0},M\otimes_{A_0}\bfA_{A_0})\otimes_{\bfA_{A_0}}\bfA_{A_0})=H^1(\RHom(T^*_{a'}\otimes^{\bL}_{\bfA_{A},\varphi}\bfA_{A_0},M\otimes_{A_0}\bfA_{A_0}\otimes_{\bfA_{A_0},\varphi}\bfA_{A_0})).
\]
Using $T^*_{a'}\otimes^{\bL}_{\bfA_{A},\varphi}\bfA_{A_0}\simeq T^*_{a'}\otimes^{\bL}_{\bfA_{A}}\bfA_{A_0}$ and $M\otimes_{A_0}\bfA_{A_0}\otimes_{\bfA_{A_0},\varphi}\bfA_{A_0}\simeq M\otimes_{A_0}\bfA_{A_0}$, we get an isomorphism $\varphi^*H^1(T^*_{a'}(M))\cong H^1(T^*_{a'}(M))$. Similarly, we can get a $\Gamma$-action on $H^1(T^*_{a'}(M))$.

Now go back to our deformation situation. Recall that there exists an obstruction class $o_{a'}(\bfA_{A'})\in H^1(T^*_{a'}(M)):=H^1(\RHom(T^*_{a'}\otimes^{\bL}_{\bfA_A}\bfA_{A_0},M\otimes_{A_0}\bfA_{A_0}))$, which vanishes if and only if there exists a lifting $\Spec(\bfA_{A'})\to \bfC_R$ of $a':\Spec(\bfA_A)\to \bfC_R$. As $\bfA_{A'}$ and $\bfA_A$ are both discrete, the existence of a lifting $\Spec(\bfA_{A'})\to \bfC_R$ of $a':\Spec(\bfA_A)\to \bfC_R$ is equivalent to the existence of a lifting $\Spec(\bfA_{A'})\to \calC_R$ of $a':\Spec(\bfA_A)\to \calC_R$, which is then controlled by the obstruction class $o_{a'}(\bfA_{A'})$.

We have the following lemma. 
\begin{lem}
We have $o_{a'}(\bfA_{A'})\in H^1(T^*_{a'}(M))^{\varphi=1,\Gamma=1}$.
\end{lem}
\begin{proof}
    We only have to show $o_{a'}(\bfA_{A'})$ is $\varphi$-invariant. The same argument can show $o_{a'}(\bfA_{A'})$ is also $\Gamma$-invariant.

    Recall that $o_{a'}(\bfA_{A'})$ corresponds to the composite $\Spec(\bfA_A\oplus (M\otimes_{A}\bfA_A)[1])\xrightarrow{s}\Spec(\bfA_A)\xrightarrow{\beta\circ a'}BG$, where $s$ is the section in Diagram \ref{square-zero-ext}. Then the image of $o_{a'}(\bfA_{A'})$ in $\varphi^*H^1(T^*_{a'}(M))$, i.e. $1\otimes o_{a'}(\bfA_{A'})$, corresponds to the composite $\Spec(\bfA_A\oplus (M\otimes_{A}\bfA_A)[1])\xrightarrow{s}\Spec(\bfA_A)\xrightarrow{\varphi}\Spec(\bfA_A)\xrightarrow{\beta\circ a'}BG$. Unraveling the construction, the isomorphism $\varphi:\varphi^*H^1(T^*_{a'}(M))\cong H^1(T^*_{a'}(M))$ is induced by the equivalence $\beta\circ a'\circ\varphi\simeq \beta\circ a':\Spec(\bfA_A)\to BG$, which actually gives the isomorphism $ \varphi^*_{\bfA_A}T^*_{\beta\circ a'}(BG/(BG)^{\circ})\simeq T^*_{\beta\circ a'}(BG/(BG)^{\circ})$. So under this equivalence, the composite $\Spec(\bfA_A\oplus (M\otimes_{A}\bfA_A)[1])\xrightarrow{s}\Spec(\bfA_A)\xrightarrow{\varphi}\Spec(\bfA_A)\xrightarrow{\beta\circ a'}BG$ is sent to  $\Spec(\bfA_A\oplus (M\otimes_{A}\bfA_A)[1])\xrightarrow{s}\Spec(\bfA_A)\xrightarrow{\beta\circ a'}BG$, which simply means $\varphi(1\otimes o_{a'}(\bfA_{A'}))=o_{a'}(\bfA_{A'})$. So we are done.

\end{proof}

Now we have a good candidate of an obstruction theory for the morphism $\calB_R^{\Gamma}\to \Spec(R)$. Given an infinitesimal extension $A\twoheadrightarrow A_0$ as in Definition \ref{def-obstruction theory}, let $a:\Spec(A)\to \calB_R^{\Gamma}\to \Spec(R)$. We define a functor
\[
\calO_a:\{\text{finite}\ A_0\text{-modules}\}\to \{A_0\text{-modules}\}
\]
by sending each finite $A_0$-module $M$ to the $A_0$-module $\calO_a(M):=H^1(T^*_{a'}(M))^{\varphi=1,\Gamma=1}$. For each deformation situation $A'\to A\to A_0, \ \Ker(A'\to A)=M$, we have shown there exists an obstruction class $o_{a'}(\bfA_{A'})\in H^1(T^*_{a'}(M))^{\varphi=1,\Gamma=1}$ controlling the existence of liftings. So in order to show our functor $\calO_a$ is indeed an obstruction theory, we need to prove $\calO_a(M)$ is indeed a finite $A_0$-module. This is our next task.

\begin{lem}
    For each finite $A_0$-module $M$, we have $H^1(T^*_{a'}(M))$ is a finite $\bfA_{A_0}$-module.
\end{lem}
\begin{proof}
    As $\bfC_R$ is locally almost of finite presentation over $\bfA_R$, we know $T^*_{a'}$ is an almost perfect complex of $\bfA_A$-modules. Then 
$T^*_{a'}\otimes^{\bL}_{\bfA_A}\bfA_{A_0}$ is also an almost perfect complex of $\bfA_{A_0}$-modules. As $M\otimes_{A_0}\bfA_{A_0}$ is discrete and we are only concerned with $H^1$, we may assume $T^*_{a'}\otimes^{\bL}_{\bfA_A}\bfA_{A_0}$ is a perfect complex and then it is easy to see $H^1(T^*_{a'}(M))$ is a finite $\bfA_{A_0}$-module.
\end{proof}

Now we know that $H^1(T^*_{a'}(M))$ is a finite $\bfA_{A_0}$-module with a $(\varphi,\Gamma)$-action. In particular, $H^1(T^*_{a'}(M))$ has a natural topology induced from that on $\bfA_{A_0}$. We claim that the $\Gamma$-action on $T^*_{a'}$ is continuous, which will be very important in our later proof that $\calO_a$ indeed defines an obstruction theory.

\begin{lem}\label{continuity}
    The $\Gamma$-action on $H^1(T^*_{a'}(M))$ is continuous.
\end{lem}

\begin{proof}
    We need to use Diagram \ref{diagram for continuity}. In particular, let $\Spec(\widetilde{\bfA_{A_0}})\xrightarrow{\pi_3}\Spec(\bfA_{A_0})$ be the base change of $\Spec(\widetilde{\bfA_{A}})\xrightarrow{\pi_2}\Spec(\bfA_{A})$ along $\Spec(\bfA_{A_0})\to \Spec(\bfA_A)$. Then $\Spec(\widetilde{\bfA_{A_0}})$ is the corresponding geometric $G$-torsor over $\Spec(\bfA_{A_0})$. By Lemma \ref{Key-0}, we know that $\Gamma$-action on $\widetilde{\bfA_{A_0}}$ is continuous with respect to the filtered colimit topology on $\widetilde{\bfA_{A_0}}$.

  Put $K:=H^1(T^*_{a'}(M))\otimes_{\bfA_{A_0}}\widetilde{\bfA_{A_0}}$. If we can prove the $\Gamma$-action on $K$ is continuous with respect to its filtered colimit topology, then the $\Gamma$-action on $H^1(T^*_{a'}(M))$ will be automatically continuous.

  Note that $H^1(T^*_{a'}(M))\otimes_{\bfA_{A_0}}\widetilde{\bfA_{A_0}}=H^1(\RHom(T^*_{a'}\otimes^{\bL}_{\bfA_A}\bfA_{A_0},M\otimes_{A_0}\bfA_{A_0})\otimes^{\bL}_{\bfA_{A_0}}\widetilde{\bfA_{A_0}})$. Again as $M\otimes_{A_0}\bfA_{A_0}$ is discrete and we are only concerned with $H^1$, we may assume $T^*_{a'}\otimes^{\bL}_{\bfA_A}\bfA_{A_0}$ is a perfect complex. Then we have $K=H^1(\RHom(T^*_{a'}\otimes^{\bL}_{\bfA_A}\widetilde{\bfA_{A_0}}, M\otimes_{A_0}\widetilde{\bfA_{A_0}}))$.

  Now by Diagram \ref{diagram for continuity}, we see that 
  \[
  T^*_{a'}\otimes^{\bL}_{\bfA_A}\widetilde{\bfA_{A_0}}\simeq T^*_{\pi_0}(BG/(BG)^{\circ})\otimes_{\bZ_p}^{\bL}\widetilde{\bfA_{A_0}}.
  \]
  For any $\gamma\in \Gamma$, we have the following commutative diagram
  \begin{equation}
      \xymatrix{
      \Spec(\widetilde{\bfA_{A_0}})\ar[rddd]\ar[rrd]^{\pi_3}\ar[rd]^{\tilde\gamma}& &\\
      & \Spec(\gamma^*_{\bfA_{A_0}}\widetilde{\bfA_{A_0}})\ar[r]^{\gamma^*\pi_3}\ar[d]^{i}& \Spec(\bfA_{A_0})\ar[d]^{\gamma}\\
      &\Spec(\widetilde{\bfA_{A_0}})\ar[d]\ar[r]^{\pi_3}&\Spec(\bfA_{A_0})\ar[d]^{\alpha}\\
      &\Spec(\bZ_p)\ar[r]^{\pi_0}& BG
      }
  \end{equation}
  where the two small squares are pullback diagrams and $\tilde \gamma:\Spec(\widetilde{\bfA_{A_0}})\to \Spec(\gamma^*_{\bfA_{A_0}}\widetilde{\bfA_{A_0}})$ is induced by the equivalence $\alpha\circ \gamma\simeq \alpha: \Spec(\bfA_{A_0})\to BG$. By abuse of notation, we write $\gamma:\Spec(\widetilde{\bfA_{A_0}})\to \Spec(\widetilde{\bfA_{A_0}})$ the composite $i\circ \tilde\gamma$ (in fact, this is the same as the $\gamma$-action on $\Spec(\widetilde{\bfA_{A_0}})$ we constructed in Section ). Then the $\gamma$-action commutes with $\pi_3$.

  Note that the $\gamma$-action on $T^*_{a'}\otimes^{\bL}_{\bfA_A}\widetilde{\bfA_{A_0}}$ is induced by the equivalence $\alpha\circ\pi_3\simeq \alpha\circ\gamma\circ\pi_3$, which is exactly given by $\gamma:\Spec(\widetilde{\bfA_{A_0}})\to \Spec(\widetilde{\bfA_{A_0}})$. Under the equivalence $T^*_{a'}\otimes^{\bL}_{\bfA_A}\widetilde{\bfA_{A_0}}\simeq T^*_{\pi_0}(BG/(BG)^{\circ})\otimes_{\bZ_p}^{\bL}\widetilde{\bfA_{A_0}}$, we see that the $\gamma$-action is just $id\otimes \gamma: T^*_{\pi_0}(BG/(BG)^{\circ})\otimes_{\bZ_p}^{\bL}\widetilde{\bfA_{A_0}}\to  T^*_{\pi_0}(BG/(BG)^{\circ})\otimes_{\bZ_p}^{\bL}\widetilde{\bfA_{A_0}}$.

  So the $\Gamma$-action on $K=H^1(\RHom(T^*_{\pi_0}(BG/(BG)^{\circ})\otimes_{\bZ_p}^{\bL}\widetilde{\bfA_{A_0}}, M\otimes_{A_0}\widetilde{\bfA_{A_0}}))$ is simply induced by the $\Gamma$-action on $\widetilde{\bfA_{A_0}}$. As $M$ is discrete and we are only concerned with $H^1$, we may assume $T^*_{\pi_0}(BG/(BG)^{\circ})$ is a perfect complex of $\bZ_p$-modules. Then $K=H^1((T^*_{\pi_0}(BG/(BG)^{\circ})^{\vee}\otimes_{\bZ_p}^{\bL}M)\otimes_{A_0}^{\bL}\widetilde{\bfA_{A_0}})=H^1((T^*_{\pi_0}(BG/(BG)^{\circ})^{\vee}\otimes_{\bZ_p}^{\bL}M))\otimes_{A_0}^{\bL}\widetilde{\bfA_{A_0}}$. Since $H^1((T^*_{\pi_0}(BG/(BG)^{\circ})^{\vee}\otimes_{\bZ_p}^{\bL}M))$ is a finite $A_0$-module and the $\Gamma$-action on $\widetilde{\bfA_{A_0}}$ is continuous, we see the $\Gamma$-action on $K$ is also continuous. So we are done.

\end{proof}

\begin{rmk}\label{why-derived}
   The key point in the above proof is that $T^*_{a'}\otimes^{\bL}_{\bfA_A}\widetilde{\bfA_{A_0}}$ has a $\bZ_p$-model, i.e. $T^*_{\pi_0}(BG/(BG)^{\circ})$. This is possible as we work with derived Artin stacks or more precisely $\bfC_R$. On the contrary, if we work with $\calC_R$ and $T^*(\calC_R/\Spec(\bfA_R))$, then it is not obvious to us that $T^*(\calC_R/\Spec(\bfA_R))$ admits a $\bZ_p$-model so that we can run the same argument to show the continuity of $\Gamma$-action. This is because cotangent complexes only commute with flat base change (or arbitrary base change in the derived setting). But we do not know the flatness of $BG\to (BG)^{\circ}$ or $\Spec(\bfA_R)\to (BG)^{\circ}$. In fact, if $BG\to (BG)^{\circ}$ is flat, then it must be an open immersion, in which case the obstruction theory will be trivial.
\end{rmk}

As a result, $H^1(T^*_{a'}(M))$ becomes a finitely generated \'etale $(\varphi,\Gamma)$-module. Note that we are working with reduced rings over $\bF_p$. So it makes sense to talk about $(\varphi,\Gamma)$-stable lattices. In particular, we have the following lemma.

\begin{lem}\label{lattice}
    $H^1(T^*_{a'}(M))$ admits a $(\varphi,\Gamma)$-stable lattice.
\end{lem}
\begin{proof}
    The proof is the same as that of \cite[Lemma 5.1.5]{EG22} once we have Lemma \ref{continuity}.
\end{proof}

\begin{prop}
    The functor $\calO_a$ sending finite $A_0$-module $M$ to $H^1(T^*_{a'}(M))^{\varphi=1,\Gamma=1}$ gives an obstruction theory for the morphism $\calB^{\Gamma}_R\to \Spec(R)$.
\end{prop}
\begin{proof}
    It remains to show $H^1(T^*_{a'}(M))^{\varphi=1,\Gamma=1}$ is a finite $A_0$-module. Once we have Lemma \ref{lattice}, the arguments in the proof of \cite[Theorem 5.1.22(1)]{EG22} work verbatim here. 
\end{proof}

Now we need to verify Artin's conditions on the obstruction theory. As mentioned earlier, the conditions on the deformation theory is trivial as $D_a$ is trivial functor in our context. So we focus on the obstruction theory. By examining the paper \cite{Art74}, one can see that in order to prove the openness of versality \cite[Theorem 4.4]{Art74}, the Item (3) in Artin's conditions \ref{Artin-condition} concerning the obstruction theory is only used at the bottom of Page 177 and top of Page 178, where $A_0$ (the ring $B_0$ in \textit{loc.cit}) is integral and the finite $A_0$-module $M$ is assumed to be finite free. So we could make the same assumptions when dealing with Item (3).

\begin{prop}\label{obstruction-theory-valid}
    \begin{enumerate}
        \item The functor $\calO_a$ is compatible with \'etale localizations: If $f:A\to B$ is \'etale, then
        \[
        \calO_b(M\otimes B_0)\cong \calO_a(M)\otimes B_0
        \]
        where $b=\calB^{\Gamma}_R(f)(a)$ and $B_0$ is the reduced ring of $B$.
        \item Let $A_0$ be integral and $M$ be a finite free $A_0$-module. There is an open dense subset of $\Spec(A_0)$ such that any finite type point $s$ in this subset satisfies
        \[
        \calO_a(M)\otimes k(s)\subset \calO_a(M\otimes k(s)),
        \]
        where $k(s)$ is the residue field of the point $s$.
    \end{enumerate}
\end{prop}

\begin{proof}
Write $\tilde M:=H^1(T^*_{a'}(M))=H^1(\RHom(T^*_{a'},M\otimes_{A}\bfA_A))$ for simplicity.
    \begin{enumerate}
        \item For the first one, note that $\tilde M^{\varphi=1,\Gamma=1}=H^0(C^{\bullet}(\tilde M))$, where $C^{\bullet}(\tilde M)$ is the Herr complex associated to the finitely generated \'etale $(\varphi,\Gamma)$-module $\tilde M$. 
        
        Since $A\to B$ is flat, we have $\bfA_A\to \bfA_B$ is also flat by \cite[Proposition 2.2.4]{EG22}. In particular, 
        \[
      \tilde M\otimes_{\bfA_A}\bfA_B=  H^1(\RHom(T^*_{a'},M\otimes_{A}\bfA_A))\otimes_{\bfA_A}\bfA_B\cong H^1(\RHom(T^*_{a'},M\otimes_{A}\bfA_A)\otimes_{\bfA_A}\bfA_B).
        \]
        As $M$ is discrete and we are only concerned with $H^1$, we may assume $T^*_{a'}$ is a perfect complex of $\bfA_A$-modules. So 
        \[
        H^1(\RHom(T^*_{a'},M\otimes_{A}\bfA_A)\otimes_{\bfA_A}\bfA_B)\cong H^1(\RHom(T^*_{a'}\otimes_{\bfA_A}\bfA_B,M\otimes_{A}\bfA_B))
        \]
   and
        \[
        \calO_b(M\otimes_{A_0}B_0)=H^1(\RHom(T^*_{a'}\otimes_{\bfA_A}\bfA_B,M\otimes_{A}\bfA_B))^{\varphi=1,\Gamma=1}\cong (\tilde M\otimes_{\bfA_A}\bfA_B)^{\varphi=1,\Gamma=1}.
        \]
        
As $A\to B$ is \'etale, we have $B_0=A_0\otimes_AB$ and $M\otimes_{A_0}B_0=M\otimes_AB$. Then it suffices to show 
        \[
        C^{\bullet}(\tilde M)\otimes_AB\simeq  C^{\bullet}(\tilde M\otimes_{\bfA_A}\bfA_B).
        \]

       By the proof of \cite[Theorem 5.1.22]{EG22}, we have
        \[
        C^{\bullet}(\tilde M)\simeq C^{\bullet}(\tilde M/T^n\frakM)
        \]
        for any $n\geq 1$, where $\frakM$ is a $(\varphi,\Gamma)$-stable lattice in $\tilde M$.  Write $\tilde M$ as the filtered colimit $\varinjlim_iT^{-i}\frakM$. As each $T^{-i}\frakM$ is $(\varphi,\Gamma)$-stable by \cite[Lemma 5.1.6]{EG22}, it suffices to prove
        \[
        T^i\frakM/T^n\frakM\otimes_AB\cong T^i\frakM_B/T^n\frakM_B
        \]
        for all $i\leq n$. Here $\frakM_B:=\frakM\otimes_{\bfA_A^+}\bfA_B^+$, which is a $(\varphi,\Gamma)$-stable lattice in $\tilde M\otimes_{\bfA_A}\bfA_B$. As $\frakM$ is finitely presented over $\bfA_A^+$, we can find a short exact sequence
        \[
        \frakM_1\to \frakM_2\to \frakM\to 0
        \]
        where $\frakM_1$ and $\frakM_2$ are finite free. It is clear that $T^i\frakM_j/T^n\frakM_j\otimes_AB\cong T^i\frakM_{j,B}/T^n\frakM_{j,B}$ for $j=1,2$.  So we are done.

        \item As $A_0$ is finite type over $\bF_p$, it is a Jacobson ring. In particular, finite type points of $\Spec(A_0)$ are the same as its closed points.

        We proceed in two steps to find such an open subset. First we look for an open subset such that for any finite type point $s$, the following map is injective
        \[
\tilde M^{\varphi=1,\gamma=1}\otimes_{A_0}k(s)\hookrightarrow (\tilde M\otimes_{A_0}k(s))^{\varphi=1,\gamma=1}.
        \]
        Note that $\tilde M^{\varphi=1,\gamma=1}$ is just $H^0(\Herr(\tilde M)$. Recall that the Herr complex of an \'etale $(\varphi,\Gamma)$-module $\calE$ is defined to be the complex
        \[
        \calE\xrightarrow{(\varphi-1,\gamma-1)} \calE\oplus \calE\xrightarrow{(\gamma-1)\oplus(1-\gamma)}\calE.
        \]

        Similarly, $(\tilde M\otimes_{A_0}k(s))^{\varphi=1,\gamma=1}$ is $H^0(\Herr(\tilde M\otimes_{A_0}k(s)))$. 
        
        We know that $\tilde M$ is finitely generated over $\bfA_{A_0}=(k_{\infty}\otimes_{\bF_p}A_0)((T))$ and $\bfA_{A_0}$ is finite free over $A_0((T))$. By \cite[Lemma 5.1.5]{EG21}, the ring $A_0((T))$ is also integral. Then by \cite[\href{https://stacks.math.columbia.edu/tag/0529}{Tag 0529}]{stacks-project}, there exists an open subset of $A_0((T))$ such that $\tilde M$ is flat over this open subset. We may assume this open subset is $\Spec(A_0((T))[\frac{1}{f(T)}])$ for some $f(T)\in A_0((T))$. Let $f\in A_0$ be the leading coefficient of $f(T)$. Then we have an inclusion $A_0((T))[\frac{1}{f(T)}]\subset A_0[\frac{1}{f}]((T))$. This implies the base change $\tilde M\otimes_{A_0((T))}A_0[\frac{1}{f}]((T))\cong \tilde M\otimes_{\bfA_{A_0}}\bfA_{A_0[\frac{1}{f}]}$ is flat over $A_0[\frac{1}{f}]((T))$. As the ring $A_0[\frac{1}{f}]((T))$ is flat over $A_0[\frac{1}{f}]$, the module $\tilde M\otimes_{\bfA_{A_0}}A_{\bfA_0[\frac{1}{f}]}$ is also flat over $A_0[\frac{1}{f}]$.

        Up to replacing $A_0$ with $A_0[\frac{1}{f}]$, we may assume $\tilde M$ is flat over $A_0$. This implies  \[
        \Herr(\tilde M)\otimes^{\bL}_{A_0}k(s)\simeq \Herr(\tilde M\otimes_{A_0}k(s)).
        \]

        By using the Kunneth spectral sequence $\Tor^{A_0}_q(H^p(\Herr(\tilde M),k(s)))\Longrightarrow H^{p-q}(\Herr(\tilde M)\otimes_{A_0}k(s)$ and \cite[\href{https://stacks.math.columbia.edu/tag/0529}{Tag 0529}]{stacks-project}, we can find a $g\in A_0$ such that \[H^0(\Herr(\tilde M))\otimes_{A_0}k(s)\otimes_{A_0}A_0[\frac{1}{g}]\hookrightarrow H^0(\Herr(\tilde M)\otimes_{A_0}k(s)\otimes_{A_0}A_0[\frac{1}{g}])\] for all finite type points $s\in \Spec(A_0[\frac{1}{g}])$. Using the facts that $\tilde M$ is flat over $A_0$ and $H^{\bullet}(\Herr(\tilde M))$ are finitely generated $A_0$-modules (note that the argument in the proof of \cite[Theorem 5.1.22(1)]{EG22} works for $\tilde M$), we can get 
        \[
        \Herr(\tilde M)\otimes_{A_0}A_0[\frac{1}{g}]\simeq \Herr(\tilde M\otimes_{\bfA_{A_0}}\bfA_{A_0[\frac{1}{g}]})
        \]
        by the same argument as in the proof of \cite[Theorem 5.1.22(2)]{EG22}. So we finish our first step.

        The second step is to show there exists an open subset of $\Spec(A_0)$ such that for all finite type points $s$ in this open subset, there is an injection
        \begin{equation}\label{step2}
       \tilde M\otimes_{A_0}k(s)\hookrightarrow H^1(  \RHom(T^*_{a'}\otimes^{\bL}_{\bfA_{A}}\bfA_{k(s)},M\otimes_{A}\bfA_{k(s)}))
        \end{equation}
        which induces an injection 
        \[
        (\tilde M\otimes_{A_0}k(s))^{\varphi=1,\Gamma=1}\hookrightarrow H^1(  \RHom(T^*_{a'}\otimes^{\bL}_{\bfA_{A}}\bfA_{k(s)},M\otimes_{A}\bfA_{k(s)}))^{\varphi=1,\Gamma=1}=\calO_a(M\otimes_{A_0}k(s))
        \]

        Note that 
        \[
       \tilde M\otimes_{A_0}k(s)\cong H^1(\RHom(T^*_{a'},M\otimes_A\bfA_A))\otimes_{\bfA_A}\bfA_{k(s)}.
       \]

 By \cite[\href{https://stacks.math.columbia.edu/tag/056V}{Tag 056V}]{stacks-project}, we may assume $A_0$ is regular and integral. Then $A_0$ has finite global dimension by \cite[\href{https://stacks.math.columbia.edu/tag/00OE}{Tag 00OE}]{stacks-project}. Moreover $\bfA_{A_0}$ will also be regular hence of finite gloabl dimension. In particular, $\bfA_{k(s)}$ has a resolution of projective $\bfA_{A_0}$-modules of length at most $\dim \bfA_{A_0}$. As $T^*_{a'}$ is an almost perfect complex and we are only concerned with $H^1$, i.e. $H^1(\RHom(T^*_{a'},M\otimes_A\bfA_A))$ and $H^1(  \RHom(T^*_{a'}\otimes^{\bL}_{\bfA_{A}}\bfA_{k(s)},M\otimes_{A}\bfA_{k(s)}))$, we may assume $T^*_{a'}$ is a perfect complex.

 Then 
 \[\RHom(T^*_{a'}\otimes^{\bL}_{\bfA_{A}}\bfA_{k(s)},M\otimes_{A}\bfA_{k(s)})\simeq (T^{*,\vee}_{a'}\otimes^{\bL}_AM)\otimes^{\bL}_{\bfA_{A_0}}\bfA_{k(s)}\simeq (T^{*,\vee}_{a'}\otimes^{\bL}_AM)\otimes^{\bL}_{A_0((T))}k(s)((T))\] and \[H^1(\RHom(T^*_{a'},M\otimes_A\bfA_A))\otimes_{\bfA_A}\bfA_{k(s)}\cong H^1(T^{*,\vee}_{a'}\otimes^{\bL}_AM)\otimes_{\bfA_{A_0}}\bfA_{k(s)}\cong H^1(T^{*,\vee}_{a'}\otimes^{\bL}_AM)\otimes_{A_0((T))}k(s)((T)).\]

        As $T^{*,\vee}_{a'}\otimes^{\bL}_{A}M$ is a perfect complex over $\bfA_{A_0}$, it is also a perfect complex over $A_0((T))$. Recall that $A_0((T))$ is integral. By some standard arguments about the base change of cohomology, we know there exists an $f(T)\in A_0((T))$ such that for any point $\tilde s\in \Spec(A_0((T))[\frac{1}{f(T)}])$ with residue field $k(\tilde s)$, there is an isomorphism
        \[
        H^1(T^{*,\vee}_{a'}\otimes^{\bL}_{A}M)\otimes_{A_0((T))}k(\tilde s)\cong H^1(T^{*,\vee}_{a'}\otimes^{\bL}_{A}M\otimes^{\bL}_{A_0((T))}k(\tilde s)).
        \]
 Now let $f$ be the leading coefficient of $f(T)$. For any maximal ideal $\frakm\subset A_0$, we see that if $f\notin \frakm$,  then $f(T)\notin \frakm A_0((T))$. Note that $(A_0/\frakm)((T))\cong A_0((T))/\frakm$. So for any finite type point $s\in \Spec(A_0[\frac{1}{f}])\subset \Spec(A_0)$ with maximal ideal $\frakm_s\subset A_0$, the corresponding point $\tilde s$ corresponding to $\frakm_s((T))$ lies in $\Spec(A_0((T))[\frac{1}{f(T)}])$ and has residue field $k(\tilde s)=k(s)((T))$. Hence the injection \ref{step2} holds for any $s\in \Spec(A_0[\frac{1}{f}])$.
        
Putting the two steps together, we then find an open dense subset of $\Spec(A_0)$ which satisfies the condition.

       \end{enumerate}
\end{proof}

\begin{prop}\label{main-2}
    $\calB_R^{\Gamma}$ is an algebraic space locally of finite presentation over $R$. In particular, it is a scheme.
\end{prop}
\begin{proof}
As mentioned earlier, together with the trivial deformation theory, Proposition \ref{obstruction-theory-valid} is sufficient to induce the openness of versality of the morphism $\calB^{\Gamma}_R\to \Spec(R)$. Combined with Proposition \ref{etale-gamma} and Lemma \ref{artin}, we conclude $\calB_R^{\Gamma}$ is an algebraic space locally of finite presentation over $R$. As $\calB^{\Gamma}_R\to \Spec(R)$ is a monomorphism, the last statement follows from \cite[\href{https://stacks.math.columbia.edu/tag/0418}{Tag 0418}]{stacks-project} and \cite[\href{https://stacks.math.columbia.edu/tag/0463}{Tag 0463}]{stacks-project}.
\end{proof}

\begin{thm}\label{main-1}
    The morphisms  $\calX_G\to \calX_{d}\times \calX_{t(d)-1}$ is representable in formal schemes and locally of finite presentation. In particular, $\calX_G$ is a formal algebraic stack locally of finite presentation over $\Spf(\bZ_p)$.
\end{thm}
\begin{proof}
    This follows from Proposition \ref{main-2}, Lemma \ref{surj1} and \ref{surj2}.
\end{proof}

\begin{rmk}
   It is not obvious that our construction of the obstruction theory for $\calB_R^{\Gamma}$ can apply to $\calB_R$. The key point is that we are not sure if the Frobenius invaiants of a finitely generated \'etale $(\varphi,\Gamma)$-module over $\bfA_{A_0}$ is a finite $A_0$-module, which seems necessary to define a reasonable construction theory.
\end{rmk}

\section{Derived Emerton--Gee stack for general groups}

In \cite{Min23}, we proved that the underlying stack of the derived stack of Laurent $F$-crystals on $(\calO_K)_{\Prism}$ is equivalent to the Emerton--Gee stack. Moreover, we proved that the derived stack of Laurent $F$-crystals is classical up to nilcompletion, i.e. when restricted to truncated animated rings, it is the \'etale sheafification of the left Kan extension of the Emerton--Gee stack along the inclusion from discrete commutative rings to animated rings. In this section, we aim to generalise these results to more general groups. 

\subsection{Classifyings stack $BG$}
Let $G$ still be a flat affine group scheme of finite type over $\bZ_p$. In order to define derived Laurent $F$-crystals with $G$-structure, we will use the classifying stack $BG$. 

\begin{dfn}\label{BG}
    Let $G^{\bullet}$ be the simplicial sheaves such that each $G^n$ is representable by the $n$-fold product of $G$. We define $BG$ to be the homotopy colimit (or geometric realisation) of $G^{\bullet}$ in the $\infty$-topos ${\rm Shv}(\bfAni(\Ring),\bfAni)$ with respect to the fpqc topology, where $\bfAni$ is the $\infty$-category of anima and $\bfAni(\Ring)$ is the $\infty$-category of animated rings.
\end{dfn}
If $R$ is a discrete animated ring, then $BG(R)$ is just the groupoid of $G$-torsors over $R$. As we have seen before, discrete $G$-torsors have several different descriptions. For general animated ring $R$, the natural way is to use Tannaka duality to study $BG(R)$. Unfortunately, it is not clear to us if there is a Tannaka duality in the setting of anmiated rings\footnote{There is a Tannka duality in the setting of $\bE_{\infty}$-rings (cf. \cite[Corollary 9.3.7.3]{Lur18}), whose existence suggests that there should be no Tannka duality for $p$-nilpotent animated rings.}. Therefore, new ideas are required to deal with $BG$. Our strategy is to use the cotangent complex to linearise everything.

\subsection{Derived stack of Laurent $F$-crystals with $G$-structure}
Our first task is to define a derived prestack of Laurent $F$-crystals with $G$-structure. In \cite{Min23}, the descent theorem \cite[Theorem 2.4]{Min23} due to Drinfeld and Mathew plays an important role in the study of such a derived prestack. However, it is not clear to us if the $\infty$-category of $G$-torsors satisfies the $I$-completely flat topology. This causes lots of trouble in studying Laurent $F$-crystals. For example, we do not know if \cite[Proposition 2.7]{BS23} still holds true for $G$-torsors in the animated setting.

To make the $I$-completely flat descent well-behaved, we consider the category of transveral prisms. Let us recall its definition.
\begin{dfn}[\cite{BL22}, Definition 2.1.3]
    A prism $(A,I)$ is called a transversal prism if $A/I$ has no $p$-torsion.
\end{dfn}
\begin{exam}
    The Breuil--Kisin prism $(\frakS=W(k)[[u]],(E(u)))$ is a transveral prism, where $E(u)$ is the Eisenstein polynomial of a fixed uniformiser $\pi$ of $\calO_K$ and $k$ is the residue field.
\end{exam}
Note that for a transveral prism $(A,I)$, the ring $A$ is $\bZ_p$-flat as it has no $p$-torsion. 
\begin{dfn}
    Let $(\calO_K)_{\Prism}^{\rm tr}$ denote the transversal absolute prismatic site, i.e. the objects in $(\calO_K)_{\Prism}^{\rm tr}$ are the transversal prisms in $(\calO_K)_{\Prism}$ and the topology is still the $I$-completely flat topology.
\end{dfn}

Let $R\in \Nilp_{\bZ_p}$ and $(A,I=(d))$ be a transveral prism. Then by \cite[Lemma 3.28]{Min23}, for any $n\geq 0$, we have $\pi_n(A\widehat\otimes^{\bL}_{\bZ_p}R[\frac{1}{I}])\cong A\widehat\otimes_{\bZ_p}\pi_n(R)[\frac{1}{I}]$, which makes $I$-completely flat descent behave well as we will see in the proof of Lemma \ref{flat descent}. Now we give the following definition of Laurent $F$-crystals with $G$-structure.

\begin{dfn}[Derived prestack Laurent $F$-crystals with $G$ structure]
Let $\frakX_G:\textbf{Nilp}_{\bZ_p}\to \textbf{Ani}$ denote the functor sending an animated ring $R\in  \textbf{Nilp}_{\bZ_p}$ to the anima of Laurent $F$-crystals with $G$-structure 
\[
\frakX_G(R):=\varprojlim_{(A,I)\in (\calO_K)_{\Prism}^{\rm tr}}BG(A\widehat\otimes^{\bL}R[\frac{1}{I}])^{\varphi=1}
\]
where 
\begin{equation*}
    \xymatrix{
    BG(A\widehat\otimes^{\bL}R[\frac{1}{I}])^{\varphi=1}=Eq(BG(A\widehat\otimes^{\bL}R[\frac{1}{I}])\ar@<.5ex>[r]^-{id}\ar@<-.5ex>[r]_-{\varphi^*}&BG(A\widehat\otimes^{\bL}R[\frac{1}{I}])).
    }
\end{equation*}
    
\end{dfn}

\begin{lem}\label{flat descent}
   Let $(A,I)\to (B,I)$ be a flat cover in $(\calO_K)_{\Prism}^{\rm tr}$ and $(B^{\bullet},I)$ be the associated \v Cech nerve. Then for any truncated animated ring $R\in \textbf{Nilp}^{<\infty}_{\bZ_p}$, we have
   \[
   BG(A\widehat\otimes^{\bL}R[\frac{1}{I}])^{\varphi=1}\simeq \varprojlim_iBG(B^i\widehat\otimes^{\bL}R[\frac{1}{I}])^{\varphi=1}.
   \]
\end{lem}
\begin{proof}
    We prove by induction on the truncation degree of $R$. Assume $R$ is discrete, i.e. $R\cong \pi_0(R)$. Then by Tannaka duality, we have
    \[
    BG(A\widehat\otimes^{\bL}R[\frac{1}{I}])^{\varphi=1}\simeq (\Fun^{{\rm ex},\otimes}(\Rep(G),\Vect(A\widehat\otimes^{\bL}R[\frac{1}{I}])))^{\varphi=1}.
    \]
    The latter is also equivalent to $\Fun^{{\rm ex},\otimes}(\Rep(G),\Vect(A\widehat\otimes^{\bL}R[\frac{1}{I}])^{\varphi=1})$. This is also true by replacing $A$ with $B^i$ for all $i$.

    By the descent theorem \cite[Theorem 2.4]{Min23} of Drinfeld--Mathew, we know 
    \[
    \Vect(A\widehat\otimes^{\bL}R[\frac{1}{I}])\simeq \varprojlim_i\Vect(B^i\widehat\otimes^{\bL}R[\frac{1}{I}])
    \]
    which induces
    \[
    \Vect(A\widehat\otimes^{\bL}R[\frac{1}{I}])^{\varphi=1}\simeq \varprojlim_i\Vect(B^i\widehat\otimes^{\bL}R[\frac{1}{I}])^{\varphi=1}.
    \]
    As all the maps $ \Vect(A\widehat\otimes^{\bL}R[\frac{1}{I}])^{\varphi=1}\to  \Vect(B^i\widehat\otimes^{\bL}R[\frac{1}{I}])^{\varphi=1}$ and $ \Vect(B^i\widehat\otimes^{\bL}R[\frac{1}{I}])^{\varphi=1}\to  \Vect(B^j\widehat\otimes^{\bL}R[\frac{1}{I}])^{\varphi=1}$ are exact tensor functors, we get
    \[
    \Fun^{{\rm ex},\otimes}(\Rep(G),\Vect(A\widehat\otimes^{\bL}R[\frac{1}{I}])^{\varphi=1})\simeq\varprojlim_i\Fun^{{\rm ex},\otimes}(\Rep(G),\Vect(B^i\widehat\otimes^{\bL}R[\frac{1}{I}])^{\varphi=1}).
    \]
    This proves the lemma for discrete rings.

    Assume the lemma holds true for all $R\in \textbf{Nilp}_{\bZ_p}^{\leq n}$. Let $R\in\textbf{Nilp}_{\bZ_p}^{\leq n+1}$. Recall that we have the pullback diagram
    \begin{equation}\label{truncation}
        \xymatrix{
        \tau_{\leq {n+1}}R\ar[r]\ar[d]&\tau_{\leq n}R\ar[d]\\
        \tau_{\leq n}R\ar[r]&\tau_{\leq n}R\oplus \pi_{n+1}(R)[n+2].
        }
    \end{equation}

    As we have mentioned, $A\widehat\otimes^{\bL}R[\frac{1}{I}]$ is also in $\textbf{Nilp}_{\bZ_p}^{\leq n+1}$ and $\pi_{n+1}(A\widehat\otimes^{\bL}R[\frac{1}{I}])=A\widehat\otimes\pi_{n+1}(R)[\frac{1}{I}]$ by \cite[Lemma 3.28]{Min23}.

    Let $e:\Spec(A\widehat\otimes^{\bL}\tau_{\leq n}R[\frac{1}{I}])\to BG$ be a $G$-torsor over $A\widehat\otimes^{\bL}\tau_{\leq n}R[\frac{1}{I}]$ and $e_i:\Spec(B^i\widehat\otimes^{\bL}\tau_{\leq n}R[\frac{1}{I}])\to BG$ be the corresponding $G$-torsor over $A\widehat\otimes^{\bL}\tau_{\leq n}R[\frac{1}{I}]$ for each $i$.

By the definition of cotangent complex, we have pullback diagrams
\begin{equation}
    \xymatrix{
    \Map(T^*_e(BG),A\widehat\otimes\pi_{n+1}(R)[\frac{1}{I}][n+2])\ar[r]\ar[d]& {*}_e\ar[d]\\
    BG(A\widehat\otimes^{\bL} R)[\frac{1}{I}])\ar[r]& BG(A\widehat\otimes^{\bL} \tau_{\leq n}R)[\frac{1}{I}]).
    }
\end{equation}
and
\begin{equation}
    \xymatrix{
    \Map(T^*_{e_i}(BG),B^i\widehat\otimes\pi_{n+1}(R)[\frac{1}{I}][n+2])\ar[r]\ar[d]& {*}_{e_i}\ar[d]\\
    BG(B^i\widehat\otimes^{\bL} R)[\frac{1}{I}])\ar[r]& BG(B^i\widehat\otimes^{\bL} \tau_{\leq n}R)[\frac{1}{I}]).
    }
\end{equation}

Putting them together, we get a big diagram consisting of pullback diagrams
\begin{equation*}
    \xymatrix{
    \Map(T^*_e(BG),A\widehat\otimes\pi_{n+1}(R)[\frac{1}{I}][n+2])\ar[r]\ar[d]& \varprojlim_i\Map(T^*_{e_i}(BG),B^i\widehat\otimes\pi_{n+1}(R)[\frac{1}{I}][n+2])\ar[r]\ar[d] & {*}_e\ar[d]\\
    \calY_A\ar[r]\ar[d] & \varprojlim_i\calY_{B^i}\ar[r]\ar[d] & [*_e/{\Aut(*_e)}]\ar[d]\\
    BG(A\widehat\otimes^{\bL} R)[\frac{1}{I}])\ar[r]& \varprojlim_iBG(B^i\widehat\otimes^{\bL} R)[\frac{1}{I}])\ar[r]& BG(A\widehat\otimes^{\bL} \tau_{\leq n}R)[\frac{1}{I}])
    }
\end{equation*}
where $\calY_A$ (resp. $\calY_{B^i}$) is the pullback of $BG(A\widehat\otimes^{\bL} R)[\frac{1}{I}])$ (resp. $BG(B^i\widehat\otimes^{\bL} R)[\frac{1}{I}])$) along the inclusion $[*_e/{\Aut(*_e)}]\to BG(A\widehat\otimes^{\bL} \tau_{\leq n}R)[\frac{1}{I}])$ (resp. $[*_{e_i}/{\Aut(*_{e_i})}]\to BG(B^i\widehat\otimes^{\bL} \tau_{\leq n}R)[\frac{1}{I}])$.

Note that $T^*_e(BG)$ (resp. $T^*_{e_i}(BG)$) is a finite projective module over $A\widehat\otimes^{\bL} \tau_{\leq n}R[\frac{1}{I}]$ (resp. $B^i\widehat\otimes^{\bL}\tau_{\leq n}R[\frac{1}{I}]$) shifted by $[-1]$. By \cite[Theorem 2.4]{Min23} of Drinfeld--Mathew, we see
\[
T^*_{e_i}(BG)\simeq T^*_e(BG)\otimes_{A\widehat\otimes^{\bL} \tau_{\leq n}R[\frac{1}{I}]}B^i\widehat\otimes^{\bL}\tau_{\leq n}R[\frac{1}{I}].
\]

This implies
\begin{equation*}
    \begin{split}
        \varprojlim_i\Map(T^*_{e_i}(BG),B^i\widehat\otimes\pi_{n+1}(R)[\frac{1}{I}][n+2]) & \simeq \varprojlim_i\Map(T^*_e(BG),B^i\widehat\otimes\pi_{n+1}(R)[\frac{1}{I}][n+2])\\
        &\simeq \Map(T^*_e(BG),\varprojlim_iB^i\widehat\otimes\pi_{n+1}(R)[\frac{1}{I}][n+2]).
    \end{split}
\end{equation*}
On the one hand, we have an exact triangle for each $i$
\[
B^i\widehat\otimes^{\bL}\tau_{\leq n}(R)[\frac{1}{I}]\to B^i\widehat\otimes^{\bL}R[\frac{1}{I}]\to B^i\widehat\otimes^{\bL}\pi_{n+1}(R)[\frac{1}{I}][n+2].
\]
The same holds by replacing $B^i$ with $A$. On the other hand, we know that
\[
A\widehat\otimes^{\bL}\tau_{\leq n}(R)[\frac{1}{I}]\simeq \varprojlim_i B^i\widehat\otimes^{\bL}\tau_{\leq n}(R)[\frac{1}{I}]
\]
and
\[
A\widehat\otimes^{\bL}R[\frac{1}{I}]\simeq \varprojlim_i B^i\widehat\otimes^{\bL}R[\frac{1}{I}].
\]
So we get 
\[
A\widehat\otimes^{\bL}\pi_{n+1}(R)[\frac{1}{I}][n+2]\simeq \varprojlim_i B^i\widehat\otimes^{\bL}\pi_{n+1}(R)[\frac{1}{I}][n+2],
\]
which then implies
\[
\Map(T^*_e(BG),A\widehat\otimes\pi_{n+1}(R)[\frac{1}{I}][n+2])\xrightarrow{\simeq} \varprojlim_i\Map(T^*_{e_i}(BG),B^i\widehat\otimes\pi_{n+1}(R)[\frac{1}{I}][n+2]).
\]

In particular, this means the two $\infty$-groupoid objects obtained by pulling back $\calY_A$ and $\varprojlim_i\calY_{B^i}$ along $*_e\to [*_e/\Aut{*_e}]$ are equivalent. We then conclude
\[
\calY_A\xrightarrow{\simeq}\varprojlim_i\calY_{B^i}
\]
by the fact that colimit in the $\infty$-topos $\bfAni$ is universal (cf. \cite[Theorem 6.1.0.6]{Lur09a}).

Write $H=\pi_0(BG(A\widehat\otimes^{\bL} \tau_{\leq n}R)[\frac{1}{I}]))$. As $BG(A\widehat\otimes^{\bL} \tau_{\leq n}R)[\frac{1}{I}])=\bigsqcup_{*_e\in H}[*_e/\Aut(*_e)]$, we see
\[
BG(A\widehat\otimes^{\bL} R)[\frac{1}{I}])\xrightarrow{\simeq} \varprojlim_iBG(B^i\widehat\otimes^{\bL} R)[\frac{1}{I}]).
\]
We are done.

\end{proof}

\begin{dfn}
    We define the nilcompletion of $\frakX_G$ to be $\frakX_G^{\nil}:\textbf{Nilp}_{\bZ_p}\to \textbf{Ani}$ sending each $R\in \textbf{Nilp}_{\bZ_p}$ to $\frakX_G^{\nil}(R):=\varprojlim_n\frakX_G(\tau_{\leq n}R)$.
\end{dfn}

\begin{prop}
    The derived prestack $\frakX_G^{\nil}$ is a derived stack for the flat topology.
\end{prop}
\begin{proof}
    Let $R\to S$ be a faithfully flat map in $\textbf{Nilp}_{\bZ_p}$. Let $S^{\bullet}$ be the \v Cech nerve associated to this cover. Then we need to prove
    \[
    \frakX_G(R)\xrightarrow{\simeq}\varprojlim_i\frakX_G(S^i).
    \]

Let $(\frakS,(E))$ be the Breuil--Kisin prism, which is a cover of the topos ${\rm Shv}(((\calO_K)_{\Prism}^{\rm tr})$, and $(\frakS^{\bullet},(E))$ be the associated \v Cech nerve in ${\rm Shv}(((\calO_K)_{\Prism}^{\rm tr})$. By Lemma \ref{flat descent}, we then have
\[
\frakX_G(A)\xrightarrow{\simeq}\varprojlim_nBG(\frakS^n\widehat\otimes^{\bL}A[\frac{1}{E}])^{\varphi=1}
\]
for any $A\in \textbf{Nilp}^{<\infty}_{\bZ_p}$. Note that for each $i$ and $j$, the map $\frakS^n\otimes^{\bL}\tau_{\leq j}R\to\frakS^i\otimes^{\bL}\tau_{\leq j}S $ is $E$-completely faithfully flat. By the same argument as Lemma \ref{flat descent}, we can show
\[
BG(\frakS^n\widehat\otimes^{\bL}\tau_{\leq j}R[\frac{1}{E}])\simeq \varprojlim_iBG(\frakS^n\widehat\otimes^{\bL}\tau_{\leq j}S^i[\frac{1}{E}]).
\]
This shows
\[
\frakX_G(\tau_{\leq j}R)\simeq \varprojlim_nBG(\frakS^n\widehat\otimes^{\bL}\tau_{\leq j}R[\frac{1}{E}])^{\varphi=1}\simeq \varprojlim_n\varprojlim_iBG(\frakS^n\widehat\otimes^{\bL}\tau_{\leq j}S^i[\frac{1}{E}])^{\varphi=1}\simeq \varprojlim_i\frakX_G(\tau_{\leq j}S^i),
\]
which implies $\frakX^{\nil}_G(R)\simeq \varprojlim_i\frakX_G^{\nil}(S^i)$. So we are done.
    
\end{proof}

\subsection{The underlying classical stack $^{\rm cl}\frakX_G$}
In this subsection, we look at the underlying stack $^{\rm cl}\frakX_G$ of $\frakX_G$ and compare it to the Emerton--Gee stack $\calX_G$. Note that for any discrete ring $R$, we have
\begin{equation*}
    \begin{split}
        ^{\rm cl}\frakX_G(R)& \simeq \varprojlim_nBG(\frakS^n\widehat\otimes R[\frac{1}{E}])^{\varphi=1}\\
        &\simeq \varprojlim_n\Fun^{\rm ex,\otimes}(\Rep(G),\Vect(\frakS^n\widehat\otimes R[\frac{1}{E}])^{\varphi=1})\\
        &\simeq \Fun^{\rm ex, \otimes}(\Rep(G),\Vect((\calO_K)_{\Prism},\calO_{\Prism,R}[\frac{1}{\calI_{\Prism}}])^{\varphi=1})
    \end{split}
\end{equation*}
by Tannaka duality and the exact $\otimes$-equivalence $\Vect((\calO_K)_{\Prism},\calO_{\Prism,R}[\frac{1}{\calI_{\Prism}}])^{\varphi=1}\simeq \varprojlim_n \Vect(\frakS^n\widehat\otimes R[\frac{1}{E}])^{\varphi=1}$.

Recall the following result from \cite{Min23}.

\begin{prop}[\cite{Min23}, Corollary 2.27]\label{finite type}
    Let $R$ be a finite type $\bZ/p^a$-algebra. Then there is an equivalence of categories
    \[
    \Vect((\calO_K)_{\Prism},\calO_{\Prism,R}[\frac{1}{\calI_{\Prism}}])^{\varphi=1}\simeq \Mod^{\varphi,\Gamma}(\bfA_R).
    \]
\end{prop}
It is easy to see that the above equivalence is an exact tensorial equivalence. This immediately implies that $^{\rm cl}\frakX_G(R)\simeq \calX_G(R)$ for any finite type $\bZ/p^a$-algebra $R$. We claim that this is true for all $p$-nilpotent rings.
\begin{lem}\label{all}
    Let $R$ be any $p$-nilpotent ring. Then there is an exact tensorial equivalence of categories
    \[
    \Vect((\calO_K)_{\Prism},\calO_{\Prism,R}[\frac{1}{\calI_{\Prism}}])^{\varphi=1}\simeq \Mod^{\varphi,\Gamma}(\bfA_R).
    \]
\end{lem}
\begin{proof}
    Write $R=\varinjlim_iR_i$ where $R_i$ goes through all finite type $\bZ_p$-subalgebras of $R$. We claim that 
    \[
     \varinjlim_i\Vect((\calO_K)_{\Prism},\calO_{\Prism,R_i}[\frac{1}{\calI_{\Prism}}])^{\varphi=1}\xrightarrow{\simeq}\Vect((\calO_K)_{\Prism},\calO_{\Prism,R}[\frac{1}{\calI_{\Prism}}])^{\varphi=1}
    \]
    and
    \[
     \varinjlim_i\Mod^{\varphi,\Gamma}(\bfA_{R_i})\xrightarrow{\simeq}\Mod^{\varphi,\Gamma}(\bfA_{R}).
    \]
    Note that both equivalences hold once we take the underlying groupoids. In particular, the above natural functors are essentially surjective. It remains to prove the fully faithfulness.

    Let $\bM_i\in \Vect((\calO_K)_{\Prism},\calO_{\Prism,R_i}[\frac{1}{\calI_{\Prism}}])^{\varphi=1}$ and $\bN_j\in \Vect((\calO_K)_{\Prism},\calO_{\Prism,R_j}[\frac{1}{\calI_{\Prism}}])^{\varphi=1}$. Then we know \[
    \Hom(\bM_i,\bN_j)=\varinjlim_{i,j\to k}\Hom(\bM_{ik},\bN_{jk})\]
    where $\bM_{ik},\bN_{jk}$ are the base change of $\bM_i$,$\bN_j$ along $\calO_{\Prism,R_i}[\frac{1}{\calI_{\Prism}}]\to \calO_{\Prism,R_k}[\frac{1}{\calI_{\Prism}}]$, $\calO_{\Prism,R_j}[\frac{1}{\calI_{\Prism}}]\to \calO_{\Prism,R_k}[\frac{1}{\calI_{\Prism}}]$ respectively.

    Note that 
    \[
    \Hom(\bM_{ik},\bN_{jk})\simeq \varprojlim_{n\leq 2} \Hom(\bM_{ik}(\frakS^n,E),\bN_{jk}(\frakS^n,E))^{\varphi=1}.
    \]
    Let $\bM$, $\bN$ be the base change of $\bM_i$, $\bN_j$ along $\calO_{\Prism,R_i}[\frac{1}{\calI_{\Prism}}]\to \calO_{\Prism,R}[\frac{1}{\calI_{\Prism}}]$, $\calO_{\Prism,R_j}[\frac{1}{\calI_{\Prism}}]\to \calO_{\Prism,R}[\frac{1}{\calI_{\Prism}}]$ respectively. By the same argument as the Proof of \cite[Lemma 2.29]{Min23}, i.e. using \cite[Lemma 2.24]{Min23}, we can get
    \[
    \varinjlim_{i,j\to k}\Hom(\bM_{ik},\bN_{jk})\simeq \Hom(\bM,\bN),
    \]
    which then implies
    \[
     \varinjlim_i\Vect((\calO_K)_{\Prism},\calO_{\Prism,R_i}[\frac{1}{\calI_{\Prism}}])^{\varphi=1}\xrightarrow{\simeq}\Vect((\calO_K)_{\Prism},\calO_{\Prism,R}[\frac{1}{\calI_{\Prism}}])^{\varphi=1}
    \]
    For the part of \'etale $(\varphi,\Gamma)$-modules, let $K^{\rm basic}:=K\cap K_0(\zeta_{p^{\infty}})$ (cf. \cite[Definition 3.2.3]{EG22}), where $K_0$ is the maximal unramified extension of $\bQ_p$ inside $K$. An \'etale $(\varphi,\Gamma)$-module $M$ of rank $n$ over $\bfA_{R}$ can be regarded as an \'etale $(\varphi,\Gamma)$-module of rank $n[K:K^{\rm basic}]$ over $\bfA_{K^{\rm basic},R}$.

    Let $M_i\in \Mod^{\varphi,\Gamma}(\bfA_{R_i})$ and $N_j\in \Mod^{\varphi,\Gamma}(\bfA_{R_j})$. It suffices to prove
    \[
    \varinjlim_{i,j\to k}(M_{ik}^{\vee}\otimes N_{jk})^{\varphi=1,\Gamma=1}\cong (M^\vee\otimes N)^{\varphi=1,\Gamma=1}
    \]
    where $M_{ik},N_{jk}$ are the base change of $M_i,N_j$ along $\bfA_{R_i}\to \bfA_{R_j}$ and similar for $M,N$. We claim 
    \[
    \varinjlim_{i,j\to k}(M_{ik}^{\vee}\otimes N_{jk})^{\varphi=1}\cong (M^\vee\otimes N)^{\varphi=1}
    \]
    which then imply the above isomorphism by taking $\gamma$-invariants.

    Now without loss of generality, we may assume $M_i,N_j$ are both finite free and view them as \'etale $\varphi$-modules over $\bfA^+_{K^{\rm basic},R_i}$, $\bfA^+_{K^{\rm basic},R_j}$ respectively. As $\bfA^+_{K^{\rm basic}}$ is $\varphi$-stable, we can always find $\varphi$-stable lattices of $M_i, N_j$. This enables us to use \cite[Lemma 2.24]{Min23} again to conclude $\varinjlim_{i,j\to k}(M_{ik}^{\vee}\otimes N_{jk})^{\varphi=1}\cong (M^\vee\otimes N)^{\varphi=1}$.

    Using Lemma \ref{finite type}, we get
    \[
    \Vect((\calO_K)_{\Prism},\calO_{\Prism,R}[\frac{1}{\calI_{\Prism}}])^{\varphi=1}\simeq \Mod^{\varphi,\Gamma}(\bfA_{R}),
    \]
    which is clearly an exact tensorial equivalence.

\end{proof}

\begin{thm}\label{classical}
    There is an equivalence of stacks
    \[
    \calX_G\simeq {^{\rm cl}\frakX_G}.
    \]
\end{thm}
\begin{proof}
    This directly follows from Lemma \ref{all}.
\end{proof}

\subsection{Classicality of $\frakX_G$}
In this subsection, we plan to investigate the classicality of the derived stack $\frakX_G$, i.e. compare $\frakX_G^{\nil}$ to $(\Lan\calX_G)^{\#,\rm nil}$. The basic strategy is the same as \cite{Min23}, i.e. use the following criterion.

\begin{prop}[{\cite[Chapter 1, Proposition 8.3.2]{GR17}}]\label{GR-criterion}
    Let $\calF_1\to \calF_2$ be a map of derived prestacks admitting a deformation theory. Suppose there exists a commutative diagram,
    \begin{equation*}
        \xymatrix@=0.6cm{
        & ^{\rm cl}\calF_0\ar[ld]^{g_1}\ar[rd]^{g_2}&\\
        \calF_1\ar[rr]^{f}&&\calF_2
        }
    \end{equation*}
where $g_1,g_2$ are nilpotent embeddings, and $^{\rm cl}\calF_{0}$ is a classical prestack. Suppose also that for any discrete commutative ring $R\in {Nilp}_{\bZ_p}$, and a map $x_0\in {^{\rm cl}\calF_{0}}(R)$ and $x_i=g_i\circ x_0$ where $i=1,2$, the induced map
    \[
    T^*_{x_2}(\calF_2)\to T^*_{x_1}(\calF_1)
    \]
    is an isomorphism. Then $f$ is an isomorphism.
\end{prop}

Guided by the above proposition, we proceed in four steps.
\begin{enumerate}
    \item[\textbf{Step 1:}]Show both $\frakX_G^{\nil}$ and $(\Lan\calX_G)^{\#,\rm nil}$ have a deformation theory. Then by \cite[Chapter 1, Proposition 8.3.2]{GR17}, we are reduced to comparing their pro-cotangent complexes at all classical points.
    \item[\textbf{Step 2:}] Show both of them are locally almost of finite type. This reduces the comparison of pro-cotangent complexes to all classical finite type points.
    \item[\textbf{Step 3:}] Show their pro-cotangent complexes at classical finite type points are indeed cotangent complexes, which reduces the comparison further to points valued in finite fields.
    \item[\textbf{Step 4:}] Relate both of them to the derived representations and use the result on the local Galois deformation ring to compare the cotangent complexes at points valued in finite fields.
\end{enumerate}

 We refer to \cite[Section 3.1]{Min23} and \cite[Chapter 1]{GR17} for more information about the derived deformation theory. Now let us begin with \textbf{Step 1}, i.e.  show that both $\frakX_G^{\nil}$ and $(\Lan\calX_G)^{\#,\rm nil}$ admit a deformation theory.

\begin{prop}
  The derived stack   $\frakX_G^{\nil}$ admits a deformation theory.
\end{prop}
\begin{proof}
    Note that $BG$ is a derived Artin stack locally of finite presentation over $\Spec(\bZ_p)$. Then the proof of \cite[Proposition 3.25]{Min23} still works here.
\end{proof}

\begin{prop}\label{EG-def}
    The derived stack $(\Lan\calX_G)^{\#,\rm nil}$ admits a deformation theory.
\end{prop}
Note that we do not know whether $\calX_G$ is an ind-Artin stack or not. So the proof of \cite[Theorem 3.29]{Min23} can not be used here. Instead, we use the following lemma.
\begin{lem}[\cite{GR19} Chapter 1, Lemma 7.4.3]\label{def-cri}
Let $f:\calY\to \calX$ be a map in $\PreStk$. Assume that:
\begin{enumerate}
    \item $\calX$ satisfies \'etale descent;
    \item $f$ is \'etale-locally surjective;
    \item $\calY$ admits deformation theory;
    \item $\calY$ admits deformation theory relative to $\calX$;
    \item $\calY$ is formally smooth over $\calX$.
\end{enumerate}
Then $\calX$ admits deformation theory.
    
\end{lem}

\begin{proof}[Proof of Proposition \ref{EG-def}]
    As $\calX_G$ is a locally Noetherian formal algebraic stack, there is a smooth covering $\calY=\bigsqcup_i\Spf(A_i)\to \calX_G$ where each $\Spf(A_i)$ is a Noetherian affine formal scheme. 

   It is clear that the sheafified left Kan extension $(\Lan\calY)^{\#,\rm nil}$ admits a deformation theory as left Kan extension commutes with colimits. By \cite[Proposition 4.4.3]{GR19} and its proof, $(\Lan\calY)^{\#,\rm nil}\to (\Lan\calX_G)^{\#,\rm nil}$ is still a smooth covering. So all the conditions in Lemma \ref{def-cri} are satisfied. We then conclude $(\Lan\calX_G)^{\#,\rm nil}$ admits a deformation theory.
\end{proof}
\begin{rmk}
    We can see from the above proof that the key input is that the Emerton--Gee stack is a formal algebraic stack.
\end{rmk}

Next we proceed to \textbf{Step 2}, i.e. deal with the issue of being locally almost of finite type. 

\begin{prop}
    The derived stack $(\Lan\calX_G)^{\#,\rm nil}$ is locally almost of finite type.
\end{prop}
\begin{proof}
    This follows from the same proof of \cite[Proposition 3.33]{Min23} as the only necessary input is that $\calX_G$ is locally of finite presentation over $\Spf(\bZ_p)$, which we have proved.
\end{proof}

To prove this for $\frakX_G^{\nil}$ is a bit more complicated as the Tannaka duality is not available. We will prove by induction and use the cotangent complex.

\begin{prop}
    The derived stack $\frakX_G^{\nil}$ is locally almost of finite type.
\end{prop}

\begin{proof}
     By definition, this means we need to prove $\frakX_G^{\rm nil}$ preserves filtered colimit when restricted to $\textbf{Nilp}_{\bZ_p}^{\leq n}$ for each $n\geq 0$.    Write $R=\varinjlim_iR_i\in\textbf{Nilp}_{\bZ_p}^{\leq n} $ as a filtered colimit. We then need to prove $\varinjlim_i\frakX_G(R_i)\simeq\frakX_G(R)$. We argue by induction. When $R$ is discrete, this is true by Corollary \ref{classical}. 
     
     Now assume this is true for all animated rings in $\textbf{Nilp}_{\bZ_p}^{\leq n}$ for some $n\geq 0$. Let $R\in \textbf{Nilp}_{\bZ_p}^{\leq n+1}$. Let $e:\Spec(\tau_{\leq n}R)\to \frakX_G$. Then by the definition of pro-cotangent complex, we have
     \begin{equation*}
         \xymatrix{
         \Map(T^*_e(\frakX_G),\pi_{n+1}(R)[n+2])\ar[r]\ar[d]& {*}_e\ar[d]\\
         \frakX_G(\tau_{\leq n}R\oplus\pi_{n+1}(R)[n+2])\ar[r]&\frakX_G(\tau_{\leq n}R).
         }
     \end{equation*}
     We have similar diagrams for all $R_i$. Note that $\tau_{\leq n}R\simeq \varinjlim_i\tau_{\leq n}R_i$ and $\pi_{n+1}(R)\cong \varinjlim_i\pi_{n+1}(R_i)$.

     using the Breuil--Kisin prism $(\frakS,(E))$ and the associated \v Cech nerve $(\frakS^{\bullet},(E))$, we see that
     \[
     \Map(T^*_e(\frakX_G),\pi_{n+1}(R)[n+2])\simeq \varprojlim_j\Map(T^*_{e_j}(BG),\frakS^j\widehat\otimes\pi_{n+1}(R)[\frac{1}{E}][n+2])^{\varphi=1}
     \]
     where $e_j:\Spec(\frakS^j\widehat\otimes^{\bL}\tau_{\leq n}R[\frac{1}{E}])\to BG$ is the point corresponding to $e:\Spec(\tau_{\leq n}R)\to \frakX_G$.

     Again, all these hold similarly for all $R_i$. By the same argument as the proof of Lemma \ref{all}, it suffices to prove
     \[
      \Map(T^*_e(\frakX_G),\pi_{n+1}(R)[n+2])\simeq \varinjlim_i \Map(T^*_{e^i}(\frakX_G),\pi_{n+1}(R_i)[n+2])
     \]
     where the filtered colimit goes through all $i$ such that the point $e$ comes from $e^i\in \frakX_G(R_i)$.

     Hence it suffices to prove for each $j$, we have
     \begin{equation}\label{ij}
     \Map(T^*_{e_j}(BG),\frakS^j\widehat\otimes\pi_{n+1}(R)[\frac{1}{E}][n+2])^{\varphi=1}\simeq \varinjlim_i\Map(T^*_{e^i_j}(BG),\frakS^j\widehat\otimes\pi_{n+1}(R_i)[\frac{1}{E}][n+2])^{\varphi=1}.
     \end{equation}

     Note that $T^*_{e_j}(BG)$ (resp. $T^*_{e^i_j}(BG)$) is a finite projective module over $\frakS^j\widehat\otimes^{\bL}\tau_{\leq n}R[\frac{1}{E}]$ (resp. $\frakS^j\widehat\otimes^{\bL}\tau_{\leq n}R_i[\frac{1}{E}]$) and $T^*_{e^i_j}(BG)\otimes \frakS^j\widehat\otimes^{\bL}(\tau_{\leq n}R[\frac{1}{E}])\simeq T^*_{e_j}(BG)$.

As $R$ is the filtered colimit $\varinjlim_iR_i$, we have \[\frakS^j\widehat\otimes \tau_{\leq n}R\simeq (\varinjlim_i\frakS^j\widehat\otimes^{\bL}\tau_{\leq n}R_i)^{\wedge}_{E}\] and \[\frakS^j\widehat\otimes \pi_{n+1}(R)\simeq (\varinjlim_i\frakS^j\widehat\otimes^{\bL}\pi_{n+1}(R_i))^{\wedge}_{E}.\]

Without loss of generality, we may assume $T^*_{e_j}(BG), T^*_{e^i_j}(BG)$ are all finite free (cf. \cite[Lemma 5.2.14]{EG21}). In this case, we can find $\varphi$-stable finite free lattice in $T^*_{e^i_j}(BG)$ as in the proof of \cite[Proposition 3.31]{Min23}. Write $T^*_{e^i_j}(BG)^{\circ}$ for the finite free lattice. Then we can get
\[
(\varinjlim_i (T^*_{e^i_j}(BG)^{\circ})^{\vee}\otimes (\frakS^j\widehat\otimes\pi_{n+1}(R_i)[n+2]))^{\wedge}_E\simeq (T^*_{e_j}(BG)^{\circ})^{\vee}\otimes (\frakS^j\widehat\otimes\pi_{n+1}(R)[n+2]).
\]
Then \ref{ij} follows from \cite[Lemma 2.24]{Min23}.
\end{proof}

Now using \cite[Proposition 3.30]{Min23}, we are able to reduce the comparison of $\frakX_G^{\nil}$ and $(\Lan\calX_G)^{\#,\rm nil}$ to the comparison of their pro-cotangent complexes $T^*_e(\frakX_G^{\nil})$ and  $T^*_e((\Lan\calX_G)^{\#,\rm nil})$ at all classical points $e:\Spec(R)\to {^{\rm cl}\frakX_G}\simeq \calX_G$ with $R$ being finite type over $\bZ_p$.

In order to reduce further to the case where $R$ is a finite field, we need to show both $T^*_e(\frakX_G^{\nil})$ and  $T^*_e((\Lan\calX_G)^{\#,\rm nil})$ are indeed cotangent complexes as in \cite[Proposition 3.38 and 3.40]{Min23}, which is our \textbf{Step 3}.

\begin{prop}\label{cot-1}
    Let $e:\Spec(R)\to (\Lan\calX_G)^{\#,\rm nil}$ be a classical point with $R$ being a finite type algebra over $\bZ/p^a$ for some $a$. Then the pro-cotangent complex $T^*_{e}((\Lan\calX_G)^{\#,\rm nil})$ is a complex of $R$-modules.
\end{prop}
\begin{proof}
    The proof is actually the same as that of \cite[Proposition 3.38]{Min23}. The only difference is that we do not know $\calX_G$ is quasi-compact or not. So we only have a smooth covering $\calY=\bigsqcup_i\Spf(A_i)\to \calX_G$ where each $\Spf(A_i)$ is a Noetherian affine formal scheme. But for any point $e:\Spec(R)\to \calX_G$, it can factor through $\calY=\bigsqcup_i\Spf(A_i)\to \calX_G$ \'etale locally. And any map $\Spec(A)\to \calY=\bigsqcup_i\Spf(A_i)$ must factor through some finite unions of $\Spf(A_i)$. So we are indeed in the same case as \cite[Proposition 3.38]{Min23}.
\end{proof}

\begin{prop}\label{cot-2}
 Let $e:\Spec(R)\to \frakX_G^{\nil}$ be a classical point with $R$ being a finite type algebra over $\bZ/p^a$ for some $a$. Then the pro-cotangent complex $T^*_{e}(\frakX_G^{\nil})$ is a complex of $R$-modules, which is exactly the shift by $(-1)$ of the dual of the Herr complex associated to the adjoint of the \'etale $(\varphi,\Gamma)$-module with $G$-structure corresponding to $e$.
\end{prop}
\begin{proof}
    The proof is the same as that of \cite[Proposition 3.40]{Min23} once we know that for any point $e:\Spec(A)\to BG$, the cotangent complex $T^*_e(BG)$ is given by $ad(M_e)[-1]$ where $M_e$ is the corresponding $G$-torsor over $A$ and $ad(M_e)$ is the finite projective module over $A$ corresponding to the vector bundle $ M_e\times^G\frakg^\vee$ with $\frakg$ the Lie algebra of $G$ (cf. \cite[Section 2.4]{Khan24}).
\end{proof}

\subsection{Derived Laurent $F$-crystals and derived representations}
We have finished the first three steps of our strategy. Next we start dealing with \textbf{Step 4}. As part of it, we need to relate $\frakX_G^{\nil}$ to derived representations. This is achieved by first comparing derived Laurent $F$-crystals on the absolute prismatic site to those on the perfect absolute prismatic site. Let us begin with the precise definition of the latter.

\begin{dfn}
Let $(\calO_K)_{\Prism}^{\rm perf, tr}$ be the perfect transveral absolute prismatic site, whose objects are the transveral perfect prisms in $(\calO_K)_{\Prism}^{\rm perf}$ and topology is still the flat topology. For any animated ring $R\in \textbf{Nilp}^{<\infty}_{\bZ_p}$, define\[
    \frakX_G^{\rm perf}(R):=\varprojlim_{(A,I)\in (\calO_K)_{\Prism}^{\rm perf, tr}}BG(A\widehat\otimes^{\bL}R[\frac{1}{I}])^{\varphi=1},
    \]
    which is $\infty$-category of derived Laurent $F$-crystals with $G$-structure on the perfect transveral absolute prismatic site 
\end{dfn}
Then we have the following result.
\begin{prop}\label{perf-G}
      Let $R$ be an animated ring in $\textbf{Nilp}_{\bZ_p,\rm ft}^{\leq n}$ for some $n$. Then there is an equivalence 
      \[
      \frakX_G(R)\simeq \frakX_G^{\rm perf}(R).
      \]
\end{prop}

\begin{proof}
   Using the Breuil--Kisin prism $(\frakS,(E))$ and its perfection $(A_{\infty},(E))$, we have 
   \[
   \frakX_G(R)=\varprojlim_iBG(\frakS^i\widehat\otimes^{\bL}R[\frac{1}{E}])^{\varphi=1}
   \]
   and
   \[
   \frakX_G^{\rm perf}(R)\simeq\varprojlim_iBG(A_{\infty}^i\widehat\otimes^{\bL}R[\frac{1}{E}])^{\varphi=1}.
   \]
   It suffices to prove that
   \[
   BG(\frakS^i\widehat\otimes^{\bL}R[\frac{1}{E}])^{\varphi=1}\to BG(A_{\infty}^i\widehat\otimes^{\bL}R[\frac{1}{E}])^{\varphi=1}
   \]
   is fully faithful for all $i$ and is an equivalence when $i=0$.

   We argue by induction. When $R$ is discrete, we can use Tannakian duality. For each $n$, we have \[
   BG(\frakS^i\widehat\otimes R[\frac{1}{E}])^{\varphi=1}\simeq \Fun^{\rm ex,\otimes}(\Rep(G),\Vect(\frakS^i\widehat\otimes R[\frac{1}{E}])^{\varphi=1})\]
   and
   \[
   BG(A_{\infty}^i\widehat\otimes R[\frac{1}{E}])^{\varphi=1}\simeq \Fun^{\rm ex,\otimes}(\Rep(G),\Vect(A_{\infty}^i\widehat\otimes R[\frac{1}{E}])^{\varphi=1}).
   \]
By the proof of \cite[Proposition 2.23]{Min23}, we know that the functor $\Vect(\frakS^i\widehat\otimes R[\frac{1}{E}])^{\varphi=1}\to \Vect(A_{\infty}^i\widehat\otimes R[\frac{1}{E}])^{\varphi=1}$ is fully faithful for all $i$ and is an equivalence when $i=0$. So we are done when $R$ is discrete.

Now assume the result holds for all animated rings in $\textbf{Nilp}_{\bZ_p,\rm ft}^{\leq n}$. Let $R\in \textbf{Nilp}_{\bZ_p,\rm ft}^{\leq n+1}$ and $e\in \frakX_G(\tau_{\leq n}R)\simeq \frakX_G^{\rm perf}(R)$. Then by using the cotangent complex, it suffices to prove
\[
\varprojlim_j\Map(T^*_{e_i}(BG),\frakS^i\widehat\otimes\pi_{n+1}(R)[\frac{1}{E}][n+2])^{\varphi=1}\simeq \varprojlim_j\Map(T^*_{\tilde e_i}(BG),A_{\infty}^i\widehat\otimes\pi_{n+1}(R)[\frac{1}{E}][n+2])^{\varphi=1}
\]
where $e_i\in BG(\frakS^i\widehat\otimes\tau_{\leq n}R[\frac{1}{E}])^{\varphi=1}$ and $e_i\in BG(A_{\infty}^i\widehat\otimes\tau_{\leq n}R[\frac{1}{E}])^{\varphi=1}$ are the points corresponding to $e$. By the same argument as in the proof of \cite[Proposition 2.23]{Min23}, we can show \[((T^*_{e_i}(BG))^{\vee}\otimes(\frakS^i\widehat\otimes\pi_{n+1}(R)[\frac{1}{E}][n+2]))^{\varphi=1}\simeq ((T^*_{\tilde e_i}(BG))^\vee\otimes (A_{\infty}^i\widehat\otimes\pi_{n+1}(R)[\frac{1}{E}])[n+2]))^{\varphi=1}.\] So we are done.
\end{proof}

Now we could proceed further to relate Laurent $F$-crystals to derived representations. For readers' convenience, we recall the unframed derived local Galois deformation functor defined in \cite[Definition 5.4]{GV18}.

\begin{cons}[Unframed derived local Galois deformation functor]
    Let $\bar \rho:G_K\to G(k_f)$ be a residual representation. We can define a functor $\calF_{K,G}:{\textbf{Art}}\to \textbf{Ani}$ as
    \[
    \calF_{K,G}(R):=\Map_{\textbf{Ani}}(|G_K|,|G(R)|):=\varinjlim_i\Map_{\textbf{Ani}}(|G_i|,|G(R)|)
    \]
    where $G_K=\varprojlim_iG_i$ is the inverse limit of finite quotient groups and $|\cdot|$ means geometric realization of the corresponding simplicial anima, i.e. $|G_i|=\varinjlim_jG_i^j$ and $|G(R)|=\varinjlim_j\GL_d^j(R)=\varinjlim_j\Map(\calO(G)^{\otimes j},R))$ . Then the unframed derived Galois deformation functor is $\calF_{K,G,\bar\rho}:{\textbf{Art}_{/k_f}}\to \textbf{Ani}$ is defined as
    \[
   \calF_{K,G,\bar\rho}(R):={\rm Fib}_{\bar\rho}( \calF_{K,G}(R)\to  \calF_{K,G}(k_f))
    \]
\end{cons}

The next proposition shows that Laurent $F$-crystals are the same as derived representations when the coefficient rings are Artinian local animated rings.

\begin{prop}\label{derivedrep-laurent}
Asssume $G$ is a connected reductive group over $\bZ_p$.    Let $R$ be an Artinian local ring in $\textbf{Nilp}_{\bZ_p}^{\leq n}$ for some $n$, whose residue field is finite. There is an equivalence
    \[
    \frakX_G(R)\simeq \calF_{K,G}(R). 
    \]
\end{prop}

\begin{proof}
    We use the same strategy as the proof of \cite[Proposition 3.47]{Min23}. We still define an anima $BG(R,G_K)^{\simeq}$ as follows 
\[
BG(R,G_K)^{\simeq}:=\varprojlim BG(R)^{\simeq}\rightrightarrows BG(C(G_K,\bZ/p^a)\otimes R)^{\simeq}\rightthreearrow\cdots BG(C(G_K^{n+1},\bZ/p^a)\otimes R)^{\simeq}
\]
where the cosimplicial animated rings $C(G_K^{\bullet},\bZ/p^a)\otimes R$ is defined using the trivial $G_K$-action on $\bZ/p^a$.

We will then compare both $\frakX_G(R)\simeq \frakX_G^{\rm perf}(R)$ and $ \calF_{K,G}(R)$ with $BG(R,G_K)^{\simeq}$. But we need some new inputs.

    We claim that $\pi_0(BG(R))=\{*\}$, i.e. there is a unique $G$-torsor over $R$ up to equivalences. To prove this, recall that for any Artinian local ring with finite residue field in $\textbf{Nilp}_{\bZ_p}^{\leq n}$, there exists a sequence $R=R_n\to R_{n-1}\to\cdots R_1\to R_0=k_f$, where $k_f$ is a finite field, such that there is an equivalence from $R_i$ to the homotopy pullback of a diagram $A_{i-1}\to k_f\oplus k_f[m_i]\xleftarrow{}k_f$ for some $m_i\geq 1$.

Note that for any $m\geq 1$, we have $\pi_0(G(k_f\oplus k_f[m])=G(k_f)$, which follows from that $G(k_f\oplus k_f[m])\simeq \bigsqcup_{e\in G(k_f)}\Map(T^*_e(G),k_f[m])$ and $\pi_0(\Map(T^*_e(G),k_f[m]))=\{*\}$.  
    
    Now assume $\pi_0(BG(R_i))=*$. We have the pullback diagram by \cite[Chapter 1, Proposition 7.2.2]{GR17}
    \begin{equation*}
        \xymatrix{
        BG(R_{i+1})\ar[r]\ar[d]&BG(k_f)\ar[d]\\
        BG(R_{i})\ar[r]& BG(k_f\oplus k_f[m_{i+1}]).
        }
    \end{equation*}
    Then by $\pi_0(G(k_f\oplus k_f[m_{i+1}])=G(k_f)$ and $\pi_0(BG(R_i))=\{*\}$, we also have $\pi_0(BG(R_{i+1}))=\{*\}$. This finally shows $\pi_0(BG(R))=\{*\}$ as $\pi_0(BG(k_f))=\{*\}$ by Lang's theorem. Then we can use the same argument as in \cite[Proposition 3.47]{Min23} to show $BG(R,G_K)^{\simeq}\simeq \frakX_G^{\perf}(R)$.

    Next in order to use the argument in \cite[Proposition 3.47]{Min23} to prove $BG(R,G_K)^{\simeq}\simeq  \frakX_G^{\perf}(R)$, we need to show that for each $j\in[0,n+1]$, the base change functor \[BG(C(G_K^j,\bZ/p^a)\otimes R)\to BG(C(G_K^j,A_{\inf}[\frac{1}{\xi}]/p^a)\otimes R)^{\varphi=1}\] is fully faithful and is an equivalence when $j=0$. 
    
    Again we argue by induction. When $R$ is discrete, we could use the Tannaka duality. Then this follows from $\Vect(C(G_K^j,\bZ/p^a)\otimes R)\to \Vect(C(G_K^j,A_{\inf}[\frac{1}{\xi}]/p^a)\otimes R)^{\varphi=1}$ is fully faithful for all $j\in [0,n+1]$ and is an equivalence when $j=0$, which has been proved in the proof of \cite[Proposition 3.47]{Min23}.

    Now assume this holds true for all $n$-th truncated rings. Let $R\in \textbf{Nilp}_{\bZ_p}^{\leq n+1}$ and $e:\Spec(C(G_K^j,\bZ/p^a)\otimes\tau_{\leq n}R)\to BG$. As usual, by using cotangent complex, we are reduced to showing
    \[
    \Map(T^*_e(BG), C(G_K^j,\bZ/p^a)\otimes\pi_{n+1}(R))\simeq \Map(T^*_{\tilde e}(BG),C(G_K^j,A_{\inf}[\frac{1}{\xi}]/p^a)\otimes\pi_{n+1}(R))^{\varphi=1}
    \]
    where $\tilde e$ is the corresponding point in $BG(C(G_K^j,A_{\inf}[\frac{1}{\xi}]/p^a)\otimes R)^{\varphi=1}$.

    It suffices to show
    \[
      (T^*_e(BG))^\vee\otimes(C(G_K^j,\bZ/p^a)\otimes\pi_{n+1}(R))\simeq ((T^*_{\tilde e}(BG))^\vee\otimes (C(G_K^j,A_{\inf}[\frac{1}{\xi}]/p^a)\otimes\pi_{n+1}(R)))^{\varphi=1}
    \]
    which then follows from \cite[Sequence 3.8]{Min23}
    \[
    0\to C(G_K^j,\bZ/p^a)\to C(G_K^j,A_{\inf}[\frac{1}{\xi}]/p^a)\xrightarrow{\varphi-1}C(G_K^j,A_{\inf}[\frac{1}{\xi}]/p^a)\to 0.
    \]
\end{proof}

\subsection{The case of connected reductive groups}
Now we are going to finish the \textbf{Step 4}. Let $G$ be a connected reductive group. Fixing a finite-field point $\bar \rho\in \calX_{\rm G}(k_f)\simeq \calF_{K,G}(k_f)$, i.e. a residual representation of the absolute Galois group $G_K$, we need to investigate the functor $(\Lan\calX_{G})^{\#,\rm nil}_{\bar \rho}: {{\textbf{Art}}_{/k_f}}\to \textbf{Ani}$ defined as
\[
R\mapsto {\rm Fib}_{\bar \rho}((\Lan\calX_{G})^{\#,\rm nil}(R)\to (\Lan\calX_{G})^{\#,\rm nil}(k_f)),
\]
and the functor $\frakX_{G,\bar \rho}^{\nil}: {{\textbf{Art}}_{/k_f}}\to \textbf{Ani}$ defined as
\[
R\mapsto {\rm Fib}_{\bar \rho}(\frakX_G^{\nil}(R)\to \frakX_G^{\nil}(k_f)),
\]
where ${{\textbf{Art}}_{/k_f}}$ is the $\infty$-category of Artinian local animated ring $R$ with an identification of $k_f$ with the residue field of $R$. Recall that an animated ring $R$ is called Artinian local if $\pi_0(R)$ is a classical Artinian local ring and $\pi_*(R)$ is a finitely generated $\pi_0(R)$-module. In particular all Artinian local animated rings over $\bZ_p$ are truncated. For any $M\in {\rm Perf}(k_f)^{[n,0]}$, we have $k_f\oplus M$ is an Artinian local animated ring. So if $(\Lan\calX_{G})^{\#}_{\bar \rho}=(\Lan\calX_{G})^{\#,\rm nil}_{\bar \rho}\simeq \frakX^{\nil}_{G,\bar \rho}=\frakX_{G,\bar \rho}$, we can deduce $T^*_{\bar \rho}(\frakX_G^{\rm nil})\to T^*_{\bar \rho}((\Lan\calX_{G})^{\#,\rm nil})$ is an isomorphism.

We get an equivalence of functors $\frakX_{G,\bar \rho}\simeq \calF_{K,G,\bar\rho}: {\textbf{Art}}_{/k_f}\to \textbf{Ani}$ by Proposition \ref{derivedrep-laurent}. Then in order to prove $(\Lan\calX_{G})^{\#}_{\bar \rho}\simeq \frakX_{G,\bar \rho}$, it remains to prove the equivalence between $(\Lan\calX_{G})^{\#}_{\bar \rho}$ and $\calF_{K,G,\bar\rho}$. To this end, we also need to introduce the derived prestack of derived representations in \cite{Zhu20}.
\begin{dfn}[{\cite[Definition 2.4.3]{Zhu20}}]
       We define the derived prestack of framed $G$-valued continuous representations of $G_K$ over $\bZ_p$ as 
    \[
    \calR_{G_K,G}^{c}:{\textbf{Nilp}_{\bZ_p}}\to {\textbf{Ani}}, \ \ \ A\mapsto \varinjlim_r\Map_{{\textbf{Nilp}_{\bZ_p}^{\Delta}}}(\bZ/p^r[G^{\bullet}],C_{\rm cts}(G_K^{\bullet},A)).
    \]
    We can also define the derived prestack $ \calR_{G_K,G/G}^{c}$ of unframed $G$-valued continuous representations of $G_K$ over $\bZ_p$ as the geometric realization of 
    \begin{equation}\label{unframed}
        \xymatrix{
        \cdots\ar@<-.9ex>[r]\ar@<-.3ex>[r]\ar@<.3ex>[r]\ar@<.9ex>[r]&G\times G\times  \calR_{G_K,G}^{c}\ar@<-.6ex>[r]\ar@<.0ex>[r]\ar@<.6ex>[r]&G\times  \calR_{G_K,G}^{c}\ar@<-.3ex>[r]\ar@<.3ex>[r]&\calR_{G_K,G}^{c}
        }
    \end{equation}
    in the $\infty$-category $\PreStk$.
\end{dfn}

As in \cite{Min23}, the last crucial input is the following theorem.

\begin{dfn}[{\cite[Definition 2.5]{PQ24}}]
    A generalised reductive group scheme over a scheme $S$ is a smooth
aﬃne $S$-group scheme $G$, such that the geometric fibres of $G^0$ are reductive and $G/G^0\to S$ is finite.
\end{dfn}

\begin{thm}[{\cite[Theorem 1.1]{PQ24}}]\label{PQ24}
   Let $G$ be a generalised reductive group. Let $k_f$ be a finite field of characteristic $p$. Let $\bar\rho:G_K\to G(k_f)$ be a residual representation. Then the framed local Galois deformation ring $R^{\square}_{\bar \rho}$ is a local complete intersection.
\end{thm}
Indeed, Theorem \ref{PQ24} implies that the derived functor $\calF_{K,G,\bar\rho}$ is classical by \cite[Lemma 7.5]{GV18}. This is the infinitesimal version of the global classicality theorem we aim to prove, which is also one of the key ingredients in the proof of the global classicality theorem.

Now we prove the classicality result for connected reductive groups.

\begin{thm}\label{main-connected}
Let $G$ be a connected reductive group over $\bZ_p$. Let $k_f$ be a finite field and  $\bar\rho:G_K\to G(k_f)$ be a residual representation. Then there is an equivalence of functors
    \[
    (\Lan\calX_{G})^{\#}_{\bar \rho}\simeq\calF_{K,G,\bar\rho}.
    \]
    As a consequence, we have $(\Lan\calX_{G})^{\#,\nil}\simeq \frakX_G^{\nil}$.
\end{thm}
\begin{proof}
    To prove $(\Lan\calX_{G})^{\#}_{\bar \rho}\simeq\calF_{K,G,\bar\rho}$, the arguments in \cite[Proposition 3.51]{Min23} still work here. We briefly explain the strategy. The proof can be divided into three steps:
\begin{enumerate}
    \item prove $(\Lan\calX_{G})_{\bar \rho}\simeq (\Lan ^{\rm cl}\calR_{G_K,G/G}^{c})_{\bar \rho}$;
    \item prove $(\Lan^{\rm cl}\calR_{G_K,G/G}^{c})_{\bar \rho}\simeq\calR_{G_K,G/G,\bar\rho}^{c}$;
    \item prove $\calR_{G_K,G/G,\bar\rho}^{c}\simeq \calF_{K,G,\bar\rho}$.
\end{enumerate}
For Step 1, we need to use that both $\calX_G$ and $^{\rm cl}\calR_{G_K,G/G}^{c}$ are limit preserving, as well as the equivalence between the category of \'etale $(\varphi,\Gamma)$-modules with $G$-structure and the category of Galois representations valued in $G$ when the coefficient rings are finite Artinian local rings (cf. Proposition \ref{derivedrep-laurent}). For Step 2, we need Theorem \ref{PQ24} or more precisely the infinitesimal classicality result guaranteed by Theorem \ref{PQ24}. The Step 3 is just about the definitions of both sides.

To conclude $(\Lan\calX_{G})^{\#,\nil}\simeq \frakX_G^{\nil}$, we use Proposition \ref{GR-criterion} and the reduction process to the case of finite-field points we have already finished.
\end{proof}

\subsection{The case of generalised reductive groups}
Now we proceed to discuss the case of generalised reductive groups. Indeed, this is crucial in the setting of Langlands parameters: one important example of generalised reductive groups is the $L$-group. More precisely, let $G$ be a (connected) reductive group over $K$. Write $\widehat G$ for the associated Langlands dual group over $\bZ_p$, which is a split (connected) reductive group. And there exists a finite Galois extension $E/K$ such that $\Gal(E/K)$ acts on $\hat G$. The semi-direct product $^LG:=\widehat G\rtimes \Gal(E/K)$ is called the $L$-group of $G$.

For a generalised reductive group $G$, we need to modify the definition of derived representations. Let us begin with the unframed local Galois deformation functor.
\begin{dfn}[Generalised unframed local deformation functor]
 Let $\bar \rho:G_K\to G(k_f)$ be a residual representation. We can define a functor $\calF_{K,G}^{\gen}:{\textbf{Art}}\to \textbf{Ani}$ by the following pullback diagram
 \begin{equation}
     \xymatrix{
     \calF_{K,G}^{\gen}(R)\ar[r]\ar[d]& |G^\circ(R)|\ar[d]\\
     \calF_{K,G}(R)\ar[r]& |G(R)|.
     }\end{equation}
   Then the unframed derived Galois deformation functor is $\calF_{K,G,\bar\rho}^{\gen}:{\textbf{Art}_{/k_f}}\to \textbf{Ani}$ is defined as
    \[
   \calF_{K,G,\bar\rho}^{\gen}(R):={\rm Fib}_{\bar\rho}( \calF_{K,G}^{\gen}(R)\to  \calF_{K,G}^{\gen}(k_f))
    \] 
\end{dfn}

For general coefficients, we consider the derived prestack $ \calR_{G_K,G/G^{\circ}}^{c}$ of unframed $G$-valued continuous representations of $G_K$ over $\bZ_p$ as the geometric realization of 
    \begin{equation}\label{unframed}
        \xymatrix{
        \cdots\ar@<-.9ex>[r]\ar@<-.3ex>[r]\ar@<.3ex>[r]\ar@<.9ex>[r]&G^{\circ}\times G^{\circ}\times  \calR_{G_K,G}^{c}\ar@<-.6ex>[r]\ar@<.0ex>[r]\ar@<.6ex>[r]&G^{\circ}\times  \calR_{G_K,G}^{c}\ar@<-.3ex>[r]\ar@<.3ex>[r]&\calR_{G_K,G}^{c}
        }
    \end{equation}
    in the $\infty$-category $\PreStk$, i.e. we only consider the $G^{\circ}$-conjugation now.

Accordingly, we give a definition of generalised derived stack of Laurent $F$-crystals.

\begin{dfn}\label{dfn-gen-derived}
    For a generalised reductive group $G$, we define $\frakX_G^{\gen}$ by the following pullback diagram
    \begin{equation}
        \xymatrix{
        \frakX_G^{\gen}\ar[r]\ar[d]& \Spf(\bZ_p)\ar[d]\\
        \frakX_G\ar[r]& \frakX_{\bar G}.
        }
    \end{equation}
    We write $\bar G:=G/G^{\circ}$ for simplicity and the right vertical map $\Spf(\bZ_p)\to \calX_{\bar G}$ means the trivial derived Laurent $F$-crystals with $\bar G$-structure.
\end{dfn}

As both $\Spf(\bZ_p)$ and $\frakX_{\bar G}$ admit a deformation theory and are locally almost of finite type, we see $\frakX_G^{\gen}$ also admits a deformation theory and is locally almost of finite type by \cite[Proposition 3.23]{Min23}. Moreover, as $\bar G$ is a finite \'etale group scheme, it has trivial Lie algebra. This means $T^*_{*_{\rm triv}}(\frakX_{\bar G})=0$ by Proposition \ref{cot-2} and its proof. In particular, the relative cotangent complex $T^*(\Spf(\bZ_p)/\frakX_{\bar G}):={\rm coFib}(T^*_{*_{\rm triv}}(\frakX_{\bar G})\to T^*(\bZ_p)=0)$ is $0$ . Then the relative cotangent complex $T^*(\frakX_G^{\gen}/\frakX_G)=T^*(\Spf(\bZ_p)/\frakX_{\bar G})\otimes_{\calO_{\frakX_{\bar G}}}\calO_{\frakX_G^{\gen}}=0$. This implies that Proposition \ref{cot-2} holds true if $\frakX_G$ is replaced by $\frakX_{G}^{\gen}$.

Similarly, we can define a generalised Emerton--Gee stack.

\begin{dfn}\label{dfn-gen-EG}
      For a generalised reductive group $G$, we define $\frakX_G^{\gen}$ by the following pullback diagram
    \begin{equation}
        \xymatrix{
        \calX_G^{\gen}\ar[r]\ar[d]& \Spf(\bZ_p)\ar[d]\\
        \calX_G\ar[r]& \calX_{\bar G}
        }
    \end{equation}
    where the right vertical map $\Spf(\bZ_p)\to \calX_{\bar G}$ means the trivial \'etale $(\varphi,\Gamma)$-module with $\bar G$-structure.
\end{dfn}

As the diagonal of $\calX_{\bar G}$ is representable in algebraic spaces and of finite presentation, we get that $\Spf(\bZ_p)\to \calX_{\bar G}$ is representable in algebraic spaces and of finite presentation, which implies $\calX_G^{\gen}\to \calX_G$ is representable in algebraic spaces  and of finite presentation. So $\calX_G^{\gen}$ is still a formal algebraic stack locally of finite presentation over $\Spf(\bZ_p)$.

Now we prove Proposition \ref{derivedrep-laurent} for these generalised definitions.

\begin{prop}\label{generalised rep-phiGamma}
      Let $R$ be an Artinian local ring in $\textbf{Nilp}_{\bZ_p}^{\leq n}$ for some $n$, whose residue field is finite. There is an equivalence
    \[
    \frakX_G^{\gen}(R)\simeq \calF_{K,G}^{\gen}(R). 
    \]
\end{prop}

\begin{proof}
 On the one hand, we still define an anima $BG(R,G_K)^{\simeq}$ as follows 
\[
BG(R,G_K):=\varprojlim BG(R)\rightrightarrows BG(C(G_K,\bZ/p^a)\otimes R)\rightthreearrow\cdots BG(C(G_K^{n+1},\bZ/p^a)\otimes R)
\]
where the cosimplicial animated rings $C(G_K^{\bullet},\bZ/p^a)\otimes R$ is defined using the trivial $G_K$-action on $\bZ/p^a$.

On the other hand, we can define an anima $|G(R,G_K)|$ as follows
\[
|G(R,G_K)|:=\varprojlim |G(R)|\rightrightarrows |G(C(G_K,\bZ/p^a)\otimes R)|\rightthreearrow\cdots |G(C(G_K^{n+1},\bZ/p^a)\otimes R)|.
\]

By the proof of \cite[Proposition 3.47]{Min23}, we can see that $\calF_{K,G}(R)\simeq |G(R,G_K)|$ and $BG(R,G_K)\simeq \frakX_G^{\perf}(R)$, $B\bar G(R,G_K)\simeq \frakX_{\bar G}^{\perf}(R)$. By Proposition \ref{perf-G}, we then get $BG(R,G_K)\simeq \frakX_G(R)$ and $B\bar G(R,G_K)\simeq \frakX_{\bar G}(R)$.

Using $\calF_{K,G}(R)\simeq |G(R,G_K)|$, we see that $\calF_{K,G}^{\gen}(R)\simeq |G(R,G_K)|\times_{|G(R)|}|G^{\circ}(R)|$. Since $G^{\circ}\to G$ is an open immersion, the natural morphism $G^{\circ}(C(G_K^{i},\bZ/p^a)\otimes R)\to |G^{\circ}(C(G_K^{i},\bZ/p^a)\otimes R)|\times_{|G(R)|}|G^{\circ}(R)|$ is an equivalence. So we get $\calF_{K,G}^{\gen}(R)\simeq |G^{\circ}(R,G_K)|$, where $|G^{\circ}(R,G_K)|$ is defined similarly as $|G(R,G_K)|$.

Next we are going to show $\frakX_G^{\gen}(R)$ is also equivalent to $|G^{\circ}(R,G_K)|$. Note that now we have $\frakX_G^{\gen}(R)\simeq BG(R,G_K)\times_{B\bar G(R,G_K)}*_{\rm triv}$. There is a natural morphism from $|G^{\circ}(R,G_K)|$ to $BG(R,G_K)\times_{B\bar G(R,G_K)}*_{\rm triv}$. To prove it is an equivalence, it suffices to prove
\begin{enumerate}
    \item[(a)] The map  $|G^{\circ}(C(G_K^{i},\bZ/p^a)\otimes R)|\to BG(C(G_K^{i},\bZ/p^a)\otimes R)\times_{B\bar G(C(G_K^{i},\bZ/p^a)\otimes R)}*_{\rm triv}$ is fully faithful for all $i$.
    \item[(b)] The map $|G^{\circ}(R)|\to BG(R)\times_{B\bar G(R)}*_{\rm triv}$ is an equivalence.
\end{enumerate}

Let us first look at (a). In fact, it is sufficient to show for any animated $\bZ/p^a$-algebra $S$, the map $G^{\circ}(S)\to G(S)\times_{\bar G(S)}*$ is an equivalence. This follows from the fact that $G\to \bar G$ is (faithfully) flat, which implies $G^{\circ}$ is isomorphic to the homotopy pullback of the diagram $G\to \bar G\leftarrow \Spec(\bZ_p)$.

Now we proceed to prove (b). Recall that as $G^\circ$ is a connected reductive group, we have $\pi_0(BG^\circ(R))=\{*\}$, i.e. the trivial torsor is the only $G^{\circ}$-torsor over $R$ up to isomorphisms. Then the natural map $|G^{\circ}(R)|\to BG^\circ(R)$ is an equivalence as the mapping spaces $\Map(*,*)$ in both of them are equivalent to $G^{\circ}(R)$ (cf. \cite[Lemma 5.2]{GV18}).

We also know that there exists a sequence $R=R_n\to R_{n-1}\to\cdots R_1\to R_0=k_f$ such that there is an equivalence from $R_i$ to the homotopy pullback of a diagram $R_{i-1}\to k_f\oplus k_f[m_i]\xleftarrow{}k_f$ for some $m_i\geq 1$. Associated to the short exact sequence $0\to G^{\circ}\to G\to \bar G\to 0$, there is a long exact sequence of pointed sets
\[
0\to G^0(k_f)\to G(k_f)\to \bar G(k_f)\to H^1(k_f,G^{\circ})\to H^1(k_f,G)\to H^1(k_f,\bar G).
\]
As $BG^{\circ}(k_f)=\{*\}$, we get a short exact sequence $0\to G^0(k_f)\to G(k_f)\to \bar G(k_f)\to 0$ and an injection $H^1(k_f,G)\hookrightarrow H^1(k_f,\bar G)$. This gives us a homotopy pullback diagram

\begin{equation*}
    \xymatrix{
    BG^\circ(k_f)\ar[r]\ar[d]& {*}_{\rm triv}\ar[d]\\
    BG(k_f)\ar[r]& B\bar G(k_f).
    }
\end{equation*}

Since $\bar G$ is a finite \'etale group scheme over $\bZ_p$, it has trivial cotangent complex. This implies $T^*_e(BG^{\circ})\simeq T^*_e(BG)$ for any point $e:\Spec(A)\to BG^{\circ}\to BG$. By the definition of cotangent complex, we have homotopy pullback diagrams
\begin{equation}
    \xymatrix{
    \Map(T^*(BG),k_f[m_i])\ar[r]\ar[d]& {*}_{\rm triv}\ar[d]\\
    BG(k_f\oplus k_f[m_i])\times_{BG(k_f)}BG^{\circ}(k_f)\ar[r]\ar[d]& BG^{\circ}(k_f)\ar[d]\\
    BG(k_f\oplus k_f[m_i])\ar[r]&BG(k_f).
    }
\end{equation}
In particular, we have \[BG(k_f\oplus k_f[m_i])\times_{BG(k_f)}BG^{\circ}(k_f)\simeq [\Map(T^*(BG),k_f[m_i])/G^{\circ}(k_f)].\] 
As $\Map(T^*(BG),k_f[m_i])\simeq \Map(T^*(B^{\circ}),k_f[m_i])$ and we know $[\Map(T^*(BG),k_f[m_i])/G^{\circ}(k_f)]\simeq BG^{\circ}(k_f\oplus k_f[m_i])$, we then have $BG^{\circ}(k_f\oplus k_f[m_i])\simeq BG(k_f\oplus k_f[m_i])\times_{BG(k_f)}BG^{\circ}(k_f)$. In other words, we have pullback diagrams
\begin{equation}
    \xymatrix{
    BG^{\circ}(k_f\oplus k_f[m_i])\ar[r]\ar[d]& BG^{\circ}(k_f)\ar[r]\ar[d]& {*}_{\rm triv}\ar[d]\\
    BG^(k_f\oplus k_f[m_i])\ar[r]& BG(k_f)\ar[r]& B\bar G(k_f)\simeq B\bar G(k_f\oplus k_f[m_i])
    }
\end{equation}
where $B\bar G(k_f)\simeq B\bar G(k_f\oplus k_f[m_i])$ is due to the triviality of the cotangent complex $T^*(B\bar G)=0$.

Now as $BG^{\circ}, BG, B\bar G$ are all infinitesimally cohesive, i.e. preserve square-zero extensions, we finish the proof of (b). So we are done.

\end{proof}

Once again, we have all the necessary ingredients and can use the full strength of \cite[Theorem 1.1]{PQ24} to prove the classicality of derived Emerton--Gee stack for generalised reductive groups.

\begin{thm}\label{main-gen}
    Let $G$ be a generalised reductive group over $\bZ_p$. Let $k_f$ be a finite field and $\bar\rho:G_K\to G(k_f)$ be a residual representation. Then there is an equivalence of functors
    \[
    (\Lan\calX^{\gen}_{G})^{\#}_{\bar \rho}\simeq\calF^{\gen}_{K,G,\bar\rho}.
    \]
    As a consequence, we have $(\Lan\calX^{\gen}_{G})^{\#,\nil}\simeq \frakX_G^{\gen,\nil}$.
\end{thm}

\begin{proof}
We can use the same arguments in the proof of Theorem \ref{main-connected} or \cite[Proposition 3.51]{Min23}, i.e. proceed in three steps:
\begin{enumerate}
    \item prove $(\Lan\calX^{\gen}_{G})_{\bar \rho}\simeq (\Lan ^{\rm cl}\calR_{G_K,G/G^{\circ}}^{c})_{\bar \rho}$;
    \item prove $(\Lan^{\rm cl}\calR_{G_K,G/G^{\circ}}^{c})_{\bar \rho}\simeq\calR_{G_K,G/G^{\circ},\bar\rho}^{c}$;
    \item prove $\calR_{G_K,G/G^{\circ},\bar\rho}^{c}\simeq \calF^{\gen}_{K,G,\bar\rho}$.
\end{enumerate}
For Step 1, we need to use that both $\calX^{\gen}_G$ and $^{\rm cl}\calR_{G_K,G/G^{\circ}}^{c}$ are limit preserving, as well as the equivalence between the category of generalised \'etale $(\varphi,\Gamma)$-modules with $G$-structure and the category of generalised Galois representations valued in $G$ when the coefficient rings are finite Artinian local rings (cf. Proposition \ref{generalised rep-phiGamma}). For Step 2, we use the full strangth of Theorem \ref{PQ24}. Again, the Step 3 is just about the definitions of both sides.

To conclude $(\Lan\calX^{\gen}_{G})^{\#,\nil}\simeq \frakX_G^{\gen,\nil}$, we use Proposition \ref{GR-criterion} and the reduction process to the case of finite-field points.
\end{proof}

\addcontentsline{toc}{section}{References}

\bibliographystyle{alpha}
\bibliography{sample}

\end{document}